\documentclass[11pt]{amsart}
\usepackage[alphabetic]{amsrefs} % disable this in "article"
\usepackage[letterpaper,margin=1in]{geometry}

\usepackage{amsmath}
\usepackage{amssymb,mathrsfs,url}
\usepackage{amsfonts}
\usepackage{amsthm}
\usepackage[utf8]{inputenc}
\usepackage[pdfencoding=auto, psdextra,colorlinks=true,linkcolor=blue, citecolor=teal!50!green,urlcolor=cyan]{hyperref}
\usepackage{verbatim}

\usepackage{tikz-cd}
\usepackage{array,booktabs}
\usepackage{xspace}
\usepackage[all,cmtip]{xy}
\usepackage{dynkin-diagrams}
\usepackage{enumerate}
\usepackage{bm}
\usepackage{enumitem}
\usepackage{xcolor}
\usepackage{xr}
\usepackage{multicol}

\newtheorem{lem}{Lemma}[section]
\newtheorem{definition}[lem]{Definition}
\newtheorem{thm}[lem]{Theorem}
\newtheorem{prop}[lem]{Proposition}
\newtheorem{cor}[lem]{Corollary}
\newtheorem{convention}[lem]{Convention}
\newtheorem{construction}[lem]{Construction}

\newtheorem{Definition and Proposition}[lem]{Definition/Proposition}
\numberwithin{equation}{section}

\newtheorem*{propS}{Proposition}

\newtheorem*{thmA}{Theorem A}
\newtheorem*{thmB}{Theorem B}
\newtheorem*{thmC}{Theorem C}
\theoremstyle{remark}
\newtheorem{rk}[lem]{Remark}
\newtheorem*{rkS}{Remark}

\newenvironment{proofof}[1][]{{\noindent\it Proof of  {#1}.}\hspace{.5em}}{\hfill $\square$\par}

\newcommand{\bb}[1]{\mathbb{#1}}

\newcommand{\ca}[1]{\mathcal{#1}}

\newcommand{\und}[1]{\und{#1}}
\newcommand{\sbst}{\subset}

\newcommand{\lra}{\longrightarrow}

\newcommand{\mbf}[1]{\mathbf{#1}}

\newcommand{\iso}{\cong}
\newcommand{\mrm}[1]{\mathrm{#1}}
\newcommand{\wat}{\widehat}

\newcommand{\dr}{\mathrm{dR}}

\newcommand{\der}{\mrm{der}}
\newcommand{\ad}{\mrm{ad}}

\newcommand{\sh}{\mrm{Sh}}
\newcommand{\wdtd}{\widetilde}

\newcommand{\A}{\bb{A}_f}
\newcommand{\bss}{\backslash}
\newcommand{\disju}{\coprod}
\newcommand{\stb}{\mrm{Stab}}

\newcommand{\zbkp}{{\bb{Z}_{(p)}}}

\DeclareMathOperator{\lie}{\mrm{Lie}}

\newcommand{\mmin}{\mrm{min}}
\newcommand{\cusp}{\mrm{Cusp}}

\newcommand{\Ap}{\bb{A}_f^p}

\DeclareMathOperator{\spec}{\mrm{Spec}}
\newcommand{\gal}{\mrm{Gal}}
\newcommand{\aut}{\underline{\mrm{Aut}}}

\newcommand{\Hom}{\mrm{Hom}}

\newcommand{\G}{\ca{G}}

\newcommand{\lcj}[2]{{}^{#2}{#1}}

\newcommand{\Q}{\ca{Q}}

\DeclareMathOperator{\im}{\mrm{Im}}

\newcommand{\ul}[1]{\underline{#1}}

\usepackage[OT2,T1]{fontenc}
\DeclareSymbolFont{cyrletters}{OT2}{wncyr}{m}{n}
\DeclareMathSymbol{\Sha}{\mathalpha}{cyrletters}{"58}

\newcommand{\can}{\mrm{can}}

\newcommand{\red}{\mrm{red}}

\newcommand{\cpl}[2]{{(#1)}{}^{\wedge{\ }}_{#2}}

\newcommand{\ag}{\mathscr{A}}
\newcommand{\agsb}{\mathscr{A}^\circ}

\newcommand{\K}{\wdtd{K}}

\newcommand{\vbd}{\mrm{Vec}}
\newcommand{\rep}{\mrm{Rep}}
\newcommand{\fil}{\mrm{Fil}}
\newcommand{\lsys}{\mrm{Loc}}
\newcommand{\autoshbetti}[1]{{}_{B}\underline{#1}}

\newcommand{\autoshdr}[1]{{}_{\dr}\underline{#1}}

\title[Canonical extensions of principal bundles]{Principal bundles on toroidal compactifications of integral canonical models of abelian-type Shimura varieties}
\author{Peihang Wu}
\address{BICMR, Peking University, Beijing 100871, China}
\email{wuph@pku.edu.cn} 
\date{Jan. 11th, 2026}
\begin{document}
\begin{abstract}
In this paper, we construct canonical extensions of principal $\G^c$- (and $M^c$-)bundles on toroidal compactifications of integral canonical models of abelian-type Shimura varieties with hyperspecial levels.
\end{abstract}
\maketitle
\tableofcontents
\section{Introduction}\label{sec-intro}
The purpose of this paper is to construct canonical and subcanonical extensions of integral models of principal bundles, as well as of automorphic vector bundles, over toroidal compactifications of integral canonical models of abelian-type Shimura varieties with hyperspecial level structure at a fixed prime number $p$.\par
Principal bundles and automorphic vector bundles are natural objects on Shimura varieties whose cohomology theory is related to representation theory and automorphic forms. If the underlying Shimura varieties are not proper, one needs to consider the \emph{canonical extensions} and \emph{subcanonical extensions} of these objects to toroidal compactifications.\par
Let $(G,X)$ be a Shimura datum. For each $K\sbst G(\A)$ neat open compact, the double coset $\sh_K(G,X)(\bb{C}):=G(\bb{Q})\bss X\times G(\A)/K$ has a unique canonical model $\sh_K:=\sh_K(G,X)$ over the reflex field $E:=E(G,X)$. Let $G^c$ be the quotient of $G$ by the anti-cuspidal part of the connected center, which can be assumed to be identical to $G$ for simplicity in this introduction. There is a \emph{standard principal $G^c_\bb{C}$-bundle} 
$$\ca{E}_\bb{C}:=G(\bb{Q})\bss (X\times G^c_\bb{C}\times G(\A)/K)$$ over $\sh_K(G,X)(\bb{C})$ with a $G^c_\bb{C}$-equivariant map $\ca{E}_\bb{C}\to \breve{X}$, such that it fits into the diagram
\begin{equation*}
\begin{tikzcd}
    &\ca{E}_\bb{C}\arrow[dl]\arrow[dr]&\\
    \sh_K(G,X)(\bb{C})&&\breve{X}.
    \end{tikzcd}
\end{equation*}
where $\breve{X}$ is the Borel embedding of $X$.
By series of works of Harris and Milne (see, e.g., \cites{Har85,Har86} and \cites{Mil88,Mil90}), this analytic $G^c_\bb{C}$-bundle algebraizes and canonically descends to an algebraic $G^c$-bundle $\ca{E}$ on the canonical model of the Shimura variety $\sh_K$. The construction of such bundles $\ca{E}$ plays an important role in studying the arithmetic properties of certain automorphic forms (see \cite{Har90survey}).\par
In order to correctly define and study log de Rham cohomology with coefficients in automorphic sheaves on integral models, or on the special fibers, of noncompact Shimura varieties, it is necessary to construct canonical extensions of certain principal $\G^c$-bundles (where $\G^c$ is a certain reductive model of $G^c$) to toroidal compactifications to produce these coefficient sheaves (see \cite{LS13} and \cite{Lan16iccm}). One can also construct coefficients for coherent cohomology on toroidal compactifications of integral models or special fibers using canonical extensions of certain $M^c$-bundles, where $M^c$ is a Levi component of $\G^c$ defined by a Hodge cocharacter. In contrast to the situation over $\bb{C}$, where Deligne’s existence theorem \cite{Del70} can be applied directly, the corresponding construction over integral models typically depends in an essential way on the explicit construction of the integral models themselves. Since we have already had integral canonical models of principal bundles and good toroidal compactifications for integral canonical models of abelian-type Shimura varieties by \cite{Lov17} and \cite{Wu25}, it can be expected that such canonical extensions should exist on these compactifications.\par 
In the PEL-type or, more generally, Hodge-type cases, these canonical extensions are constructed by Lan \cite{Lan12} and Madapusi \cite[Sec. 4.3]{Mad19}, respectively. The first work constructed canonical extensions by toroidal compactifications of Kuga families, while the second one achieved this by gluing vector bundles on a cover of the toroidal compactification. Although the two methods appear to be different, they both constructed the unique vector bundle that simultaneously extends a certain vector bundle on the union of the integral canonical model $\ca{S}_K$ and the toroidal compactification of $\sh_K$. In this paper, we will use the construction in \cite{Wu25} to generalize their constructions to abelian-type cases.\par
Let us fix some conventions to state the main result. 
Let $(G_2,X_2)$ be an abelian-type Shimura datum. Fix a prime $p>0$ and a place $v_2|p$ of $E_2:=E(G_2,X_2)$. Let $\ca{O}_2:=\ca{O}_{E_2,(v_2)}$. Assume that $G_{2,\bb{Q}_p}$ is (quasi-split and) unramified. Choose a hyperspecial subgroup $K_{2,p}\sbst G_2(\bb{Q}_p)$ and a neat open compact subgroup $K_2^p\sbst G_2(\Ap)$. There is a smooth reductive model $\G_2$ of $G_2$ over $\zbkp$ such that $\G_2(\bb{Z}_p)=K_{2,p}$. Denote $K_2:=K_{2,p}K^p_2$. 
Let $\{\ca{S}_{K_{2,p}K_2^p}(G_2,X_2)\}_{K_2^p}$ be the inverse system of smooth integral models constructed in \cite{Kis10} and \cite{KM15}. By \cite[Thm. 4.39]{Wu25} Case (HS), the integral model $\ca{S}_{K_2}:=\ca{S}_{K_{2,p}K_2^p}(G_2,X_2)$ admits smooth and projective toroidal compactifications. We choose an admissible projective and smooth cone decomposition $\Sigma_2$ such that the toroidal compactification $\ca{S}_{K_2}^{\Sigma_2}$ determined by $\Sigma_2$ is constructed in \cite{Wu25}.
The main theorem of this paper is stated as follows.
\begin{thmA}[{See Theorem \ref{thm-extend-main-theorem}}]
The principal $\G^c_2$-bundle $\mathscr{E}$ on $\ca{S}_{K_2}$ defined in \cite{Lov17} extends uniquely to a principal $\G^c_2$-bundle $\mathscr{E}^\can$ on $\ca{S}_{K_2}^{\Sigma_2}$, such that the restriction of $\mathscr{E}^\can$ to $\sh_{K_2}^{\Sigma_2}$ coincides with the canonical extension of the $G^c_2$-bundle $\ca{E}:=\mathscr{E}|_{\sh_{K_2}}$ to $\sh_{K_2}^{\Sigma_2}$.\par 
Let $W$ be an algebraic representation of $\G^c_2$. Then $\autoshdr{W}_{K_2}^{\can}:=\mathscr{E}^\can\times^{\G^c_2} W$ is a vector bundle equipped with an integrable log connection extending $\autoshdr{W}_{K_2}:=\mathscr{E}\times^{\G^c_2}W$.\par
The bundle $\mathscr{E}^\can$ determines and is determined by an exact tensor functor
$$\omega^\can: \rep(\G^c_2)\lra \mrm{Vec}_{\ca{S}^{\Sigma_2}_{K_2}}$$
from the category of $\G^c_2$-representations to the category of vector bundles with integrable log connections on $\ca{S}^{\Sigma_2}_{K_2}$.
\end{thmA}
The vector bundles $\autoshdr{W}_{K_2}$ are called automorphic sheaves and $\autoshdr{W}^\can_{K_2}$ are called their canonical extensions. Moreover,
\begin{thmB}[See Corollary \ref{cor-gr-extend}]
There is a correspondence 
\begin{equation*}
    \begin{tikzcd}
    &\mathscr{E}^\can\arrow[ld,"\pi^\can_1"']\arrow[dr,"\pi_2^\can"]&\\
    \ca{S}^{\Sigma_2}_{K_2}&& \ca{GR}_{\mu_2^c}
    \end{tikzcd}
    \end{equation*}
where the right arrow is $\G_{2,\ca{O}_2}^c$-equivariant. Equivalently, there is a $\mu_2^c$-filtration on $\mathscr{E}^\can$ in the sense of \cite[Lem. 3.3.1]{Lov17}.
\end{thmB}
Consequently, we can also define canonical extensions of the principal bundles under the parabolic subgroup or the Levi subgroup of $\G_2^c$ determined by a Hodge cocharacter over a certain ring extension. 
%\begin{rkS}
%In a forthcoming paper by the author, we will show some vanishing theorems for various cohomologies with coefficients in automorphic sheaves.
%\end{rkS}
\subsubsection*{Idea of proof}
We essentially reduce the problem to answering the following question:
\begin{itemize}
    \item \textbf{Q:} Let $X$ be a proper smooth scheme over a base ring $\ca{O}$ (which is, for example, a DVR) with a normal crossings boundary divisor $D$. Let $\pi:Y\to X$ be a finite Kummer {\'e}tale cover of $X$ that is the quotient map of $Y$ by a finite group $\Delta$ acting on it. Moreover, assume that the action of $\Delta$ on $Y^\circ:=Y\times_X (X\bss D)$ is free. (Note that $Y$ might not be smooth while it is so on $Y^\circ$, and that a Kummer {\'e}tale map might not be flat!)\par 
    Given a vector bundle $\ca{V}$ on $Y$ with an integrable connection with log poles along $Y\bss Y^\circ$, and an equivariant $\Delta$-action, when does $\ca{W}:=(\pi_*\ca{O}_\ca{V})^\Delta$ again correspond to a vector bundle (with an integrable log connection)?
\end{itemize}
In our situation, $X$ is a toroidal compactification $\ca{S}_K^\Sigma(G,X_b)$ associated with a smooth and projective cone decomposition $\Sigma$, and $Y$ is a cover of $\ca{S}_K^\Sigma(G,X_b)$ by a disjoint union of Hodge-type toroidal compactifications, whose cone decompositions are chosen as those induced by $\Sigma$. The sheaf $\ca{W}$ is a locally free sheaf associated with the canonical extensions of certain automorphic vector bundles on the disjoint union. In this situation, reorganizing the results in the literature shows that $\ca{W}$ has the desired property on a locus of codimension $\geq 2$. Hence, we only have to show the \textbf{local-freeness} of $\ca{W}$.\par
Roughly speaking,
\begin{thmC}[{Theorem \ref{thm-loc-free-quo}}]
In our situation, which will be clear in the paper, the torsion-free coherent sheaf $\ca{W}$ is always locally free. 
\end{thmC}
The proof of the above theorem relies on an understanding of the canonical extensions at the boundary in the Hodge-type case.
Let $(G_0,X_0)$ be a Hodge-type Shimura datum. 
We need to understand the extension of the Hodge tensors to the boundary. Based on the works of Kisin (see \cite[Cor. 2.3.9]{Kis10}) and Madapusi (see \cite[Prop. 4.3.7]{Mad19}), the de Rham tensors $s_{\lambda,\dr}\in (\ca{V}_{\zbkp}\otimes \bb{Q})^\otimes$ uniquely extend to $\wdtd{s}_{\lambda,\dr}^\can\in (\ca{V}^\can_\zbkp)^\otimes$.
The canonical extension of the principal $\G_0$-bundle $\mathscr{E}_{\G_0,K_0}$ to the toroidal compactification $\ca{S}_{K_0}^{\Sigma_0}$ is given by 
$$\ul{\mrm{Isom}}_{s_{\lambda}\mapsto\wdtd{s}_{\dr,\lambda}}(V_{\zbkp},\ca{V}_{\zbkp}^\can),$$ where $V_\zbkp$ is a lattice of a $(V,\psi)$ defining the Siegel Shimura datum that $(G_0,X_0)$ embeds, and $\ca{V}_{\zbkp}^\can$ is the canonical extension of the 1st relative de Rham homology of the pullback of the universal abelian scheme. \par
We show that
\begin{propS}
For each $\lambda$ and each cusp label representative $\Phi_0$ with a cone $\sigma\in \Sigma_0^+(\Phi_0)$, there is a vector bundle $\ca{V}_{\Phi_0}(\sigma)$ on the boundary toric scheme $\ca{S}_{K_{\Phi_0}}(\sigma)$ and a tensor $\wdtd{s}_{\Phi_0,\lambda,\dr}(\sigma)\in \ca{V}_{\Phi_0}(\sigma)^\otimes$ such that $\wdtd{s}_{\lambda,\dr}^\can$ and $\wdtd{s}_{\Phi_0,\lambda,\dr}(\sigma)$ coincide on $\cpl{\ca{S}^{\Sigma_0}_{K_0}}{\ca{Z}_{[(\Phi_0,\sigma)],K_0}}$. 
\end{propS}
With the above proposition, we can characterize the canonical extensions of principal $\G_0$-bundles on the integral model $\ca{S}_{K_0}^{\Sigma_0}$ using a property in \cite{Har89} and \cite{HZ01} (see Proposition \ref{prop-torsor-hz}). We note that, to the best of our knowledge, this result concerning $\G_0$-bundles on integral models has not previously appeared in the literature, even in the context of PEL-type Shimura varieties.\par 
The next step is to pass the torsors and vector bundles to a quotient by a finite Kummer {\'e}tale group action.
By the results proved in \cite{Wu25}, we know that this finite Kummer {\'e}tale quotient is finite {\'e}tale when restricted to each individual stratum, and one can explicitly compute such a group quotient.
The first task now is to extend such a group quotient to the torsor. To do this, we need to have a robust and functorial theory of canonical models of canonical extensions over the generic fiber, and over the integral models of Shimura varieties. We can almost extract it from the literature, although most references had the assumptions such as assuming $G=G^c$ or assuming $Z(G)^\circ$ splits over a CM field extension. For our purpose, we define the desired canonical models of canonical extensions on $\sh_K^\Sigma(G,X)$ or integral canonical models on $\ca{S}_K(G,X)$ by pulling back the corresponding objects defined by $(G^c,X^c)$.\par
Next, by an explicit computation of the quotient $\ca{W}$ at the formal completion at the boundary, we show that $\ca{W}$ is still formal locally a locally free sheaf. This computation uses toric charts and Proposition \ref{prop-torsor-hz} crucially (see Lemma \ref{lem-quotient-tors-general}). The proof also relies on a generalization of Deligne's induction functor in \cite{Del79} and \cite{Mil88}, especially a generalization of the $\ag$-type groups, to treat torsors. See \S\ref{subsec-ind-torsor}.\par
Finally, by adopting the construction in \cite{Wu25}, we extend the construction to all abelian-type Shimura data.
\subsubsection*{Structure of the paper}
In Section \ref{sec-can-ex-0}, we review the notion of the canonical model of a principal $G^c$-bundle and record the related results. In \S\ref{subsec-f-u-f}, we give a list of facts in \cite{Pin89}, which will be used freely in the paper. In \S\ref{subsec-prelim}, we collect some basic facts about $G^c$ when $G$ is a general connected linear algebraic group. In \S\ref{subsubsec-can-ext-can-mod}, we consider a general mixed Shimura datum $(Q,\ca{X})$. We review how the (canonical models of) automorphic vector bundles and principal bundles on $\sh_K(Q,\ca{X})$ are defined. When $(Q,\ca{X})=(G,X)$ is a usual Shimura datum, we explain how the (canonical models of) canonical extensions are defined and characterized. We will emphasize on how the action of $\mbf{E}_{K_\Phi}$ is defined on the principal bundles on mixed Shimura varieties and on the functoriality between the objects. In \S\ref{subsec-ind-torsor}, we give a general formalism of Deligne's induction construction for torsors. \par
In Section \ref{sec-can-ex-hyper}, we construct the integral models of the canonical extensions on toroidal compactifications of the integral canonical models of Shimura varieties of abelian type, based on the known results due to Madapusi in the case of Hodge-type Shimura varieties. In \S\ref{subsec-tor-ab-review}, we list some useful steps in the construction of integral models of toroidal compactifications for abelian-type Shimura varieties in \cite{Wu25}. (We will only focus on the case of hyperspecial levels, which will to some extent simplify the procedures.) In \S\ref{subsec-can-ext-hodge}, we give a characterization of the canonical extensions constructed in \cite{Mad19} by studying the boundary in more detail. 
In \S\ref{subsec-can-ext-ab}, we finish the construction and show the desired properties. In \S\ref{subsec-auto-vb}, we define the principal bundles under the parabolic and Levi subgroups of $G^c_{2,\zbkp}$ defined by the Hodge cocharacter, and define subcanonical extensions. 
\subsubsection*{Acknowledgments}
This note was extracted and revised from a very early and preliminary draft of \cite{Wu25}, so part of it was written when the author was a graduate student at the University of Minnesota. The author thanks his advisor Kai-Wen Lan for his guidance and advice on doing and writing math. The author is also grateful to Kentaro Inoue and Shengkai Mao for helpful communications.\par
The author extends his gratitude to BICMR for an excellent academic environment.
\subsection*{Notation and conventions}
\begin{itemize}
\item The general system of conventions in this note is compatible with \cite{Wu25}. Although we will review the notation in this note, the readers are encouraged to check the list of symbols and the list of conventions in \emph{loc. cit.}. 
\item The torsors over Shimura varieties will be denoted by the symbol ``$\ca{E}$''. The extensions of these torsors to integral models will be written in the symbol ``$\mathscr{E}$''.
\item In this paper, a representation of an affine group scheme $\G$ over a base ring $R$ means a finite free $R'$-module for some $R$-algebra $R'$ equipped with a $\G$-action. Let $W$ be a representation. The corresponding vector bundles will be denoted by the symbol ``$\ul{W}$''. Denote by $\rep(\G)$ the category of $\G$-representations, by $\rep_{R'}(\G)$ the $\G$-representations $W$ over $R'$.
\item Canonical extensions are denoted by ``$(-)^\can$''.
\item The symbol ``$\mrm{Vec}_S$'' in this paper always means the category of vector bundles with (log) connections on a scheme $S$.
\item (\textbf{Vector bundles and torsors coincide on formal completions}) Let $X$ and $Y$ be schemes with stratifications of locally closed subschemes. Let $C\sbst X$ (resp. $D\sbst Y$) be a stratum of $X$ (resp. $Y$). Suppose that there is an isomorphism $\cpl{X}{C}\iso\cpl{Y}{D}$ that induces an isomorphism between $C$ and $D$. Let $V_1$ and $V_2$ be vector bundles/torsors on $X$ and $Y$, respectively. We say that the pullback of $V_1$ and $V_2$ to $\cpl{X}{C}\iso\cpl{Y}{D}$ \textbf{coincide} if, for any open affine formal subscheme $\mrm{Spf}(A,I)$ of $\cpl{X}{C}\iso\cpl{Y}{D}$ that induces morphisms $c_1:\spec A\to X$ and $c_2:\spec A\to Y$, there is an isomorphism $f_{(A,I)}: c_1^*V_1\iso c_2^*V_2$, and if the isomorphisms satisfy the compatibility for varying $(A,I)$.
\end{itemize}
\section{Canonical models and canonical extensions of principal bundles}\label{sec-can-ex-0}
We recall the characteristic $0$ theory of canonical extensions of canonical models of principal bundles on Shimura varieties and their boundary mixed Shimura varieties due to Harris, Harris-Zucker (see \cite{Har85}, \cite{Har89}, \cite{HZ94} and \cite{HZ01}), and Milne (see \cite{Mil88} and \cite{Mil90}). 
%We also remove some assumptions in these references (see, e.g., (2.1.4) and (2.1.2*) in \cite[II. 2.]{Mil90}, (1.1.3) and (1.1.4) in \cite{Har89} and (1.1.7.2) and (1.1.7.3) in \cite{HZ01}).\par
\subsection{Mixed Shimura varieties}\label{subsec-f-u-f}
The following definitions and facts on mixed Shimura varieties and their canonical models are needed for this note:
\begin{enumerate}
    \item For any connected linear algebraic group $H$ over $\bb{Q}$, define $H^c=H/Z_{ac}(H)$ where $Z_{ac}(H)$ is the minimal $\bb{Q}$-subgroup in the connected multiplicative center $Z(H)^\circ$ of $H$ such that $Z(H)^\circ/Z_{ac}(H)$ has equal $\bb{Q}$-split rank and $\bb{R}$-split rank.
    \item A mixed Shimura datum is denoted by $(Q,\ca{X})$. More precisely, let $Q$ be a connected linear algebraic group over $\bb{Q}$. Denote by $W_Q$ the unipotent radical of $Q$ and by $U_Q\sbst W_Q$ a normal subgroup of $Q$. Let $\ca{X}$ be a homogeneous space with a $Q(\bb{R})U_Q(\bb{C})$-equivariant map $\hbar:\ca{X}\to \Hom(\bb{S}_\bb{C},Q_\bb{C})$. Then $(Q,\ca{X})$ is a mixed Shimura datum if it satisfies the conditions listed in \cite[Def. 2.1, (i)-(viii)]{Pin89}.
    \item For mixed Shimura data, one can define mixed Shimura varieties, morphisms between mixed Shimura data and varieties, and quotients in a way similar to the usual Shimura data/varieties. See \cite[Ch. 2 and Ch. 3]{Pin89}. 
    \item If $W_Q=1$, the mixed Shimura datum is called a pure Shimura datum. Then $\ca{X}$ is, by \cite[Cor. 2.12]{Pin89}, a finite cover of $\hbar(\ca{X})$ that is a homeomorphism over each connected component; $(Q,\hbar(\ca{X}))$ is a usual Shimura datum. 
    \item There is a sequence of morphisms between Shimura data defined by the quotient of $Q$ by $U_Q$ and $W_Q$: 
    $$(Q,\ca{X})\lra (Q/U_Q,\overline{\ca{X}})\lra (Q/W_Q, \ca{X}_h).$$
    Denote $\overline{Q}:=Q/U_Q$ and $Q_h:=Q/W_Q$.
    \item For any mixed Shimura datum $(Q,\ca{X})$ and any neat open compact subgroup $K\sbst Q(\A)$, denote $\sh_K(Q,\ca{X})(\bb{C}):=$
    $$Q(\bb{Q})\bss \ca{X}\times Q(\A)/K.$$
    This complex manifold uniquely algebraizes to a complex algebraic variety $\sh_K(Q,\ca{X})_\bb{C}$. 
    Denote by $\sh_K(Q,\ca{X})$ the canonical model of $\sh_K(Q,\ca{X})_\bb{C}$ over the reflex field $E(Q,\ca{X})$. Such a variety uniquely exists by \cite[Thm. 11.18]{Pin89}.
    \item Denote by $\overline{K}$ and $K_h$ the projection of $K$ to $\overline{Q}(\A)$ and $Q_h(\A)$. Then there is a sequence of mixed Shimura varieties
    \begin{equation}\label{eq-projection}\sh_K(Q,\ca{X})\xrightarrow{\ \mbf{p}_1\ } \sh_{\overline{K}}(\overline{Q},\overline{\ca{X}})\xrightarrow{\ \mbf{p}_2\ } \sh_{K_h}(Q_h,\ca{X}_h).\end{equation}
    The first morphism is a torsor under a split torus $\mbf{E}_K$, and the second morphism is a torsor under an abelian scheme over $\sh_{K_h}(Q_h,\ca{X}_h)$. We usually use the abbreviations $\sh_K:=\sh_K(Q,\ca{X})$, $\overline{\sh}_K:=\sh_{\overline{K}}(\overline{Q},\overline{\ca{X}})$, and $\sh_{K_h}:=\sh_{K_h}(Q_h,\ca{X}_h)$.
    \item Let $(G,X)$ be a usual Shimura datum. Denote by $\ca{CLR}(G,X)$ the set of cusp label representatives $\Phi$ of the form $\Phi=(Q_\Phi,X^+_\Phi,g_\Phi)$, where $Q_\Phi$ is an admissible $\bb{Q}$-parabolic subgroup of $G$, $X^+_\Phi\sbst X$ a connected component and $g_\Phi\in G(\A)$ (see \cite[\S 2.1]{Mad19}).
    There is a $Q_\Phi(\bb{R})$-equivariant map 
    $$\tau:X\to \pi_0(X)\times \Hom(\bb{S}_\bb{C},P_{\Phi,\bb{C}}),$$ 
    where $P_\Phi$ is a normal subgroup of $Q_\Phi$ with the same unipotent radical as $Q_\Phi$ defined as in \cite[4.7]{Pin89}. Note that $P_\Phi$ is already defined and determined by $Q_\Phi$, so one does not need the entire datum $(Q_\Phi,X^+_\Phi,g_\Phi)$; this group is called the \emph{canonical subgroup} of $Q_\Phi$. Denote by $D_\Phi$ the $P_\Phi(\bb{R})U_\Phi(\bb{C})$-orbit of the image of $X^+_\Phi$. Let $K_\Phi:=P_\Phi(\A)\cap g_\Phi K g_\Phi^{-1}$.\par 
    In summary, one associates with $\Phi$ a mixed Shimura datum $(P_\Phi,D_\Phi)$ and a mixed Shimura variety $\sh_{K_\Phi}:=\sh_{K_\Phi}(P_\Phi,D_\Phi)$. We call these mixed Shimura data (resp. varieties) boundary mixed Shimura data (resp. varieties).
    \item We denote by $\Sigma$ (usually followed by super/sub-scipts) a (rational polyhedral) cone decomposition defined on $\ca{CLR}(G,X)$. In this note, we say that a cone decomposition is admissible, if it is admissible in the sense of \cite{Pin89}, is complete and finite, and has no self-intersections (cf. \cite[Convention 1.17]{Wu25}); in fact, \textbf{we will always assume that $\Sigma$ is admissible}. We often choose a smooth and projective refinement of $\Sigma$ (cf. \cite[Def. 1.16]{Wu25}), which always exists. 
    \item Let $S$ be an $\mbf{E}$-torsor over $\overline{S}$. Let $\sigma$ be a rational cone in the cocharacter $\bb{R}$-group of $\mbf{E}$. Denote by $S\to S(\sigma)$ the twisted affine torus embedding, and by $S_\sigma$ the $\sigma$-stratum of $S(\sigma)$.
    \item There are compatible equivalence relations and partial orders defined on $\ca{CLR}(G,X)$ and the set of cusp label representatives $\Phi$ with cones in $\Sigma^+(\Phi)$'s. The quotients by the equivalence relations are called cusp labels $\cusp_K(G,X)$ and cusp labels with cones $\cusp_K(G,X,\Sigma)$. 
    \item Let $K$ be a neat open compact subgroup in $G(\A)$. For an admissible cone decomposition $\Sigma$ of $(G,X,K)$, there is a proper algebraic space $\sh^\Sigma_K(G,X)$ over the reflex field $E(G,X)$. Assuming that $\Sigma$ is projective (resp. smooth), $\sh_K^\Sigma$ is projective (resp. smooth) over $\spec E(G,X)$. There is a stratification of $\sh_K^\Sigma$ labeled by the finite set of cusp labels with cones $\cusp_K(G,X,\Sigma)$. Pick any $\Upsilon=[(\Phi,\sigma)]\in \cusp_K(G,X,\Sigma)$. There is a mixed Shimura variety $\sh_{K_\Phi}$ and a locally closed stratum $\mrm{Z}_{\Upsilon,K}\sbst \sh_K^\Sigma$ with a canonical isomorphism $\mrm{Z}_{\Upsilon,K}\iso \Delta^\circ_{\Phi,K}\bss \sh_{\Phi,K,\sigma}$. The group $\Delta^\circ_{\Phi,K}$ acts on $\sh_{K_\Phi}(\sigma)$ through the free action of a finite group.
    \item Let us consider the analytification $\sh^\Sigma_K(\bb{C})$ of $\sh^\Sigma_K$. There is a commutative diagram
    \begin{equation}\label{eq-ana-chart}
    \begin{tikzcd}
    \Delta^\circ_{\Phi,K}\bss\sh_{K_\Phi}(\bb{C})\arrow[rr,"\text{open dense}","\sbst"']&&\Delta^\circ_{\Phi,K}\bss\sh_{K_\Phi}(\sigma)(\bb{C})&\Delta^\circ_{\Phi,K}\bss \sh_{K_\Phi,\sigma}(\bb{C})\arrow[l,"\supset"']\\
   \Delta^\circ_{\Phi,K}\bss U(\Phi)\arrow[u,hook,"\mrm{open}"]\arrow[rr,"\text{open dense}","\sbst"']&& \Delta^\circ_{\Phi,K}\bss U(\Phi,\sigma)\arrow[u,hook,"\mrm{open}"]& \Delta^\circ_{\Phi,K}\bss U_\sigma\arrow[u,equal]\arrow[d,equal]\arrow[l,"\supset"']\\
   V(\Phi)\arrow[u,hook,"\text{open}"]\arrow[rr,"\text{open dense}","\sbst"']\arrow[d,hook,"\text{open}"]&& V(\Phi,\sigma)\arrow[u,hook,"\text{open}"]\arrow[d,hook,"\text{open}"]&\Delta^\circ_{\Phi,K}\bss U_\sigma\arrow[d,equal]\arrow[l,"\supset"']\\
    \sh_{K}(\bb{C})\arrow[rr,"\text{open dense}","\sbst"']&&\sh_K^\Sigma(\bb{C})&\mrm{Z}_{\Upsilon,K}(\bb{C}).\arrow[l,"\supset"']
    \end{tikzcd}
    \end{equation}
In a word, there is an open analytic neighborhood $V(\Phi,\sigma)$ around $\Delta^\circ_{\Phi,K}\bss \sh_{K_\Phi,\sigma}(\bb{C})$ isomorphic to an open analytic neighborhood $V(\mrm{Z}_{\Upsilon,K})$ of $\mrm{Z}_{\Upsilon,K}(\bb{C})$ in $\sh^\Sigma_K(\bb{C})$. This isomorphism induces a canonical algebraic isomorphism $\cpl{\sh^\Sigma_K}{\mrm{Z}_{\Upsilon,K}}\iso \Delta^\circ_{\Phi,K}\bss\cpl{\sh_{K_\Phi}(\sigma)}{\sh_{K_\Phi,\sigma}}$. See \cite[Ch. 6, Prop. 7.15, Prop. 9.37, Thm. 12.4(c)]{Pin89}.
\end{enumerate}
\subsection{Preliminaries}\label{subsec-prelim}
Let $(T,\ca{Y})$ be a pure Shimura datum (in the sense of Pink) associated with a torus $T$ over $\bb{Q}$. There is a natural morphism 
$$(T,\ca{Y})\lra (T,\hbar(\ca{Y}))\lra (T^c,\hbar(\ca{Y})^c)$$ 
defined by mapping $(T,\ca{Y})$ to the Shimura datum in the usual sense, and then to the quotient datum $(T^c,\hbar(\ca{Y})^c)$ induced by $T\to T^c$. 
\begin{lem}\label{lem-cusp-cm}
Denote $\{h:\bb{S}\to T^c_\bb{R}\}:=\hbar(\ca{Y})^c$. The weight cocharacter $\omega: \bb{G}_{m,\bb{C}}\lra T^c_{\bb{C}}$ associated with $h$ is defined over $\bb{Q}$. Moreover, $T^c$ does not contain nontrivial $\bb{R}$-split but $\bb{Q}$-anisotropic subtori. In particular, $(T^c,\{h\})$ satisfies the conditions (2.1.4) and (2.1.2*) in \cite[II. 2.]{Mil90}, (1.1.3) and (1.1.4) in \cite{Har89}, and (1.1.7.2) and (1.1.7.3) in \cite{HZ01}.
\end{lem}
\begin{proof}
Fix a complex conjugation $\iota\in\gal(\overline{\bb{Q}}/\bb{Q})$. Denote the Hodge cocharacter of $h$ by $\mu$. Since $\omega= \mu+\iota\mu$, it is fixed by $\iota$. By the equivalence of statements (vii) and (iv) in \cite[Lem. 1.5.5]{KSZ21}, $\omega$ is fixed by all $\sigma\in\gal(\overline{\bb{Q}}/\bb{Q})$.
The second statement follows from the equivalence of statements (ii) and (iv) in \emph{loc. cit.} Also in \emph{loc. cit.}, the second statement is equivalent to saying that $T^c$ is isogenous to the product of a $\bb{Q}$-split torus and an $\bb{R}$-anisotropic $\bb{Q}$-torus. By the Serre condition (vii) of \emph{loc. cit.}, $T^c$ splits over a CM field. Then the last statement is true.
\end{proof}
\begin{lem}\label{lem-tc-emb}
    Let $(Q,\ca{X})$ be a mixed Shimura datum. Let $(T,\ca{Y})\to (Q,\ca{X})$ be an embedding of a pure Shimura datum $(T,\ca{Y})$ for a $\bb{Q}$-torus $T$ into $(Q,\ca{X})$.\par 
    Suppose that $T$ is $\bb{Q}$-maximal, i.e., any $\bb{Q}$-torus $T'$ in $G$ containing $T$ is $T$ itself. Then $Z_{ac}(T)\iso Z_{ac}(Q)$ and $T^c\iso T/Z_{ac}(Q)$. Hence, there is an embedding $T^c\hookrightarrow Q^c$. As a consequence, this embedding induces an embedding $(T,\ca{Y})/Z_{ac}(T)\hookrightarrow (Q,\ca{X})/Z_{ac}(Q)$.
\end{lem}
\begin{proof}
Denote by $U_Q$ the unipotent radical of $Q$. By assumption, $T$ contains $Z(Q)^\circ$. Since $Z_{ac}(T)\supset Z_{ac}(Q)$ by definition, we show the other direction. Since for any $x\in \ca{X}$, $ ad h_x(i)$ induces a Cartan involution of $(Q/U_Q)^\ad_{\bb{R}}$, $T/Z(Q)^\circ$ is $\bb{R}$-anisotropic. Then $T/Z_{ac}(Q)\cap T$ is isogenous to the product of a $\bb{Q}$-split torus and an $\bb{R}$-anisotropic $\bb{Q}$-torus. This shows that $Z_{ac}(T)\iso Z_{ac}(Q)$ and $T^c\iso T/Z_{ac}(Q)$. Now we show the last statement. We choose any $x_0\in \ca{Y}$, whose image in $\ca{X}$ is also denoted by $x_0$. Denote the image of $x_0$ in the quotients $\ca{Y}/Z_{ac}(Q)$ and $\ca{X}/Z_{ac}(Q)$ abusively by $\overline{x}_0$. From \cite[Prop. 2.9]{Pin89}, there is a commutative diagram:
 \begin{equation*}
     \begin{tikzcd}
         \ca{Y}\iso T(\bb{R})/\stb_{T(\bb{R})}(x_0)\arrow[r,hook,"\varepsilon"]\arrow[d]&\ca{X}\iso Q(\bb{R})U_Q(\bb{C})/\stb_{Q(\bb{R})U_Q(\bb{C})}(x_0)\arrow[d]\\
    \ca{Y}/Z_{ac}(Q)\iso T^c(\bb{R})/\im \stb_{T(\bb{R})}(x_0)\arrow[r,"\overline{\varepsilon}"]&\ca{X}/Z_{ac}(Q)\iso Q^c(\bb{R})U_Q(\bb{C})/\im \stb_{Q^c(\bb{R})U_Q(\bb{C})}(x_0).
     \end{tikzcd}
 \end{equation*}
Suppose that there is an element $a\in T^c(\bb{R})$ which maps into $\im \stb_{Q^c(\bb{R})U_Q(\bb{C})}(x_0)$. We denote $b:=\overline{\varepsilon}(a)$. By construction, $b$ lifts to some $c\in \stb_{Q(\bb{R})U_Q(\bb{C})}(x_0)$. Denote by $\overline{b}$ (resp. $\overline{c}$) the image of $b$ (resp. $c$) in $(Q^c/U_Q)(\bb{R})$ (resp. $(Q/U_Q)(\bb{R})$). Finally, we find an element $d\in T(\bb{R})$ since $T$ is the fiber product of $T^c$ and $Q/U_Q$ over $Q^c/U_Q$. Since $\varepsilon (d)=c$ and $\varepsilon$ is injective, we see that $d\in \stb_{T(\bb{R})}(x_0)$. We now have the desired result.
\end{proof}
\begin{lem}[{\cite[Lem. 3.1.3]{Lov17}}]\label{lem-hc-func} Suppose that $(Q_2,\ca{X}_2)$ is a mixed Shimura datum, and that there is a homomorphism $f: Q_1\to Q_2$ from another connected linear algebraic group $Q_1$ over $\bb{Q}$ to $Q_2$. Then $f(Z_{ac}(Q_1))\sbst Z_{ac}(Q_2)$, and $f$ induces a homomorphism $f^c:Q_1^c\to Q_2^c$. 
In particular, if there is a morphism between the mixed Shimura data $f:(Q_1,\ca{X}_1)\to (Q_2,\ca{X}_2)$, then $f$ induces a morphism between mixed Shimura data $f^c:(Q_1,\ca{X}_1)/Z_{ac}(Q_1)\to (Q_2,\ca{X}_2)/Z_{ac}(Q_2)$. 
\end{lem}
\begin{proof}
The proof of it is essentially identical to \cite[Lem. 3.1.3]{Lov17}. To show that $f(Z_{ac}(Q_1))\sbst Z_{ac}(Q_2)$, it suffices to show that the image of the connected multiplicative center $Z(Q_1)^\circ$ in $f(Q_1)/f(Q_1)\cap Z_{ac}(Q_2)$ is isogenous to a product of a $\bb{Q}$-split torus and an $\bb{R}$-anisotropic torus. Since $Z(Q_2)/Z_{ac}(Q_2)$ and the center of $Q_2/Z(Q_2)W_{Q_2}$ are products of $\bb{Q}$-split tori and $\bb{R}$-anisotropic tori by the definition of $Z_{ac}(Q_2)$ and by \cite[2.1 (viii)]{Pin89}, and since, for any $x\in \ca{X}_2$, $\ad h_x(i)$ induces a Cartan involution of $(Q_2/W_{Q_2})^\ad_\bb{R}$, the image of $Z(Q_1)^\circ$ has the desired property. We have the proof of the first statement. Other statements follow from the definitions.
\end{proof}
\subsection{Canonical extensions and their canonical models}\label{subsubsec-can-ext-can-mod}
\subsubsection{}\label{subs-can-ex-summary}Fix a mixed Shimura datum $(Q,\ca{X})$. 
Let us recall some background materials for principal bundles over $\sh_K(\bb{C}):=\sh_K(Q,\ca{X})(\bb{C})$ and $\sh_{K,\bb{C}}:=\sh_K(Q,\ca{X})_\bb{C}$. (We assume that $K$ is neat as usual.) The main references are \cite{Mil90} and the summary in \cite[Sec. 5.2]{DLLZ} for usual Shimura data.\par
Define the principal $Q^c(\bb{C})$-bundle over $\sh_{K}(Q,\ca{X})(\bb{C})$ to be $$\ca{E}_{Q^c,K}(Q,\ca{X})(\bb{C}):=Q(\bb{Q})\bss \ca{X}\times Q^c(\bb{C})\times Q(\A)/K.$$
We have, by definition, that $[(x,\ q_c,\ q_f)]=[(q\cdot x,\ q\cdot q_c,\ q\cdot q_f\cdot k)]$ for any $x\in X,\ q_c\in Q^c(\bb{C}),\ q_f\in Q(\A),\ q\in Q(\bb{Q})$ and $k\in K$.\par
The principal bundle $\ca{E}_{Q^c,K}(Q,\ca{X})(\bb{C})$ determines and is determined by a tensor functor $\omega_{B,Q^c,K,\bb{C}}^{an}(Q,\ca{X}): \mrm{Rep}_\bb{C}(Q^c)\to \lsys_{\sh_K(\bb{C})}$, from the category $\rep_\bb{C}(Q^c)$ of finite dimensional complex algebraic representations of $Q^c$ to the category of local systems of finite dimensional $\bb{C}$-vector spaces over $\sh_K(\bb{C})$. \par 
Let $(W,\rho)\in \rep_\bb{C}(Q^c)$. Denote $\autoshbetti{W}_K^{an}:=\omega_{B,Q^c,K,\bb{C}}^{an}(Q,\ca{X})(W)= \ca{E}_{Q^c,K}(Q,\ca{X})(\bb{C})\times^{Q^c(\bb{C})} W$; if $W\in \rep_F(Q^c)$ for some subfield $F\sbst \bb{C}$, $\autoshbetti{W}_K^{an}$ descends to a local system of $F$-vector spaces.\par 
Moreover, we can choose any smooth projective (complete) admissible cone decomposition of $\sh_K(Q,\ca{X})$ and there is an open embedding $j:\sh_K(\bb{C})\to \sh^\Sigma_K(\bb{C}):=\sh^\Sigma_K(Q,\ca{X})(\bb{C})$. The boundary divisor is a normal crossings divisor. See \cite[Prop. 9.20]{Pin89}.
\begin{lem}\label{lem-mono-uni}The monodromy action on $j_*\autoshbetti{W}_K^{an}$ of the topological fundamental group of each connected component along the boundary divisor is unipotent.
\end{lem}
\begin{proof}The proof is identical to \cite[p.18]{Har89} and \cite[p.34]{LS13}, which treated the case where $(Q,\ca{X})$ is a usual Shimura datum such that $Q=Q^c$. The completion of any point at the boundary is isomorphic to the completion of a twisted torus embedding $\sh_{K_\Phi}(\Sigma)(\bb{C}):=\sh_{K_\Phi}(P_\Phi,\ca{D}_\Phi,\Sigma)(\bb{C})$, where $(P_\Phi,\ca{D}_\Phi)$ is a rational boundary component of $(Q,\ca{X})$. Each of the connected components of $\sh_{K_\Phi}(\Sigma)(\bb{C})$ can be written as $\Gamma\bss\ca{D}_\Phi\to \overline{\Gamma}\bss \overline{\ca{D}}_\Phi$, which is a torus torsor under $\mbf{E}:=\Lambda \bss U_\Phi(\bb{C})$ for some lattice $\Lambda$ in a unipotent group $U_\Phi(\bb{R})(-1)$. More precisely, the group $\Lambda$ is (up to twisting $-1$) the projection to $U_\Phi(\bb{Q})$ of a \emph{neat} congruence subgroup $\wdtd{\Lambda}$ of $Z_Q(\bb{Q})\times U_\Phi(\bb{Q})$ by \cite[3.13, 4.10]{Pin89} (cf. $\Lambda_{K_\Phi}$ in \cite{Wu25}). Hence, the monodromy along the boundary is determined by an action $\rho:\wdtd{\Lambda}\to \wdtd{\Lambda}/(Z_{ac}(Q)(\bb{Q})\cap \wdtd{\Lambda})\to \mrm{GL}(W)$, which is unipotent because $(Z_Q/Z_{ac}(Q))(\bb{Q})$ has no nontrivial neat congruence subgroups. 
\end{proof}
\begin{rk}
If $W$ is a $Q(\bb{C})$-representation with nontrivial action of $Z_{ac}(Q)(\bb{Q})$, the monodromy above is only quasi-unipotent.
\end{rk}
By classical Riemann-Hilbert correspondence, the category $\lsys_{\sh_K(\bb{C})}$ is equivalent to $\vbd_{\sh_K(\bb{C})}$, the category of finite dimensional analytic vector bundles over $\sh_K(\bb{C})$ equipped with integrable connections. Hence, composing with this equivalence, we see that $\omega_{B,Q^c,K,\bb{C}}^{an}(Q,\ca{X})$ is equivalent to a tensor functor $$\omega^{an}_{\dr,Q^c,K,\bb{C}}(Q,\ca{X}):\rep_\bb{C}(Q^c)\lra \vbd_{\sh_K(\bb{C})}.$$ 
Denote $(\autoshdr{W}_{K,\bb{C}}^{an},\nabla_\bb{C}^{an})$ the value of $W$ under $\omega_{\dr,Q^c,K,\bb{C}}^{an}(Q,\ca{X})$ where $\nabla_\bb{C}^{an}$ denotes the integrable connection. 
There is a decreasing filtration $\fil^\bullet_x$ on $W$ defined by the Hodge cocharacter $\mu_{x,\bb{C}}$ of any $x\in \ca{X}$; these filtrations induce a decreasing filtration $\fil^\bullet$ on $\autoshbetti{W}_{K,\bb{C}}^{an}$, and via Riemann-Hilbert correspondence, on $\autoshdr{W}_{K,\bb{C}}^{an}$. The integrable connection $\nabla_\bb{C}^{an}$ satisfies Griffiths transversality with respect to $\fil^\bullet$.\par
The tensor functor $\omega^{an}_{\dr,Q^c,K,\bb{C}}(Q,\ca{X})$ determines and is determined by a $Q^c_\bb{C}$-bundle in complex analytic topology
$$\ca{E}^{an}_{Q^c,K,\bb{C}}(Q,\ca{X}):=Q(\bb{Q})\bss \ca{X}\times Q^c_\bb{C}\times Q(\A)/K.$$

By \cite[II, 5.2(d)]{Del70} and \cite[p.235]{Sch73}, the triple $(\autoshdr{W}_{K,\bb{C}}^{an},\nabla_{\bb{C}}^{an},\fil^\bullet)$ uniquely extends to a triple $(\autoshdr{W}_{K,\bb{C}}^{\Sigma,an},\nabla^{\Sigma,an}_\bb{C},\fil^\bullet)$ on 
$\sh_K^\Sigma(\bb{C})$, where $\autoshdr{W}_{K,\bb{C}}^{\Sigma,an}$ is an analytic vector bundle with an integrable log connection $\nabla_\bb{C}^{\Sigma,an}$ with nilpotent residues along the boundary divisor, and the decreasing filtration $\fil^\bullet$ extends the one on $\autoshdr{W}_{K,\bb{C}}^{an}$ (and is denoted abusively by the same notation). \par
%In fact, the extension of Hodge filtrations can be seen by the construction of toroidal compactifications, where the Hodge cocharacters $\mu_{x,\bb{C}}$ are preserved through $\ca{X}\xrightarrow{\tau} \pi_0(\ca{X})\times \Hom(\bb{S}_\bb{C},P_{\Phi,\bb{C}})\to \Hom(\bb{S}_\bb{C},P_{\Phi,\bb{C}})$.\par 
By \cite{Del70}, the triple $(\autoshdr{W}_{K,\bb{C}}^{\Sigma,an},\nabla^{\Sigma,an}_\bb{C},\fil^\bullet)$ algebraizes to a tuple $(\autoshdr{W}_{K,\bb{C}}^\Sigma,\nabla_\bb{C}^\Sigma,\fil^\bullet)$ of algebraic vector bundle $\autoshdr{W}_{K,\bb{C}}^\Sigma$ over $\sh^\Sigma_{K,\bb{C}}:=\sh^\Sigma_K(Q,\ca{X})_\bb{C}$ equipped with an integrable log connection $\nabla^{\Sigma}_\bb{C}$ with nilpotent residues along the boundary divisor and a decreasing filtration $\fil^\bullet$. The restriction of $(\autoshdr{W}_{K,\bb{C}}^{\can},\nabla^{\Sigma}_\bb{C},\fil^\bullet)$ to $\sh_{K,\bb{C}}$ defines an algebraization $(\autoshdr{W}_{K,\bb{C}},\nabla_\bb{C},\fil^\bullet)$ of $(\autoshdr{W}_{K,\bb{C}}^{an},\nabla^{an}_\bb{C},\fil^\bullet)$. The extension $(\autoshdr{W}_{K,\bb{C}}^{\Sigma},\nabla^{\Sigma}_\bb{C},\fil^\bullet)$ is called the \textbf{canonical extension} of $(\autoshdr{W}_{K,\bb{C}},\nabla_\bb{C},\fil^\bullet)$.\par 
By the last paragraph, $\omega^{an}_{\dr,Q^c,K,\bb{C}}(Q,\ca{X})$ algebraizes to a tensor functor 
$$\omega_{\dr,Q^c,K,\bb{C}}(Q,\ca{X}):\rep_\bb{C}(Q^c)\lra \vbd_{\sh_{K,\bb{C}}},$$
and this tensor functor canonically extends to a tensor functor
$$\omega_{\dr,Q^c,K,\bb{C}}^\Sigma(Q,\ca{X}):\rep_{\bb{C}}(Q^c)\lra \vbd_{\sh_{K,\bb{C}}^\Sigma}.$$
The functor $\omega_{\dr,Q^c,K,\bb{C}}(Q,\ca{X})$ (resp. $\omega_{\dr,Q^c,K,\bb{C}}^\Sigma(Q,\ca{X})$) determines a principal (algebraic) $Q^c_\bb{C}$-bundle $\ca{E}_{Q^c,K,\bb{C}}(Q,\ca{X})$ (resp. a canonical extension of principal $Q^c_\bb{C}$-bundle $\ca{E}_{Q^c,K,\bb{C}}^\Sigma(Q,\ca{X})$).\par
\begin{convention}We will replace superscripts such as ``$\Sigma$'' specifying the cone decompositions with just ``$\can$'' if the underlying compactifications are clear in the context. For example, we will just write the algebraic canonical extensions as $(\autoshdr{W}_{K,\bb{C}}^\can,\nabla_\bb{C}^{\can},\fil^\bullet)$ when $\Sigma$ is clear. We will also omit the brackets $(Q,\ca{X})$ in the symbols when it is clear in the context.\end{convention}
\begin{lem}\label{lem-funct-complex-principal-bundle}
For a morphism between mixed Shimura data $f:(Q_1,\ca{X}_1)\to (Q_2,\ca{X}_2)$ and neat open compact $K_1\sbst Q_1(\A)\to K_2\sbst Q_2(\A)$, we have an algebraic isomorphism
$$\ca{E}_{Q_1^c,K_1,\bb{C}}(Q_1,\ca{X}_1)\times^{Q_1^c}Q_2^c\xrightarrow{\sim}f^*\ca{E}_{Q_2^c,K_2,\bb{C}}(Q_2,\ca{X}_2)$$
whose analytification is the one constructed by functoriality.
\end{lem}
\begin{proof}
Choose compatible admissible smooth projective cone decompositions $\Sigma_1$ and $\Sigma_2$. The corresponding morphism between toroidal compactifications extending $f$ is denoted by $f^{\mrm{tor}}$. By the uniqueness of canonical extensions and the Tannakian formalism, $\ca{E}^{an,\Sigma_1}_{Q^c_1,K_1,\bb{C}}(Q_1,\ca{X}_1)\times^{Q^c_1}Q^c_2\iso f^{\mrm{tor},*}\ca{E}^{an,\Sigma_2}_{Q^c_2,K_2,\bb{C}}(Q_2,\ca{X}_2)$. The desired result follows from GAGA and restriction.
\end{proof}
\subsubsection{}\label{subsubsec-E-action-mix-sh-complex} We now assume that $(Q,\ca{X})$ is a Shimura datum $(G,X)$ in the usual sense. 
Let us pick any $\Phi=(Q_\Phi,X^+_\Phi,g_\Phi)$ that determines a boundary component contained in $\sh_K^\mmin(G,X)\bss \sh_K(G,X)$. This defines an embedding $P_\Phi\hookrightarrow Q_\Phi\hookrightarrow G$ of algebraic groups.\par 
Let $K_\Phi:=P_\Phi(\A)\cap g_\Phi Kg_\Phi^{-1}$. Denote by $\sh_{K_\Phi}:=\sh_{K_\Phi}(P_\Phi,D_\Phi)$ the mixed Shimura variety associated with $\Phi$, whose analytification is 
$$\sh_{K_\Phi}(P_\Phi,D_\Phi)(\bb{C}):=P_\Phi(\bb{Q})\bss D_\Phi\times P_\Phi(\A)/K_\Phi.$$

Since $W\in \rep_F(G^c)$ is lifted to $\rep_F(P_\Phi^c)$ by Lemma \ref{lem-hc-func}, for any $W\in \rep_F(G^c)$ (recall that $F\sbst \bb{C}$ is a subfield), define
$$\autoshbetti{W}_{K_\Phi}^{an}:= P_\Phi(\bb{Q})\bss D_\Phi \times W\times P_\Phi(\A)/K_\Phi\iso \ca{E}_{P_\Phi^c,K_\Phi}(\bb{C})\times^{P_\Phi^c(\bb{C})}W(\bb{C}).$$
The group $P_\Phi(\bb{Q})$ acts on $W$ by viewing $W$ as a representation $\rho: P^c_\Phi\to \mrm{GL}(W)$ in $\rep_F(P^c_\Phi)$. \par
Similarly, let $\autoshdr{W}_{K_\Phi,\bb{C}}^{?}:=\omega^{?}_{\dr,P^c_\Phi,K,\bb{C}}(W)$ for $?=an$ or $\emptyset$. We can associate a filtered integrable connection with nilpotent residues $(\nabla^{an}_\bb{C},\fil^\bullet)$ (resp. $(\nabla_\bb{C},\fil^\bullet)$) with $\autoshdr{W}_{K_\Phi,\bb{C}}^{an}$ (resp. ($\autoshdr{W}_{K_\Phi,\bb{C}}$)) as explained before. \par 
Let $U_\Phi$ be the center of the unipotent radical $W_\Phi$ of $P_\Phi$. Then, by viewing $W(\bb{C})$ as a $U_\Phi(\bb{C})$-representation via restriction, there is a morphism $U_\Phi(\bb{C})\times W(\bb{C})\to W(\bb{C})$ determined by $\rho$. There is also an action of $U_\Phi(\bb{C})$ on $D_\Phi$ by conjugation. Consider the composition of maps 
\begin{equation*}
\begin{tikzcd}
    {(x,w)\in D_\Phi\times W(\bb{C})}\arrow[rr,"{\rho\mapsto \rho(u)\rho\rho(u)^{-1}}"]&&{(x,u\cdot w)\in D_\Phi\times W(\bb{C})}\arrow[r,"\mrm{int}(u)"]& {(u\cdot x,u\cdot w)\in D_\Phi\times W(\bb{C}).}
   \end{tikzcd}
\end{equation*}
This induces an isomorphism $[u]: \ca{E}_{P^c_\Phi,K_\Phi,\bb{C}}^{an}\xrightarrow{\sim} \mrm{int}(u)^* \ca{E}^{an}_{P^c_\Phi,K_\Phi,\bb{C}}$. In other words, there is a map 
$$U_\Phi(\bb{C})\times \ca{E}_{P^c_\Phi,K_\Phi,\bb{C}}^{an}\to \ca{E}^{an}_{P^c_\Phi,K_\Phi,\bb{C}}$$
covering the left action of $U_{\Phi}(\bb{C})$ on $\sh_{K_\Phi}(\bb{C})$.\par
%For later use and for simplicity, we can assume that the center of $P_\Phi$ is isogenous to a product of $\bb{Q}$-split tori and compact-type tori.
Write $\sh_{K_\Phi}(\bb{C})=P_\Phi(\bb{Q})_+\bss D_\Phi^+\times P_\Phi(\A)/K_\Phi$. For any $p\in P_\Phi(\A)$, the morphism $\sh_{K_\Phi}(\bb{C})\to \overline{\sh}_{K_\Phi}(\bb{C})$ is a torsor under $\Lambda_p\bss U_\Phi(\bb{C})$ for $\Lambda_p$ the projection to $U_\Phi(\bb{Q})$ of $\wdtd{\Lambda}_p:=(Z_G\cdot U_\Phi)^\circ(\bb{Q})\cap pK_\Phi p^{-1} \sbst P_\Phi(\bb{Q})_+\cap pK_\Phi p^{-1}$. If $u\in \wdtd{\Lambda}_p$, we have 
$$[(u\cdot x,u\cdot w,p)]=[(x,w, u^{-1}\cdot p)]=[(x,w,p)]$$
for points on $\ca{E}_{P_\Phi^c,K_\Phi,\bb{C}}^{an}$.\par
In summary, 
\begin{lem}\label{lem-u-action-complex}
There is a commutative diagram of complex manifolds 
\begin{equation}\label{eq-conj-torus-complex}
    \begin{tikzcd}
    \mbf{E}_{K_\Phi}(\bb{C})\times_\bb{C} \ca{E}_{P^c_\Phi,K_\Phi,\bb{C}}^{an}\arrow[r]\arrow[d]&\ca{E}_{P^c_\Phi,K_\Phi,\bb{C}}^{an}\arrow[d]\\
\mbf{E}_{K_\Phi}(\bb{C})\times_\bb{C} \sh_{K_\Phi}(\bb{C})\arrow[r]&\sh_{K_\Phi}(\bb{C})
    \end{tikzcd}
\end{equation}
induced by the action $U_\Phi(\bb{C})\times \ca{E}^{an}_{P^c_\Phi,K_\Phi,\bb{C}}\to \ca{E}^{an}_{P^c_\Phi,K_\Phi,\bb{C}}$ above.
\end{lem}
\begin{lem}\label{lem-u-action-alg}
The commutative diagram (\ref{eq-conj-torus-complex}) algebraizes to a commutative diagram of complex algebraic varieties
\begin{equation}\label{eq-conj-torus-complex-alg}
    \begin{tikzcd}
    \mbf{E}_{K_\Phi,\bb{C}}\times_{\bb{C}} \ca{E}_{P^c_\Phi,K_\Phi,\bb{C}}\arrow[r]\arrow[d]&\ca{E}_{P^c_\Phi,K_\Phi,\bb{C}}\arrow[d]\\
\mbf{E}_{K_\Phi,\bb{C}}\times_{\bb{C}} \sh_{K_\Phi,\bb{C}}\arrow[r]&\sh_{K_\Phi,\bb{C}}.
    \end{tikzcd}
\end{equation}
\end{lem}
\begin{proof}We only need to check the algebracity. 
The $\mbf{E}_{K_\Phi}$-torsor structure of $\sh_{K_\Phi}\to\overline{\sh}_{K_\Phi}$ induces a canonical algebraic isomorphism $\mbf{E}_{K_\Phi,\bb{C}}\times_{\bb{C}} \sh_{K_\Phi,\bb{C}}\iso \sh_{K_\Phi,\bb{C}}\times_{\overline{\sh}_{K_\Phi,\bb{C}}}\sh_{K_\Phi,\bb{C}}$ given by $(e,x)\mapsto (x,e\cdot x)$. Then the bottom arrow is given by 
$$\mbf{E}_{K_\Phi,\bb{C}}\times_{\bb{C}} \sh_{K_\Phi,\bb{C}}\iso \sh_{K_\Phi,\bb{C}}\times_{\overline{\sh}_{K_\Phi,\bb{C}}}\sh_{K_\Phi,\bb{C}}\xrightarrow{p_2} \sh_{K_\Phi,\bb{C}},$$
the second morphism above is the projection to the second factor, which is algebraic by \cite[Prop. 9.24]{Pin89} and the page following it.\par
Similarly, $\mbf{E}_{K_\Phi}(\bb{C})\times \ca{E}^{an}_{P^c_\Phi,K_\Phi,\bb{C}}\iso $
$$(P_\Phi(\bb{Q})\times_{\overline{P}_\Phi(\bb{Q})}P_\Phi(\bb{Q}))\bss( D_\Phi\times_{\overline{D}_\Phi}D_\Phi)\times P^c_{\Phi,\bb{C}}\times( P_\Phi(\A)\times_{\overline{P}_\Phi(\A)} P_\Phi(\A))/(K_\Phi\times_{\overline{K}_\Phi}K_\Phi)$$
given by $(e,(x,c,p))\mapsto [((x,\wdtd{e}\cdot x),\wdtd{e}c,(p,p))]$ for any lifting of $e\in \mbf{E}_{K_\Phi}(\bb{C})$ to $\wdtd{e}\in U_\Phi(\bb{C})$. 
This is the analytification of an algebraic isomorphism
\begin{equation*}
 \begin{split}   
&(U_{\Phi,\bb{C}}\times_\bb{C}\mbf{E}_{K_\Phi,\bb{C}}\times_\bb{C}\ca{E}_{P^c_\Phi,K_\Phi,\bb{C}})\times^{U_{\Phi,\bb{C}}\times_\bb{C} \mbf{E}_{K_\Phi,\bb{C}}}\mbf{E}_{K_{\Phi},\bb{C}}\iso \\
&(\ca{E}_{P^c_\Phi,K_\Phi,\bb{C}}\times_{\overline{\sh}_{K_\Phi,\bb{C}}}\ca{E}_{P^c_\Phi,K_\Phi,\bb{C}})\times^{U_{\Phi,\bb{C}}\times_\bb{C} \mbf{E}_{K_\Phi,\bb{C}}}\mbf{E}_{K_\Phi,\bb{C}}.
\end{split}
\end{equation*}
Thus, the top arrow is induced by this algebraic isomorphism composed with the projection of the last displayed expression to the second factor $\ca{E}_{P^c_\Phi,K_\Phi,\bb{C}}$ in the bracket, which is algebraic by Lemma \ref{lem-funct-complex-principal-bundle}.
\end{proof}
Fix a cusp label with cones $\Upsilon=[(\Phi,\sigma)]$. Define 
\begin{equation}\label{eq-ext-W-boundary}\autoshdr{W}_{K_\Phi,\bb{C}}^{?}(\sigma):=\mbf{p}_1(\sigma)^*((\mbf{p}_{1,*}\autoshdr{W}_{K_\Phi,\bb{C}}^{?})^{\mbf{E}_{K_\Phi}}).
\end{equation}
and 
\begin{equation}\label{eq-ext-W-boundary-e}\ca{E}_{P_\Phi^c,K_\Phi,\bb{C}}(\sigma):=\mbf{p}_1(\sigma)^*((\mbf{p}_{1,*}\ca{E}_{P^c_\Phi,K_\Phi,\bb{C}})^{\mbf{E}_{K_\Phi}}).
\end{equation}
In the definition, $\mbf{p}_1$ is the first projection in (\ref{eq-projection}) over $\bb{C}$, $\mbf{p}_1(\sigma)$ is the projection $\sh_{K_\Phi}(\sigma)_\bb{C}\to\overline{\sh}_{K_\Phi,\bb{C}}$, and $?=an$ or $\emptyset$. The $\mbf{E}_{K_\Phi}$-invariance makes sense by (\ref{eq-conj-torus-complex-alg}) above.\par
In addition, note that inclusion $P_\Phi\to G$ induces a map $P^c_\Phi\to G^c$.
%(see \textcolor{red}{cite Mao-Wu}).
\begin{thm}[{Harris, see \cite{Har89}}]\label{thm-can-ex-Harris-complex}Let $\Sigma$ be an admissible projective cone decomposition of $\sh_{K,\bb{C}}:=\sh_K(G,X)_\bb{C}$ (which might not be smooth). Denote by $\sh_{K,\bb{C}}^\Sigma$ the toroidal compactification associated with it.\par
There is a unique extension of $\autoshdr{W}_{K,\bb{C}}$ from a locally free sheaf $\sh_{K,\bb{C}}$ to a locally free sheaf $\autoshdr{W}_{K,\bb{C}}^{\Sigma}$ on $\sh_{K,\bb{C}}^\Sigma$. The sheaf $\autoshdr{W}^\Sigma_{K,\bb{C}}$ is uniquely characterized by the following property:\par 
For any cusp label with cones $\Upsilon=[(\Phi,\sigma)]$, the pullback of the analytification of $\autoshdr{W}_{K,\bb{C}}^\can$ to the open analytic neighborhood $V(\mrm{Z}_{\Upsilon,K})$ of $\mrm{Z}_{\Upsilon,K}(\bb{C})$ is isomorphic to the pullback of the analytification of $\autoshdr{W}_{K_\Phi,\bb{C}}(\sigma)$ to the open analytic neighborhood $V(\Phi,\sigma)$ of $\Delta^\circ_{\Phi,K}\bss \sh_{K_\Phi,\sigma}(\bb{C})$ such that $V(\mrm{Z}_{\Upsilon,K})\iso V(\Phi,\sigma)$ in the analytic construction of toroidal compactification (see \cite[Ch, 6]{Pin89} and \S\ref{subsec-f-u-f} (13)).\par
When $\Sigma$ is smooth, it coincides with the canonical extension defined by \cite{Del70}. In particular, the integrable connection $\nabla_\bb{C}$ of $\autoshdr{W}_{K,\bb{C}}$ uniquely extends to a log connection $\nabla_\bb{C}^\can$ with nilpotent residues along irreducible components of the boundary divisor.\par
Moreover, let $\Sigma'$ be a refinement of $\Sigma$. This refinement determines a proper morphism $\pi_{\Sigma',\Sigma,\bb{C}}: \sh_{K,\bb{C}}^{\Sigma'}\to \sh_{K,\bb{C}}^\Sigma$. Then $\pi^*_{\Sigma',\Sigma,\bb{C}}\autoshdr{W}_{K,\bb{C}}^\Sigma\iso \autoshdr{W}_{K,\bb{C}}^{\Sigma'}$ and $\pi_{\Sigma',\Sigma,\bb{C},*}\autoshdr{W}_{K,\bb{C}}^{\Sigma'}\iso \autoshdr{W}_{K,\bb{C}}^{\Sigma}$.\par
%Similar statemtents are true for $\autoshdr{W}_{K,\bb{C}}^{an}$.
\end{thm}
\begin{proof}
    All proofs can be found in \cite[Sec. 4.2-4.3]{Har89}. Note that all these proofs work without the assumptions (1.1.3) and (1.1.4) there. We can also sketch a direct proof as follows. Consider the morphism $f:(G,X)\to (G^c,X^c)=(G,X)/Z_{ac}(G)$. Choose any special pair $(T,\{h\})\hookrightarrow (G,X)$. By Lemma \ref{lem-cusp-cm}, we see that $(G^c,X^c)$ satisfies those assumptions in \emph{loc. cit}. Choose neat open compact subgroups $K\sbst G(\A)$ and $K^c\sbst G^c(\A)$ such that $f(K)\sbst K^c$. Also, choose any pair of compatible admissible cone decompositions $(\Sigma,\Sigma^c)$ for $(G,X,K)$ and $(G^c,X^c,K^c)$, respectively. The pullback of $\autoshdr{W}_{K^c,\bb{C}}^{\Sigma^c}$ under $f^{\mrm{tor}}: \sh_{K,\bb{C}}^\Sigma\to \sh_{K^c,\bb{C}}^{\Sigma^c}$ satisfies the property that uniquely characterizes the canonical extension in the statement. The compatibility in the second paragraph is checked using this characterization. Finally, when $\Sigma$ is smooth, since the log connections associated with $\autoshdr{W}_{K_\Phi,\bb{C}}(\sigma)$ has nilpotent residues by the proof of Lemma \ref{lem-mono-uni}, this construction of canonical extensions coincides with that of \cite{Del70}.
\end{proof}
\begin{cor}
The theorem above also gives a characterization of $\ca{E}_{G^c,K,\bb{C}}^\can(G,X)$:\par
The pullback of the analytification of $\ca{E}_{G^c,K_\Phi,\bb{C}}(\sigma):=\ca{E}_{P^c_\Phi,K_\Phi,\bb{C}}(\sigma)\times^{P^c_\Phi}G^c$ to $V(\Phi,\sigma)$ is isomorphic to the pullback of the analytification of $\ca{E}_{G^c,K,\bb{C}}^{\can}$ to $V(\mrm{Z}_{\Upsilon,K})$.
\end{cor}
\begin{proof}
    This follows from Theorem \ref{thm-can-ex-Harris-complex} and the Tannakian formalism.
\end{proof}
Note that the principal bundle $\ca{E}_{G^c,K,\bb{C}}$ and the vector bundle $\autoshdr{W}_{K,\bb{C}}$ constructed above are compatible with pulling back through the transition map $\pi_{K',K,\bb{C}}:\sh_{K',\bb{C}}\to \sh_{K,\bb{C}}$ for any $K'\sbst K$. Let $\ca{E}_{G^c,\bb{C}}:=\varprojlim_K \ca{E}_{G^c,K,\bb{C}}$ be the principal $G^c_\bb{C}$-bundle over $\sh_\bb{C}:=\sh(G,X)_\bb{C}$; there is an equivariant $G(\A)$-action on $\ca{E}_{G^c,\bb{C}}$ and $\autoshdr{W}_{\bb{C}}$ covering the right $G(\A)$-action on $\sh_{\bb{C}}$. \par
We can also define $\autoshdr{W}_{\bb{C}}^\can$ and $\ca{E}_{G^c,\bb{C}}^\can$ by pulling back the corresponding objects to $\sh_{\bb{C}}^\Sigma:=\varprojlim_K \sh_{K,\bb{C}}^\Sigma$ for a fixed $\Sigma$. Note that there might not be $G(\A)$-actions on those pullbacks in general.
\subsubsection{}\label{subsubsec-can-model-can-principal-bundle}
For mixed Shimura data, the definition of canonical models of the principal bundles is determined by the torus case and functoriality. \par
Let $(T,\ca{Y})$ be a pure Shimura datum of a torus $T$. There is a natural morphism $(T,\ca{Y})\to (T^c,\hbar(\ca{Y})^c)$. By Lemma \ref{lem-cusp-cm}, we can apply Milne and Harris-Zucker's theory of canonical models of principal bundles to $(T^c,\hbar(\ca{Y})^c)$. The principal $T^c_\bb{C}$-bundle $\ca{E}_{T^c,\bb{C}}$ over $\sh(T^c,\ca{Y}^c)_\bb{C}$ descends to a 
$T^c(\A)$-equivariant principal $T^c$-bundle $\ca{E}_{T^c}$ over $\sh(T^c,\ca{Y}^c)$, which is called the \emph{canonical model} of $\ca{E}_{T^c,\bb{C}}$; such a $\ca{E}_{T^c}$ is determined by the actions of $\sigma\in \aut (\bb{C}/E(T^c,\hbar(\ca{Y})^c))$ on the canonical point of the period torsor. (See \cite[III.]{Mil90} and \cite[1.2.A]{HZ01}.)
Define the canonical model of the principal bundle $\ca{E}_{T^c}(T,\ca{Y})_\bb{C}$ on $\sh(T,\ca{Y})_\bb{C}$ as the pullback of $\ca{E}_{T^c}$ along $\sh(T,\ca{Y})\to \sh(T^c,\hbar(\ca{Y})^c)$. We denote it by $\ca{E}_{T^c}(T,\ca{Y})$ but we will omit the bracket if the mixed Shimura datum is clear in the context.\par
\begin{prop}[{\cite[Prop. 1.2.4 and 1.2.8]{HZ01}}]\label{prop-can-mod-HZM}
Let $(Q,\ca{X})$ be a mixed Shimura datum. The principal bundle $\ca{E}_{Q^c,\bb{C}}$ on $\sh(Q,\ca{X})_\bb{C}$ admits a \emph{canonical model} $\ca{E}_{Q^c}(Q,\ca{X})$ on $\sh(Q,\ca{X})$, which is equipped with an equivariant $Q(\A)$-action. The canonical model $\ca{E}_{Q^c}(Q,\ca{X})$ is uniquely characterized by the following properties:\par
For any embedding $(T,\ca{Y})\hookrightarrow (Q,\ca{X})$ with $T$ a torus, the morphism $\ca{E}_{T^c,\bb{C}}\to \ca{E}_{Q^c,\bb{C}}$ descends to a $(T^c_{E}\to Q^c_{E})$-equivariant and $(T(\A)\to Q(\A))$-equivariant morphism $\ca{E}_{T^c}(T,\ca{Y})\to \ca{E}_{Q^c}(Q,\ca{X})$, where $E:=E(T,\ca{Y})$ is the reflex field of $(T,\ca{Y})$.\par
Let $f:(Q_1,\ca{X}_1)\to (Q_2,\ca{X}_2)$ be a morphism between mixed Shimura data. For any pair of open compact subgroups $K_1$ and $K_2$ of $Q_1(\A)$ and $Q_2(\A)$, respectively, such that $K_1\sbst K_2$, there is a morphism over $E(Q_1,\ca{X}_1)$,
$f: \ca{E}_{Q^c_1}(Q_1,\ca{X}_1)/K_1\to \ca{E}_{Q^c_2}(Q_2,\ca{X}_2)_{E(Q_1,\ca{X}_1)}/K_2$, 
descent from the morphism over $\bb{C}$ given by the functoriality.
\end{prop}
\begin{proof} Set $(Q^c,\ca{X}^c):=(Q,\ca{X})/Z_{ac}(Q)$.
Again, since $(Q^c,\ca{X}^c)$ satisfies the assumptions in \emph{loc. cit.} by Lemma \ref{lem-cusp-cm}, we can define the canonical model of the principal $Q^c_\bb{C}$-bundle on $\sh(Q,\ca{X})_\bb{C}$ to be the pullback to $\sh(Q,\ca{X})$ of the canonical model $\ca{E}_{Q^c}(Q^c,\ca{X}^c)$ of $\ca{E}_{Q^c}(Q^c,\ca{X}^c)_\bb{C}$ defined over $\sh(Q^c,\ca{X}^c)$. We denote this pullback by $\ca{E}_{Q^c}(Q,\ca{X})$. In fact, this pullback satisfies the condition in the proposition because we can verify the condition by applying \cite[Prop. 1.2.4]{HZ01} after assuming that $T$ is $\bb{Q}$-maximal and reducing to the case of $(T^c,\ca{Y}^c)\hookrightarrow (Q^c,\ca{X}^c)$, thanks to Lemma \ref{lem-tc-emb} and Lemma \ref{lem-hc-func}. The uniqueness follows from a standard argument by \cite[Lem. 11.6 and Lem. 11.7]{Pin89}, reducedness of $\sh(Q,\ca{X})$ and separatedness of $Q^c$.
\end{proof}
Define $\ca{E}_{Q^c,K}(Q,\ca{X}):=\ca{E}_{Q^c}(Q,\ca{X})/K$, and call it the canonical model of $\ca{E}_{Q^c,K,\bb{C}}(Q,\ca{X})$.
\subsubsection{}\label{subsubsec-comm-diag-reflex}
We now come back to the case where $(Q,\ca{X})=(G,X)$ is a usual Shimura datum.
\begin{cor}\label{cor-e-action}
The commutative diagram (\ref{eq-conj-torus-complex-alg}) descends to a commutative diagram of algebraic varieties over $E=E(G,X)$
\begin{equation}\label{eq-conj-torus-reflex}
    \begin{tikzcd}
    \mbf{E}_{K_\Phi,E}\times_{\spec E} \ca{E}_{P^c_\Phi,K_\Phi}\arrow[r]\arrow[d]&\ca{E}_{P^c_\Phi,K_\Phi}\arrow[d]\\
\mbf{E}_{K_\Phi,E}\times_{\spec E} \sh_{K_\Phi}\arrow[r]&\sh_{K_\Phi}.
    \end{tikzcd}
\end{equation}
\end{cor}
\begin{proof}
    This is again based on the argument in Lemma \ref{lem-u-action-alg}, \cite[Thm. 11.18]{Pin89} and Proposition \ref{prop-can-mod-HZM}. Note that the reflex field of $\sh_{K_\Phi}$ is equal to $E$ by \cite[Prop. 12.1]{Pin89}.
\end{proof}
\begin{lem}
The bundles $\ca{E}_{P^c_\Phi,K_\Phi,\bb{C}}$ and $\ca{E}_{P^c_\Phi,K_\Phi,\bb{C}}(\sigma)$, together with the isomorphism 
$$\mbf{p}^*_1((\mbf{p}_{1,*}\ca{E}_{P^c_\Phi,K_\Phi,\bb{C}})^{\mbf{E}_{K_\Phi}})\xrightarrow{\sim}\ca{E}_{P^c_\Phi,K_\Phi,\bb{C}},$$
all descend to $E(G,X)$.
\end{lem}
\begin{proof}
    By Proposition \ref{prop-can-mod-HZM} and Corollary \ref{cor-e-action} above, the morphism in the statement exists over $E(G,X)$. The rest of the lemma is clear.
\end{proof}
\begin{thm}[{\cite[Thm. 4.2 (iii)]{Har89}}]\label{thm-harris-can-ex-can-model}Let $(G,X)$ be a Shimura datum. Let $\Sigma$ be an admissible projective cone decomposition for $(G,X,K)$, where $K$ is neat open compact in $G(\A)$. Set $\sh^\Sigma_K:=\sh^\Sigma_K(G,X)$.\par For any $W\in \rep_\bb{Q}(G^c)$,
the canonical extension $\autoshdr{W}_{K,\bb{C}}^\can$ of $\autoshdr{W}_{K,\bb{C}}$ over $\sh^\Sigma(G,X)_\bb{C}$ descends to a locally free sheaf $\autoshdr{W}_K^\can$ over $\sh_K^\Sigma$. When $\Sigma$ is smooth, the log connection $\nabla^\can_\bb{C}$ also descends to a log connection $\nabla^\can$, and the residues of $\nabla^\can$ along irreducible components of the boundary divisor are nilpotent. \par
The principal $G^c$-bundle $\ca{E}^\can_{G^c,K,\bb{C}}$ uniquely descends to a principal $G^c$-bundle $\ca{E}^\can_{G^c,K}$ over $\sh_K^\Sigma$ extending $\ca{E}_{G^c,K}$.
\end{thm}
\begin{proof}
The proof can be found in \cite[p. 20]{Har89}.
\end{proof}
\begin{prop}\label{prop-functoriality-generic}
Let $f:(G_1,X_1)\to (G_2,X_2)$ be a morphism between Shimura data. Fix $K_1\sbst G_1(\A)$ and $K_2\sbst G_2(\A)$ such that $f(K_1)\sbst K_2$. Fix admissible projective cone decompositions $\Sigma_1$ and $\Sigma_2$ for $(G_1,X_1,K_1)$ and $(G_2,X_2,K_2)$, respectively, such that $\Sigma_1$ and $\Sigma_2$ are compatible (see \cite[Def. 1.18(3)]{Wu25}). Then there is a $(G_1^c\to G_2^c)$-equivariant morphism $f:\ca{E}_{G_1^c,K_1}^\can\to\ca{E}^\can_{G_2^c,K_2}$ over $E(G_1,X_1)$ covering the morphism $f_{\Sigma_1,\Sigma_2}:\sh_{K_1}^{\Sigma_1}(G_1,X_1)\to \sh_{K_2}^{\Sigma_2}(G_2,X_2)$ given by the functoriality between the toroidal compactifications and extending the one given by Proposition \ref{prop-can-mod-HZM}. 
\end{prop}
\begin{proof}Let $W\in \rep(G_2^c)$. Assume that $\Sigma_1$ is smooth and projective.
Applying Theorem \ref{thm-harris-can-ex-can-model} and Proposition \ref{prop-can-mod-HZM}, we see that the pullback $f^*_{\Sigma_1,\Sigma_2}\ca{E}_{G_2^c,K_2}^\can\times^{G_2^c}W$ is uniquely characterized by Deligne's existence theorem and extends $(\ca{E}_{G_1^c,K_1}\times^{G_1^c}G_2^c)\times^{G_2^c}W$. So $f^*_{\Sigma_1,\Sigma_2}\ca{E}_{G_2^c,K_2}^\can\times^{G_2^c}W\iso \ca{E}_{G_1^c,K_1}^\can\times^{G_1^c}W$ over $\bb{C}$, and this isomorphism descend to $E$ again by Deligne's existence theorem.\par
When $\Sigma_1$ is not smooth, choose a smooth (projective) refinement $\Sigma_1'$. We then have $f^*_{\Sigma_1',\Sigma_2}\ca{E}_{G_2^c,K_2}^\can\times^{G_2^c}W\iso \ca{E}_{G_1^c,K_1}^\can\times^{G_1^c}W$ by the paragraph above. We then take $f_{\Sigma_1',\Sigma_1,*}$ on both sides of the isomorphism. By the projection formula, we are reduced to showing $f_{\Sigma_1',\Sigma_1,*}\ca{O}_{\sh^{\Sigma_1'}_{K_1}}\iso \ca{O}_{\sh^{\Sigma_1}_{K_1}}$, and this follows from \cite[Ch. I]{KKMS73}.
\end{proof}
The following result is also recorded for later use.
\begin{prop}\label{prop-functoriality-generic-conti}
Let $(G_1,X_1,K_1,\Sigma_1)$ be as in Proposition \ref{prop-functoriality-generic} above. Let $\mbf{g}\in G_1(\bb{Q})$ and $g\in G_1(\A)$. Let $\Sigma^g$ be the cone decomposition such that $\Sigma_1$ is induced by $\Sigma^g$ via the morphism $[g]:\sh_{K_1}^{\Sigma_1}\to \sh_{gK_1g^{-1}}^{\Sigma^g}$ as in \cite[6.7(a)]{Pin89}. Then, \begin{enumerate}
\item There is a $G^c_1$-equivariant morphism $\ca{E}^\can_{G_1^c,K_1}\to\ca{E}^\can_{G_1^c,\lcj{K_1}{g}}$ extending $\ca{E}_{G_1^c,K_1}\to\ca{E}_{G_1^c,{gK_1g^{-1}}}$ given by Proposition \ref{prop-can-mod-HZM} and covering $[g]$. 
\item There is a $G^c_1\to \mbf{g}G^c\mbf{g}^{-1}$-equivariant morphism defined by $\ca{E}^\can_{G_1^c,K_1}\to \ca{E}^\can_{G_1^c,K_1}\times^{G_1^c}\mbf{g}G_1^c\mbf{g}^{-1}$.
\end{enumerate}
\end{prop}
\begin{proof}
The first part follows from exactly the same argument as in Proposition \ref{prop-functoriality-generic}. The second part is self-explanatory.
\end{proof}
\subsection{Induction for torsors}\label{subsec-ind-torsor}
In this subsection, we generalize the construction in \cite{Del79} and \cite{Mil88} to deal with the principal bundles on mixed Shimura varieties. 
Note that our exposition will also use the theorems on canonical models of principal bundles on mixed Shimura varieties as known facts, while these theorems were proved by reducing to the pure Shimura data case. Hence, recording such results is still necessary for later use.\par
\subsubsection{}\label{subsubsec-ind-recall}Fix a base scheme $\ca{O}$. Let $\Gamma$ and $\Gamma'$ be locally profinite groups. Suppose that $\Gamma'$ acts on a profinite set $\wat{\pi}$ continuously and transitively. Suppose that there is a continuous homomorphism $f: \Gamma\to \Gamma'$ such that $\Gamma$ is surjectively mapped to $\Gamma'_e$, the stabilizer in $\Gamma'$ of some element $e\in \wat{\pi}$, with a compact kernel.\par 
Consider the inverse system $\{S_{\ca{H}}\}_{\ca{H}}$ consisting of quasi-projective schemes ${S}_{\ca{H}}$ over $\ca{O}$ indexed by a cofinal collection of open compact subgroups $\ca{H}$ in $\Gamma$. Suppose that there is a continuous right $\Gamma$-action on $S:=\varprojlim_{\ca{H}}S_{\ca{H}}$. See \cite[2.7.1]{Del79}.\par
Then, by \cite[2.7.2]{Del79}, 
\begin{lem}[{\cite[Lem. 2.7.3]{Del79}}]\label{lem-deligne-ind}
If we fix $f:\Gamma\to\Gamma'$, $\Gamma'\to \mrm{Aut}\wat{\pi}$ and $e\in \wat{\pi}$ as above, then there is a functor $\mrm{Ind}_{\Gamma}^{\Gamma'}$ from the category of schemes $S$ as above to the category of $(S',c)$ consisting of $\ca{O}$-schemes $S'=\varprojlim_{K}S'_K$ with $\{S'_K\}_K$ an inverse system of quasi-projective $\ca{O}$-schemes indexed by open compact subgroups of $\Gamma'$, and continuous morphisms $c:S'\to \wat{\pi}$ equipped with continuous equivariant $\Gamma'$-actions. The functor is defined as 
$$S\mapsto \mrm{Ind}_\Gamma^{\Gamma'}S:=(S\times \Gamma')/\Gamma,$$
called an induction functor.

Moreover, when the kernel of $f:\Gamma\to \Gamma'_e\hookrightarrow\Gamma'$ is trivial, this is an equivalence of categories.
\end{lem}
\begin{proof}
The functor in the first paragraph makes sense because the kernel of $f$ is compact, and therefore, for each open compact $K\sbst \Gamma'$, $[S\times (\Gamma'/K)]/\Gamma$ is a finite disjoint union of a finite quotient of $S_{\ca{H}}$ for sufficiently small $\ca{H}$ (see \cite[4.1.2]{Wu25} for a computation). Moreover, if $f$ is injective, the inverse is given by $S'\mapsto S'_e:=S'\times_{c,\wat{\pi}}e$. 
\end{proof}
\subsubsection{}
Let $G$ and $G'$ be linear algebraic groups over $\ca{O}$ with a homomorphism $g:G\to G'$. Fix $f:\Gamma\to\Gamma'$, $\Gamma'\to \mrm{Aut}\wat{\pi}$ and $e\in \wat{\pi}$ as above.\par
We fix an extension $\Gamma^G$ of $\Gamma$ by $G$, that is, an extension $1\to G\to \Gamma^G\to \Gamma\to 1$ viewed as group functors over $\ca{O}$.
Assume further that there is an extension $\Gamma^{G'}$ of $\Gamma'$ by $G'$ that fits into the following commutative diagram
\begin{equation*}
  \begin{tikzcd}
  1\arrow[r]& G\arrow[r]\arrow[d,"g"]&\Gamma^G\arrow[r,"\pi_G"]\arrow[d,"\wdtd{g}"]&\Gamma\arrow[r]\arrow[d,"f"]&1\\
  1\arrow[r]& G'\arrow[r]&\Gamma^{G'}\arrow[r,"\pi_{G'}"]& \Gamma'\arrow[r]& 1.
  \end{tikzcd}  
\end{equation*}
For any subgroup $H\sbst \Gamma$ (resp. $H'\sbst \Gamma'$), denote by $\Gamma^G_H$ (resp. $\Gamma^{G'}_{H'}$) the pullback of $\Gamma^G$ (resp. $\Gamma^{G'}$) along $H\to \Gamma$ (resp. $H'\to \Gamma'$).\par  
Consider the category of $G$-torsors $\ca{E}\to S$ with continuous equivariant $(\Gamma^G\to \Gamma)$-actions:\par
The objects consist of inverse systems in the form of $\{\ca{E}_{\ca{H}}\to \ca{S}_{\ca{H}}\}_{\ca{H}}$ such that $\{S_{\ca{H}}\}_{\ca{H}}$ is equipped with a continuous right $\Gamma$-action. Furthermore, we require that 
\begin{enumerate}
\item Each $\ca{E}_{\ca{H}}$ is a $G$-torsor over $S_{\ca{H}}$, and the morphism $\ca{E}_{\ca{H}}\to S_{\ca{H}}$ is $(\Gamma^G\xrightarrow{\pi_G}\Gamma)$-equivariant. 
\item The restriction to $G$ of the action of $\Gamma^G_{\ca{H}}$ on $\ca{E}_{\ca{H}}$ induces the $G$-action on $\ca{E}_{\ca{H}}$ given by the torsor structure. 
\item For any normal open compact subgroup $\ca{H}'\sbst \ca{H}$, there is an isomorphism between $G$-torsors $\ca{E}_{\ca{H}'}/(\Gamma^G_{\ca{H}}/\Gamma^G_{\ca{H}'})\xrightarrow{\sim} \ca{E}_{\ca{H}}$ compatible with the isomorphism $S_{\ca{H}'}/(\ca{H}/\ca{H}')\xrightarrow{\sim}S_{\ca{H}}$ and determines the transition map $\ca{E}_{\ca{H}'}\to \ca{E}_{\ca{H}}$ in the inverse system.
\end{enumerate}
Write $\ca{E}:=\varprojlim_{\ca{H}}\ca{E}_{\ca{H}}$.
Define an induction functor
$$\mrm{Ind}_{\Gamma^G}^{\Gamma^{G'}}\ca{E}:=\ca{E}\times \Gamma^{G'}/\Gamma^G.$$
\begin{lem}\label{lem-deligne-ind-torsor}
If $\ker(\Gamma^{G}\times^{G}G'\to \Gamma^{G'})$ acts on $\varprojlim_\ca{H}\ca{E}_{\ca{H}}\to \varprojlim_{\ca{H}} S_{\ca{H}}$ freely and equivariantly, then the quotient $\ca{E}\times \Gamma^{G'}/\Gamma^G$ exists as a $G'$-torsor $\ca{E}'$ over $S'=\mrm{Ind}_{\Gamma}^{\Gamma'}S$ with continuous equivariant $(\Gamma^{G'}\to \Gamma')$-action.\par
If we further assume that $f$ is injective and $G=G'$, the functor $\mrm{Ind}_{\Gamma^G}^{\Gamma^{G'}}$ is an equivalence from the category of $G$-torsors $\ca{E}\to S$ with continuous equivariant $(\Gamma^G\to \Gamma)$-actions to the category of $(\ca{E}'\to S',c)$ consisting of $G$-torsors $\ca{E}'\to S'$ with continuous equivariant $(\Gamma^{G'}\to \Gamma')$-actions, and morphisms $c:S'\to \wat{\pi}$ equipped with continuous equivariant $\Gamma'$-actions.
\end{lem}
\begin{proof}
Let $\ca{H}''$ be an open compact subgroup of $\Gamma'$. Then $\ca{E}\times^{\Gamma^G} (\Gamma^{G'}/\Gamma^{G'}_{\ca{H}''})\iso \ca{E}\times^{\Gamma^G}(\Gamma^G\times^G G')\times^{(\Gamma^G\times^G G')}(\Gamma^{G'}/\Gamma^{G'}_{\ca{H}''})\iso (\ca{E}\times^G G')\times^{(\Gamma^G\times^G G')}(\Gamma^{G'}/\Gamma^{G'}_{\ca{H}''})$. As in \cite[4.1.2]{Wu25}, let $J:=f(\Gamma)\bss \Gamma'/\ca{H}''=\wdtd{g}(\Gamma^G)\bss \Gamma^{G'}/\Gamma^{G'}_{\ca{H}''}$ and choose a complete collection of representatives $j\in \Gamma'$ of this finite double coset.\par
We have $(\ca{E}\times^G G')\times^{(\Gamma^G\times^G G')}(\Gamma^{G'}/\Gamma^{G'}_{\ca{H}''})\iso$ 
$$\disju_{j\in J} (\ca{E}\times^G G')/\ker(\Gamma^G\times^G G'\to \Gamma^{G'}/\Gamma ^{G'}_{j\ca{H}''j^{-1}}).$$
The kernel in the last expression is isomorphic to
$$\ker(\Gamma\to \Gamma'/ {j\ca{H}''j^{-1}}).$$

Let $\ca{H}_0$ be any open compact subgroup of $\Gamma$ contained in $K:=\ker(\Gamma\to \Gamma'/j\ca{H}''j^{-1})$ as a (finite index) normal subgroup. The quotient of $\ca{E}_{\ca{H}_0}\times^G G'$ by the finite group $(\Gamma^G_{K}\times^G G')/(\Gamma^G_{\ca{H}_0}\times^G G')$ \emph{with free action} exists as a $G'$-torsor.\par
The second paragraph follows again from the construction of the inverse functor as $(\ca{E}'\to S',c)\mapsto (\ca{E}'_e:=\ca{E}'\times_{c,\wat{\pi}}e\to S_e')$.
\end{proof}
\subsubsection{}\label{subsubsec-ind-mixed-sh}Let $(G,X)$ be a Shimura datum. Choose a cusp label representative $\Phi$. Let $Y_\Phi$ be a connected normal subgroup of $Q_\Phi$ such that $P_\Phi\sbst Y_\Phi\sbst ZP_\Phi:=(Z_G\cdot P_\Phi)^\circ$. 
Then $Y_\Phi^c\hookrightarrow G^c$. 
%\textcolor{red}{cite Mao-Wu}
Let $G_0$ be either $Y_\Phi^c$ or $ G^c$. Then $G_0$ is stable under the conjugation action of $Y^\ad_\Phi(\bb{Q})^+$.\par 
Define (cf. \cite[p. 106]{Mil88} and \cite[4.1.4]{Wu25})
$${}^{G_0}\ag(Y_\Phi):=(G_0\times\frac{Y_\Phi(\A)}{Z^Y_{\Phi,ac}(\bb{Q})^{\overline{\ }}})*_{Y_{\Phi}(\bb{Q})_+}Y_{\Phi}^{\ad}(\bb{Q})^+,$$
and 
$${}^{G_0}\agsb(Q):=(G_0\times \frac{Y_{\Phi}(\bb{Q})_+^{\overline{\ }}}{Z^Y_{\Phi,ac}(\bb{Q})^{\overline{\ }}})*_{Y_{\Phi}(\bb{Q})_+}Y_{\Phi}^\ad(\bb{Q})^+$$
where $Y_\Phi^\ad:=Y_{\Phi}/(Z_G\cap Y_\Phi)$ and $Z_{\Phi,ac}^Y:=(Z_{G}\cap Y_\Phi)_{ac}$. The middle factors $Y_\Phi(\bb{Q})_+$ multiply to the left factors on the right.\par
Note that for a multiplicative group $Z$ over $\bb{Q}$, $Z_{ac}(\bb{Q})^{\overline{\ }}Z(\bb{Q})=Z(\bb{Q})^{\overline{\ }}$. Then there is an extension 
$$1\to G_0\to {}^{G_0}\ag(Y_\Phi)\to \ag(Y_\Phi)\to 1$$
for ${}^{G_0}\ag(Y_\Phi)$, and an extension
$$1\to G_0\to {}^{G_0}\agsb(Y_\Phi)\to\agsb(Y_\Phi)\to 1$$
for ${}^{G_0}\agsb(Y_\Phi)$.\par
We only treat the case where $G_0=Y_\Phi^c$ because the other follows immediately by a contraction product $\times^{Y^c_\Phi} G^c$. 
Then ${}^{Y_\Phi^c}\ag(Y_\Phi)$ acts on $\sh_{Y_\Phi}(\bb{C}):=\varprojlim_{K_\Phi^Y\sbst Y_\Phi(\A)\text{\ neat open compact}}Y_\Phi(\bb{Q})\bss Y_\Phi(\bb{Q})D_\Phi\times Y_\Phi^c(\bb{C})\times Y_\Phi(\A)/K_\Phi^Y$ by a right action
$$[(\mbf{g},g,\gamma^{-1})]\cdot [(x,\mbf{h},y)]\mapsto [(\gamma x\gamma^{-1},\gamma\mbf{h}\mbf{g}\gamma^{-1},\gamma yg\gamma^{-1})],$$
where $\mbf{g},\mbf{h}\in Y_\Phi^c(\bb{C})$, $g,y \in Y_\Phi(\A)$, $\gamma^{-1}\in Y^\ad_\Phi(\bb{Q})^+$ and $x\in Y_\Phi(\bb{Q})D_\Phi$. Note that $Y^\ad_\Phi(\bb{Q})^+$ acts on $Y_\Phi^c$ by a left conjugation, but we write it as a right action by writing $\gamma^{-1}$ instead of $\gamma$.\par
This action algebraizes and descends to an action of ${}^{Y_\Phi^c}\ag(Y_\Phi)$ on $\ca{E}_{Y_\Phi}:=\varprojlim_{K_\Phi^Y}\ca{E}_{Y_\Phi,K_\Phi^Y}$ since, by Proposition \ref{prop-can-mod-HZM}, the formation of canonical models of principal bundles satisfies functoriality, and is $Y_\Phi(\A)$- and $Y_\Phi^c$-equivariant.\par
One can check again on complex points that the stabilizer of ${}^{Y_\Phi^c}\ag(Y_\Phi)$ on a connected component $\ca{E}^+_{Y_\Phi,\overline{\bb{Q}}}:=\ca{E}_{Y_\Phi,\overline{\bb{Q}}}\times_{\sh_{Y_\Phi,\overline{\bb{Q}}}}\sh_{Y_\Phi,\overline{\bb{Q}}}^+$ is ${}^{Y_\Phi^c}\agsb(Y_\Phi)$. By Lemma \ref{lem-deligne-ind-torsor}, we have an ${}^{Y^c_\Phi}\ag(Y_\Phi)$-equivariant isomorphism
\begin{equation}\label{eq-deligne-ind-torsor}\ca{E}_{Y_\Phi,\overline{\bb{Q}}}\iso \ca{E}_{Y_\Phi,\overline{\bb{Q}}}^+\times ({}^{Y_\Phi^c}\ag(Y_\Phi))/({}^{Y_\Phi^c}\agsb(Y_\Phi)).\end{equation}

We care mainly about the following situation. \par 
Let $\pi: (G_0,X_0)\to (G,X)$ be a morphism between Shimura data. Suppose that $\ker(\pi: G_0\to G)$ is contained in the center of $G_0$, so $\pi:G^\der_0\to G^\der$ is an isogeny. \par
Let $\Phi_0$ be a cusp label representative of $\ca{CLR}(G_0,X_0)$ which maps to $\Phi\in\ca{CLR}(G,X)$. There is a morphism between mixed Shimura data $\pi: (P_{\Phi_0},D_{\Phi_0})\to (Y_\Phi,Y_\Phi(\bb{\bb{Q}})D_\Phi)$. By Proposition \ref{prop-can-mod-HZM}, there is a $(P_{\Phi_0}(\A)\to Y_\Phi(\A))$- and $(P_{\Phi_0}^c\to Y_\Phi^c)$-equivariant morphism between principal bundles 
$$\ca{E}_{P_{\Phi_0}}\to \ca{E}_{Y_\Phi}.$$
By checking on $\bb{C}$-points, this morphism is in fact equivariant under ${}^{P_{\Phi_0}^c}\ag(P_{\Phi_0})\to {}^{Y_\Phi^c}\ag(Y_\Phi)$.\par
Fix a connected component $\ca{E}^+_{P_{\Phi_0},\overline{\bb{Q}}}\sbst \ca{E}_{P_{\Phi_0},\overline{\bb{Q}}}$. It can be checked again on complex points that $\ca{E}^+_{P_{\Phi_0},\overline{\bb{Q}}}$ maps surjectively to a connected component $\ca{E}^+_{Y_\Phi,\overline{\bb{Q}}}\sbst \ca{E}_{Y_\Phi,\overline{\bb{Q}}}$.
Define
$$\ca{E}'_{Y_\Phi,\overline{\bb{Q}}}:=\ca{E}_{P_{\Phi_0},\overline{\bb{Q}}}^+\times ({}^{Y^c_\Phi}\ag(Y_\Phi))/({}^{P^c_{\Phi_0}}\agsb(P_{\Phi_0})).$$
Then the map $\ca{E}_{P_{\Phi_0},\overline{\bb{Q}}}\to \ca{E}_{Y_\Phi,\overline{\bb{Q}}}$ factors through a ${}^{Y^c_\Phi}\ag(Y_\Phi)$-equivariant morphism between $Y^c_\Phi$-torsors over $\sh_{Y_\Phi,\overline{\bb{Q}}}$ between $\ca{E}'_{Y_\Phi,\overline{\bb{Q}}}$ and $\ca{E}_{Y_\Phi,\overline{\bb{Q}}}$. So we have a ${}^{Y^c_\Phi}\ag(Y_\Phi)$-equivariant isomorphism
$$\ca{E}'_{Y_\Phi,\overline{\bb{Q}}}\iso\ca{E}_{Y_\Phi,\overline{\bb{Q}}}.$$

Since the ${}^{Y^c_\Phi}\ag(Y_\Phi)$-action on $\ca{E}_{Y_{\Phi}}$, the ${}^{P_{\Phi_0}^c}\ag(P_{\Phi_0})$-action on $\ca{E}_{P_{\Phi_0}}$ and the morphism $\ca{E}_{P_{\Phi_0}}\to \ca{E}_{Y_\Phi}$ are defined over $E(P_{\Phi_0},D_{\Phi_0})$, we see that
\begin{lem}\label{lem-check-group-quotient}
The morphism $\ca{E}_{P_{\Phi_0}}\to \ca{E}_{Y_\Phi}$ over $E(P_{\Phi_0},D_{\Phi_0})$ is $({}^{P_{\Phi_0}^c}\ag(P_{\Phi_0})\to{}^{Y^c_\Phi}\ag(Y_\Phi))$-equivariant. It factors through a quotient $(\ca{E}_{P_{\Phi_0}}\times^{P_{\Phi_0}^c}Y^c_\Phi)/\ker(({}^{P_{\Phi_0}}\ag(P_{\Phi_0}))\times^{P_{\Phi_0}^c}Y^c_\Phi\to{}^{Y^c_\Phi}\ag(Y_\Phi))$, which embeds into $\ca{E}_{Y_\Phi}$.
\end{lem}
\begin{proof}
    The first sentence follows from the paragraphs above. It suffices to check the second sentence over $\overline{\bb{Q}}$.\par
Since $$\ca{E}_{P_{\Phi_0},\overline{\bb{Q}}}\iso \ca{E}_{P_{\Phi_0},\overline{\bb{Q}}}^+\times ({}^{P_{\Phi_0}^c}\ag(P_{\Phi_0}))/({}^{P_{\Phi_0}^c}\agsb(P_{\Phi_0}))$$
and $$\ca{E}_{Y_\Phi,\overline{\bb{Q}}}\iso \ca{E}_{P_{\Phi_0},\overline{\bb{Q}}}^+\times ({}^{Y^c_\Phi}\ag(Y_\Phi))/({}^{P^c_{\Phi_0}}\agsb(P_{\Phi_0})),$$
the morphism $\ca{E}_{P_{\Phi_0}}\to \ca{E}_{Y_\Phi}$ is written as 
$$\ca{E}_{P_{\Phi_0},\overline{\bb{Q}}}^+\times ({}^{P_{\Phi_0}^c}\ag(P_{\Phi_0}))/({}^{P_{\Phi_0}^c}\agsb(P_{\Phi_0}))\to \ca{E}_{P_{\Phi_0},\overline{\bb{Q}}}^+\times ({}^{Y^c_\Phi}\ag(Y_\Phi))/({}^{P^c_{\Phi_0}}\agsb(P_{\Phi_0})).$$
The last expression factors through 
$$(\ca{E}_{P_{\Phi_0},\overline{\bb{Q}}}^+)\times ({}^{P_{\Phi_0}^c}\ag(P_{\Phi_0})\times^{P^c_{\Phi_0}}Y^c_\Phi)/({}^{P_{\Phi_0}^c}\agsb(P_{\Phi_0}))\iso \ca{E}_{P_{\Phi_0}}\times^{P^c_{\Phi_0}} Y^c_\Phi.$$
Then the kernel of the natural morphism from the last expression above to 
$$\ca{E}_{P_{\Phi_0},\overline{\bb{Q}}}^+\times ({}^{Y^c_\Phi}\ag(Y_\Phi))/({}^{P^c_{\Phi_0}}\agsb(P_{\Phi_0}))$$
is $\ker(({}^{P_{\Phi_0}}\ag(P_{\Phi_0}))\times^{P_{\Phi_0}^c}Y^c_\Phi\to{}^{Y^c_\Phi}\ag(Y_\Phi))$.
\end{proof}
%\subsubsection{Galois descent}\textcolor{red}{Galois invariant?}
\section{Canonical extensions on toroidal compactifications of integral canonical models}\label{sec-can-ex-hyper}
\subsection{Integral models of toroidal compactifications}\label{subsec-tor-ab-review}
Let us review the conventions in \cite[Sec. 4.2]{Wu25} in the hyperspecial case.\par
Fix a prime number $p> 0$. Let $(G_2,X_2)$ be an abelian-type Shimura datum. 
From now on, we assume that $G_{2,\bb{Q}_p}$ is (quasi-split and) unramified. We fix a smooth reductive model $G_{2,\zbkp}$ over $\zbkp$ such that $G_{2,\zbkp}(\bb{Z}_p)=K_{2,p}$ is a hyperspecial subgroup of $G_2(\bb{Q}_p)$. 
We make the following choices:
\begin{itemize} 
\item By \cite[Lem. 3.4.13]{Kis10}, there is a Hodge-type Shimura datum $(G_0,X_0)$ such that there is a central isogeny $G_0^\der\to G_2^\der$ and an isomorphism between adjoint Shimura data $(G_0^\ad,X_0^\ad)\iso (G_2^\ad,X_2^\ad)$, and such that $G_{0,\bb{Q}_p}$ is unramified. 
\item We choose intermediate Shimura data $(G,X_a)$ and $(G,X_b)$. We first choose $G$ such that $G^\der=G_2^\der$ and there is an embedding $\pi^a:G_2\to G$ and a map $\pi^b:G_0\to G$ with a finite kernel contained in $Z_{G^\der_0}$. The center of $G$ can be chosen to be quasi-split and unramified at $p$. Hence, $G_{\bb{Q}_p}$ is also unramified. Let $(G,X_a)$ (resp. $(G,X_b)$) be the Shimura datum determined by $(G_2,X_2)$ and $\pi^a:G_2\to G$ (resp. $(G_0,X_0)$ and $\pi^b:G_0\to G$). The choice of $G$ makes $G(\bb{R})$-homogeneous spaces $X_a$ and $X_b$ homeomorphic to $X^\ad_2\iso X^\ad_0$. See \cite[\S 1.4.4]{Wu25}.
\item We choose reductive models according to the last two items (see \cite[\S 4.2 (HS)]{Wu25}). Let $I_{G/G_0,K}:=\stb_{G(\bb{Q})}(X_0)\pi^b(G_0(\A))\bss G(\A)/K$. For the case we are considering, we can choose a complete collection of representatives of $I_{G/G_0,K}$ in $G(\Ap)$. For $\alpha\in I_{G/G_0,K}$, denote by $g_\alpha\in G(\Ap)$ a representative of $\alpha$. 
We can choose a smooth reductive model $G_{\bb{Z}_p}$ of $G$ over $\bb{Z}_p$ such that there is an embedding $G_{2,\bb{Z}_p}\hookrightarrow G_{\bb{Z}_p}$. Let $K_p:=G_{\bb{Z}_p}(\bb{Z}_p)$. 
We choose a smooth reductive model $G_{0,\bb{Z}_p}:=G_{0,\bb{Z}_p}^\alpha$ of $G_0$, which is independent of $\alpha$, such that there is a finite map $G_{0,\bb{Z}_p}\to G_{\bb{Z}_p}$ with kernel in the center of $G_{0,\bb{Z}_p}^\der$ and $K_{0,p}=G_{0,\bb{Z}_p}(\bb{Z}_p)$.
\end{itemize}
Set $E_0:=E(G_0,X_0)$, $E_2:=E(G_2,X_2)$ and $E':=E_0\cdot E_2$. We can and will choose $(G_0,X_0)$ such that the resulting $E'$ is unramified over $E_2$ at places over $p$.\par
Fix a place $v_2|p$ of $E_2$ that lifts to a place $v'|v_2$ of $E'$. Let $v|p$ be a place of $E_0$ such that $v'|v$. Define $\ca{O}:=\ca{O}_{E',(v)}=\ca{O}_{E'}\otimes_{\ca{O}_{E_0}}\ca{O}_{E_0,(v)}$ and $\ca{O}_2:=\ca{O}_{E_2,(v_2)}$.
Most of the constructions are done over $\ca{O}$ or $E'$ unless otherwise noted. Let $\ca{O}^{ur}:= \ca{O}_{E^{\prime,p}}\otimes_{\ca{O}_{E_0}}\ca{O}_{E_0,(v)}$, where $E^{\prime,p}$ is the union of all finite field extensions of $E'$ that are unramified at places of $E'$ over $p$.\par
We outline some key steps in the construction of toroidal compactifications of integral canonical models of abelian-type Shimura varieties. See \cite[Sec. 1.4, Sec. 4.2-4.3]{Wu25} for details.
\begin{enumerate}
    \item Choosing cone decompositions and levels. Let $\Sigma_2$ be an admissible cone decomposition of $(G_2,X_2,K_2:=K_{2,p}K^p_2)$. There are neat open compact subgroups $K_2^p\sbst G_2(\Ap)$, $K^p\sbst G(\Ap)$, and $K^{\alpha,p}_0\sbst G_0(\Ap)$ such that $K_2^p\sbst K^p$ via $\pi^a$ and $K_0^\alpha$ maps to $\lcj{K}{\alpha}:=\lcj{K}{g_\alpha}$ via $\pi^b$. Note that we use $\lcj{K}{\alpha}$ to denote left conjugation $g_\alpha K g_\alpha^{-1}$ for simplicity. Write $K_?=K_{?,p}K_?^p$, where $?={}^\alpha_0$, $\emptyset$ or ${}_2$. There are also admissible cone decompositions $\Sigma_2'$, $\Sigma$, and $\Sigma_0^\alpha$ for $(G_2,X_2,K_2)$, both $(G,X_a,K)$ and $(G,X_b,K)$, and $(G_0,X_0,K_0^\alpha)$, respectively. We can and will require that they are all projective, that $\Sigma_2'$ and $\Sigma$ are smooth, that $\Sigma_0^\alpha$ is induced by $\Sigma$, and that $\Sigma_2'$ is induced by $\Sigma$. (In fact, the smoothness assumption can be dropped in the construction of this paper.)\par
    From this point to the end of this paper, we write $\Sigma_2'$ as $\Sigma_2$ for simplicity.\par
    One can choose a Hodge embedding between Shimura data $\iota:(G_0,X_0,K_{0,p})\hookrightarrow (G^\ddag,X^\ddag,K_p^\ddag)$, where $(G^\ddag,X^\ddag,K_p^\ddag)$ is a Siegel Shimura datum $(GSp(V,\psi),X^\ddag)$ with $K_p^{\ddag}$ (resp. $K_{0,p}$) the stabilizer in $G^\ddag(\bb{Q}_p)$ (resp. $G_0(\bb{Q}_p)$) of a self-dual $\bb{Z}_p$-lattice $V_{\bb{Z}_p}\sbst V_{\bb{Q}_p}$. By \cite[Lem. 2.3.1 and Lem. 1.3.2]{Kis10} and the remark in \cite[Lem. 4.7]{KM15} and by Zarhin's trick, we can choose a self-dual $\bb{Z}$-lattice $V_{\bb{Z}}\sbst V$ such that there is a reductive $\zbkp$-model $G_{0,\zbkp}$ with a closed embedding $G_{0,\zbkp}\hookrightarrow \mrm{GL}(V_\zbkp)$ and such that $G_{0,\zbkp}$ is the pointwise stabilizer of a collection of tensors $\{s_{\lambda}\}_{\lambda\in \Lambda}\sbst (V_{\bb{Z}}\otimes\zbkp)^{\otimes}$, where $\Lambda$ is a finite index set.\par
    From $G_{0,\zbkp}$, we extend $G_{\bb{Z}_p}$ and $G_{2,\bb{Z}_p}$ to reductive $\zbkp$-models $G_{\zbkp}$ and $G_{2,\zbkp}$ such that $G_{0,\bb{Z}_p}\to G_{\bb{Z}_p}$ extends to $G_{0,\zbkp}\to G_{\zbkp}$ with finite kernel in the center of $G^\der_{0,\zbkp}$, and such that the closed embedding $G_{2,\bb{Z}_p}\hookrightarrow G_{\bb{Z}_p}$ also extends to $G_{2,\zbkp}\hookrightarrow G_\zbkp$.\par
    Choose neat open compact subgroups $K^{\ddag,\alpha,p}\sbst G^\ddag(\Ap)$ compatible with $K_0^{\alpha,p}$, and denote $K^{\ddag,\alpha}:=K_p^\ddag K^{\ddag,\alpha,p}$.
    \item Constructing the quotients. We construct the integral model $\ca{S}^\Sigma_K$ of the toroidal compactification $\sh_K^\Sigma(G,X_b)$ first. The compactification $\ca{S}^\Sigma_K$ is constructed as $\disju_{\alpha\in I_{G/G_0,K}} \ca{S}_{K^\alpha_0}^{\Sigma_0^\alpha}/\Delta_{\lcj{K}{\alpha}}(G_0,G)$, where $\Delta_{\lcj{K}{\alpha}}(G_0,G)=\ker(\ag(G_0)\to \ag(G)/\lcj{K}{\alpha})$. One can also construct the stratification on $\ca{S}^\Sigma_K$. For any $ZP$-cusp with cone $[ZP^b(\Phi,\sigma)]$ as in \cite[Def. 1.28]{Wu25}, one constructs the corresponding integral model of mixed Shimura variety $\ca{S}_{\K_\Phi}$ as a disjoint union 
    $$\disju_{\alpha\in I_{G/G_0,K}}\disju_{\pi^b(g^\alpha_0)\alpha\sim g^b}\disju_{[\sigma_0^\alpha]\in [\sigma]_{ZP}}\ca{S}_{K_0^\alpha}/\Delta_{\lcj{K}{g_0^\alpha g_\alpha}}(P_{0},ZP_\Phi),$$
    where $\Delta_{\lcj{K}{g_0^\alpha g_\alpha}}(P_{0},ZP_\Phi)=\ker(\ag(P_0)\to \ag(ZP_\Phi)/ZP_\Phi(\A)\cap \lcj{K}{g_0^\alpha g_\alpha})$. See Proposition \ref{prop-recall-constr-wu25} for detailed meanings of symbols. 
    \item Descent and pullback. Fixing $K_2^p$, one can make finer choices in step (1) so that $\pi^a:\sh_{K_2}^{\Sigma_2}\hookrightarrow \sh_K(G,X_a)$ is an open and closed embedding over $\spec E'$. From step (2), one can construct $\ca{S}^\Sigma_K(G,X_a)$ over $\ca{O}$ by descending from $\ca{S}^\Sigma_{K,\ca{O}^{ur}}$ by a descent datum whose generic fiber is that of $\sh_K^\Sigma(G,X_a)_{E^{\prime,ur}}$ from $E^{\prime,ur}$ to $E'$. We construct $\ca{S}^{\Sigma_2}_{K_2,\ca{O}}$ as the normalization in $\sh_{K_2}^{\Sigma_2}$ of $\ca{S}_K^\Sigma(G,X_a)$. 
    In particular, $\ca{S}_K(G,X_a)_{\ca{O}^{ur}}\iso \ca{S}_K(G,X_b)_{\ca{O}^{ur}}$ and $\ca{S}_K^\Sigma(G,X_a)_{\ca{O}^{ur}}\iso \ca{S}_K^\Sigma(G,X_b)_{\ca{O}^{ur}}$.
    Finally, we could show that the towers of integral models of mixed Shimura varieties in the boundary components of $\ca{S}_{K_2,\ca{O}}^{\Sigma_2}$ satisfy the extension property. So the descent datum of $\sh^{\Sigma_2}_{K_2,E'}$ associated with the faithfully flat base change $E'/E_2$ extends to that of $\ca{S}^{\Sigma_2}_{K_2,\ca{O}_{E',(v')}}$ associated with the faithfully flat base change $\ca{O}_{E',(v')}/\ca{O}_{E_2,(v_2)}$.
\end{enumerate}
Write $I_{G/G_0}:=I_{G/G_0,K}$, $\Delta^\alpha:=\Delta_{\lcj{K}{\alpha}}(G_0,G)$ and $\Delta^{g_0^\alpha}:=\Delta_{\lcj{K}{g_0^\alpha g_\alpha}}(P_0,G)$. 
\subsection{Canonical extensions on Hodge-type compactifications}\label{subsec-can-ext-hodge}
In this subsection, we mainly review the works of Lovering \cite{Lov17} and Madapusi \cite{Mad19} on the definition and existence of integral models of standard principal bundles and their (sub-)canonical extensions. An ingredient added to the theory here is that we prove an integral (and algebraic) version of Theorem \ref{thm-can-ex-Harris-complex}.
\subsubsection{}\label{subsubsec-can-ext-hodge-review}
Let $\ca{V}^\alpha$ be the first de Rham homology of the pullback of the universal abelian scheme over $\ca{S}_{K^{\ddag,\alpha}}$ to $\ca{S}_{K_0^\alpha}$. The scheme $\ca{S}_{K^{\ddag,\alpha}}$ denotes the Siegel moduli scheme defined by $(V_\bb{Z},\psi|_{V_\bb{Z}},K_p^\ddag,K^{\ddag,\alpha,p})$ as in, e.g., \cite[\S 2.1]{Wu25}. The tensors $s_\lambda$ canonically induce Betti tensors $s_{\lambda,B}^\alpha$ on Betti local systems, which in turn canonically induce parallel tensors $s_{\lambda,\dr}^{\alpha,an}\in\ca{V}^\otimes|_{\sh_{K_0^\alpha}(\bb{C})}$ by classical Riemann-Hilbert, and algebraize to tensors $s_{\lambda,\dr,\bb{C}}^\alpha$ in $\ca{V}^\otimes|_{\sh_{K_0^\alpha,\bb{C}}}$ by \cite{Del70}. Moreover, $s_{\lambda,\dr,\bb{C}}^\alpha$ descend to tensors $s_{\lambda,\dr}^\alpha$ in $\ca{V}^\otimes|_{\sh_{K_0^\alpha}}$ by \cite[Cor. 2.2.2]{Kis10}. By \cite[Cor. 2.3.9]{Kis10}, $s_{\lambda,\dr}^\alpha$ extend to tensors $\wdtd{s}_{\lambda,\dr}^\alpha$ in $\ca{V}^{\alpha,\otimes}$.\par 
Define a sheaf on $(\ca{S}_{K_0^\alpha})_{fppf}$
\begin{equation}
    \mathscr{E}_{G_{0,\zbkp},K_0^\alpha}:=\underline{\mrm{Isom}}_{\{s_\lambda\otimes 1\mapsto \wdtd{s}_{\lambda,\dr}^\alpha\}}(V_\zbkp\otimes_\zbkp \ca{O}_{\ca{S}_{K_0^\alpha}},\ca{V}^\alpha).
\end{equation}
\begin{thm}[{\cite[Prop. 4.5.6]{Lov17} and \cite[4.3.1 and Prop. 4.3.7]{Mad19}}]\label{thm-lov-mad}
For each $\alpha$, the following statements are true:
\begin{enumerate}
\item The fppf sheaf $\mathscr{E}_{G_{0,\zbkp},K_0^\alpha}$ is a $G_{0,\zbkp}$-torsor (in the {\'e}tale topology). It determines and is determined by a tensor functor
$$\omega_{G_{0,\zbkp},K_0^\alpha}:\rep(G_{0,\zbkp})\to \mrm{Vec}_{\ca{S}_{K_0^\alpha}}$$
from the category of $G_{0,\zbkp}$-representations to the category of vector bundles on $\ca{S}_{K_0^\alpha}$ with integrable connections. The torsor $\mathscr{E}_{G_{0,\zbkp},K_0^\alpha}$ is the integral canonical model of the torsor $\ca{E}_{G_0,K_0^\alpha}(G_0,X_0)$ as in \S\ref{subsubsec-can-model-can-principal-bundle} in the sense of \cite[\S 4.4.5]{Lov17}. (We will recall the definition of integral canonical models of principal bundles later in \S\ref{subsubsec-lov-can-model})
\item $\ca{V}^\alpha$ canonically extends to a vector bundle $\ca{V}^{\alpha,\can}$ with log integrable connection on $\ca{S}^{\Sigma_0^\alpha}_{K_0^\alpha}$. For each cusp label representative with cone $(\Phi_0^\alpha,\sigma_0^\alpha)$, there is a vector bundle with integrable connection $\ca{V}_{\Phi_0^\alpha}$ on $\ca{S}_{K_{\Phi_0^\alpha}}$ associated with the $1$-motive $\Q_{\Phi_0^\alpha}$ on it constructed by its covariant de Rham realization. Moreover, $\ca{V}_{\Phi_0^\alpha}$ admits an extension $\ca{V}_{\Phi_0^\alpha}(\sigma_0^\alpha)$ to $\ca{S}_{K_{\Phi_0^\alpha}}(\sigma_0^\alpha)$, which is the vector bundle with integrable log connection associated with the log $1$-motive extending $\Q_{\Phi_0^\alpha}$. The vector bundle $V_{\Phi_0^\alpha}$ with weight filtration and integrable log connection extends the vector bundle $\autoshdr{V}_{K_{\Phi_0^\alpha},\bb{C}}(\sigma_0^\alpha)$ together with its weight filtration and integrable log connection.\par
Over $$\cpl{\ca{S}_{K_0^\alpha}^{\Sigma_0^\alpha}}{\ca{Z}_{[(\Phi_0^\alpha,\sigma_0^\alpha)],K_0^\alpha}}\iso \cpl{\ca{S}_{K_{\Phi_0^\alpha}}(\sigma_0^\alpha)}{\ca{S}_{K_{\Phi_0^\alpha},\sigma_0^\alpha}},$$
the two vector bundles with additional structures, $\ca{V}^{\alpha,\can}$ and $\ca{V}_{\Phi_0^\alpha}(\sigma_0^\alpha)$, coincide. This compatibility is compatible with the analytic one in Theorem \ref{thm-can-ex-Harris-complex}.
\item The tensors $\wdtd{s}_{\lambda,\dr}^\alpha$ extend to $\wdtd{s}_{\lambda,\dr}^{\alpha,\can}\in \ca{V}^{\alpha,\can,\otimes}$ on $\ca{S}^{\Sigma_0^\alpha}_{K_0^\alpha}$. The torsor $\mathscr{E}_{G_{0,\zbkp},K_0^\alpha}$ extends to a $G_{0,\zbkp}$-torsor $\ca{E}_{G_{0,\zbkp},K_0^\alpha}^\can$ over $\ca{S}_{K_0^\alpha}^{\Sigma_0^\alpha}$ defined by
$$\mathscr{E}_{G_{0,\zbkp},K_0^\alpha}^\can:=\underline{\mrm{Isom}}_{\{s_\lambda\otimes 1\mapsto \wdtd{s}_{\lambda,\dr}^{\alpha,\can}\}}(V_\zbkp\otimes_\zbkp \ca{O}_{\ca{S}_{K_0^\alpha}^{\Sigma_0^\alpha}},\ca{V}^{\alpha,\can}).$$
\item For any representation $W_\zbkp\in \rep(G_{0,\zbkp})$, $\autoshdr{W}_{K_0^\alpha,\zbkp}:=\mathscr{E}_{G_{0,\zbkp},K_0^\alpha}\times^{G_{0,\zbkp}}W_\zbkp$ is a locally free sheaf with an integrable connection $\nabla_\zbkp$. This locally free sheaf extends to a locally free sheaf $\autoshdr{W}_{K_0^\alpha,\zbkp}^{\Sigma_0^\alpha}:=\mathscr{E}_{G_{0,\zbkp},K_0^\alpha}^{\Sigma_0^\alpha}\times^{G_{0,\zbkp}}W_\zbkp$ over $\ca{S}^{\Sigma_0^\alpha}_{K_0^\alpha}$ with extended integrable log connection. Pulling back the sheaf to $\ca{S}_{K_0^\alpha}^{\Sigma_0^{\alpha,\prime}}$ for any smooth refinement $\Sigma_0^{\alpha,\prime}$ of $\Sigma_0^\alpha$, the log connection has nilpotent residues along the boundary. 
%\item Recall that $\Sigma_0^\alpha$ is a refinement of some smooth cone decomposition $\Sigma_0^{\alpha,\prime}$ and $\Sigma_0^{\alpha,\prime}$ is induced by a smooth cone decomposition $\Sigma^{\ddag,\alpha,\prime}$ for $\ca{S}_{K^{\ddag,\alpha}}$. Under these choices, the Hodge filtration extends and the connection $\nabla_\zbkp$ extends to an integrable log connection $\nabla_\zbkp^\can$. 
\end{enumerate}
\end{thm}
Again, we will omit the cone decompositions and just write ``$\can$/$\mrm{tor}$'' for the superscripts in the symbols of those extensions, if the underlying toroidal compactifications are clear.
\begin{proof}All statements, except for the last item, can be directly found in the references mentioned above. We note that the extension of vector bundle described in \cite[2.1.21]{Mad19} coincides with that in (\ref{eq-ext-W-boundary}). For the last item, note that the assertions hold over $\sh_{K_0^\alpha}^{\Sigma_0^\alpha}\cup \ca{S}_{K_0^\alpha}$ of codimension $\geq 2$; see the summary in \S\ref{subs-can-ex-summary}. Now, all but the nilpotence of residues follow from \S\ref{subs-can-ex-summary} and the first three items. For the nilpotence of residues, since the irreducible components of the boundary divisor are all flat over $\ca{O}_{E_0,(v_0)}$, it suffices to check over $\sh_{K_0^\alpha}^{\Sigma_0^\alpha}$, and this follows from \S\ref{subs-can-ex-summary} again.
\end{proof}
On each $\sh_{K_{\Phi_0^\alpha}}:=\sh_{K_{\Phi_0^\alpha}}(P_{\Phi_0^\alpha},D_{\Phi_0^\alpha})$, denote by $\mbf{Q}_{\Phi_0^\alpha}:=\Q_{\Phi_0^\alpha}|_{\sh_{K_0^\alpha}(P_{\Phi_0^\alpha},D_{\Phi_0^\alpha})}$ the pullback of $\Q_{\Phi_0^\alpha}$ in the theorem above to $\sh_{K_{\Phi_0^\alpha}}$. To be compatible with the notation in the last section, the pullback of $\ca{V}_{\Phi_0^\alpha}$ (resp. $\ca{V}_{\Phi_0^\alpha}(\sigma_0^\alpha)$) to $\sh_{K_{\Phi_0^\alpha}}$ (resp. $\sh_{K_{\Phi_0^\alpha}}(\sigma_0^\alpha)$) is denoted by $\autoshdr{V}_{K_{\Phi_0^\alpha}}$ (resp. $\autoshdr{V}_{K_{\Phi_0^\alpha}}(\sigma_0^\alpha)$).\par 
The same argument as at the beginning of \S\ref{subsubsec-can-ext-hodge-review} works for $\sh_{K_{\Phi_0^\alpha}}$. More precisely, by \cite[Sec. 3.1]{Mad19}, the tensors $s_\lambda$ induce tensors $s_{\lambda,\Phi_0^\alpha,B}$ in $\autoshbetti{V}_{K_{\Phi_0^\alpha}}^{an,\otimes}$, tensors $s_{\lambda,\Phi_0^\alpha,\dr}^{an}$ in $\autoshdr{V}^{an,\otimes}_{K_{\Phi_0^\alpha},\bb{C}}$, and tensors $s_{\lambda,\Phi_0^\alpha,\dr,\bb{C}}$ in $\autoshdr{V}^{\otimes}_{K_{\Phi_0^\alpha},\bb{C}}$. Tensors $s_{\lambda,\Phi_0^\alpha,\dr,\bb{C}}$ also descend to tensors $s_{\lambda,\Phi_0^\alpha,\dr}$ in $\autoshdr{V}^\otimes_{K_{\Phi_0^\alpha}}$ by \cite[Prop. 3.1.2]{Mad19}. \par
By \cite[Prop. 3.1.6]{Mad19}, the two tensors $s_{\lambda,\dr}^\alpha$ and $s_{\lambda,\Phi_0^\alpha,\dr}$ coincide on the open dense strata $U$ of $\cpl{\sh_{K_0^\alpha}^{\Sigma_0^\alpha}}{\mbf{Z}_{[(\Phi_0^\alpha,\sigma_0^\alpha)],K_0^\alpha}}\iso \cpl{\sh_{K_{\Phi_0^\alpha}}(\sigma_0^\alpha)}{\sh_{K_{\Phi_0^\alpha},\sigma_0^\alpha}}$ under the canonical isomorphism between $\autoshdr{V}_{K}|_U$ and $\autoshdr{V}_{K_{\Phi_0^\alpha}}|_U$.
\subsubsection{Tensors and principal bundles at the boundary}\label{subsubsec-tensors-bd}
\begin{lem}\label{lem-E-action-bd-v}
There is a natural $\mbf{E}_{K_{\Phi_0^\alpha}}$-action on $\ca{V}_{\Phi_0^\alpha}$ extending the one in Corollary \ref{cor-e-action} taking the value at $V$.
\end{lem}
\begin{proof}
Suppose that $\Phi_0^\alpha$ maps to a cusp label representative $\Phi^\ddag\in \ca{CLR}(G^\ddag,X^\ddag)$. The scheme $\ca{S}_{K_{\Phi^\ddag}}\to \overline{\ca{S}}_{K_{\Phi^\ddag}}$ is an $\mbf{E}_{K_{\Phi^\ddag}}$-torsor and there is an $(\mbf{E}_{K_{\Phi_0^\alpha}}\to \mbf{E}_{K_{\Phi^\ddag}})$-equivariant map $\ca{S}_{K_{\Phi_0^\alpha}}\to \ca{S}_{K_{\Phi^\ddag}}$. The $1$-motive $\Q_{\Phi_0^\alpha}$ is pulled back from the universal family $\Q_{\Phi^\ddag}$ on $\ca{S}_{K_{\Phi^\ddag}}$.\par 
There is an action of $\mbf{E}_{K_{\Phi^\ddag}}$ on the (covariant) de Rham realization $\mbf{H}_{\dr}(\Q_{\Phi^\ddag})$ of $\Q_{\Phi^\ddag}$. Let us explain this. Fix any $\zbkp$-algebra $R$ and write $S:=\spec R$. Let $x_1, x_2\in \ca{S}_{K_{\Phi^\ddag}}(S)$ mapping to the same point $\overline{x}\in \overline{\ca{S}}_{K_{\Phi^\ddag}}(S)$. Then $e\cdot x_1=x_2$ for some $e\in\mbf{E}_{K_{\Phi^\ddag}}(S)$. 
For $i=1,2$, let $\Q_{i}$ be the pullback of $\Q_{\Phi^\ddag}$ via $x_i$. Since $x_1$ and $x_2$ map to the same point $\overline{x}$, by the moduli interpretation of $\overline{S}_{K_{\Phi^\ddag}}$, we have canonical identifications $\G^\natural_{\Q_1}=\G^\natural_{\Q_2}$ and $\G^{\natural,\vee}_{\Q_1}=\G^{\natural,\vee}_{\Q_2}$ (see \cite[\S 2.2.1]{Wu25} for conventions).
Then, from the very construction of $\mbf{H}_{\dr}(\Q_{\Phi^\ddag})$, we obtain a canonical identification $\mbf{H}_\dr(\Q_1)=\mbf{H}_{\dr}(\Q_2)$. Indeed, $\mbf{H}_{\dr}(\Q_i)$ is the Lie algebra of the pushout
\begin{equation}
    \begin{tikzcd}
    0\arrow[r]&\lie^\vee_{A^\vee_{\Q_i}}\arrow[r]\arrow[d]&E(\G^\natural_{\Q_i})\arrow[r]& \G^\natural_{\Q_i}\arrow[r]&0,\\
    &\lie^\vee_{\G^{\natural,\vee}_{\Q_i}}
    \end{tikzcd}
\end{equation}
which depends only on $\G^\natural_{\Q_i}$ and $\G^{\natural,\vee}_{\Q_i}.$ See \cite[\S 2.2]{Ber05} or \cite[\S 1.4]{BVS98}. Here, $E(-)$ denotes universal vector extensions. Therefore, there is an $\mbf{E}_{K_{\Phi_0^\alpha}}$-action on the pullback.\par
Let us now check its compatibility with Corollary \ref{cor-e-action}. It suffices to check when $S=\spec \bb{C}$. Suppose that $e$ lifts to $u\in U_{\Phi^\ddag}(\bb{C})$. 
The action of $u$ on $\autoshbetti{V_\bb{C}}^{an}_{K_{\Phi^\ddag}}$, as explained above Lemma \ref{lem-u-action-complex}, is given by the following diagram
\begin{equation}\label{eq-betti-action-u}
    \begin{tikzcd}
    V_\bb{C}\arrow[r,"\alpha_{x_1}"]\arrow[d,"u"]& \autoshbetti{V_\bb{C}}^{an}_{K_{\Phi^\ddag}}|_{x_1}\arrow[d,"\iso"]\\
    V_\bb{C}\arrow[r,"\alpha_{x_2}"]& \autoshbetti{V_\bb{C}}_{K_{\Phi^\ddag}}|_{x_2}.
    \end{tikzcd}
\end{equation}
Here, we fix an identification of vector spaces $\alpha_{x_1}:V_\bb{C}\to \autoshbetti{V_\bb{C}}^{an}_{K_{\Phi^\ddag}}|_{x_1}$, which gives rise to the mixed Hodge structure at $x_1$. The right vertical arrow is a canonical identification following from the same reason as the last paragraph that the Betti realization depends only on $\overline{x}$. See \cite[Lem. 10.1.3.2 (c)]{Del74}. 
The compatibility follows from the Betti-de Rham comparison \cite[10.1.8]{Del74}.  
\end{proof}
\begin{lem}\label{lem-hz-int-v}
 We have isomorphisms
    \begin{equation}\label{eq-hz-int-v}
    \ca{V}_{\Phi_0^\alpha}\iso \mbf{p}_1^*(\mbf{p}_{1,*}\ca{V}_{\Phi_0^\alpha})^{\mbf{E}_{K_{\Phi_0^\alpha}}}
    \end{equation}
    and 
    \begin{equation}\label{eq-hz-int-ext-v}
\ca{V}_{\Phi_0^\alpha}(\sigma_0^\alpha)\iso \mbf{p}_1(\sigma_0^\alpha)^*(\mbf{p}_{1,*}\ca{V}_{\Phi_0^\alpha})^{\mbf{E}_{K_{\Phi_0^\alpha}}}.
    \end{equation}
\end{lem}
\begin{proof}
In fact, $\ca{V}_{\Phi_0^\alpha}$ descends to a vector bundle $\overline{\ca{V}}_{\Phi_0^\alpha}$ on $\overline{\ca{S}}_{K_{\Phi_0^\alpha}}$: Indeed, 
this statement follows from the fpqc descent of quasi-coherent modules and \cite[\href{https://stacks.math.columbia.edu/tag/05B2}{Lem. 05B2}]{stacks-project}. The descent datum is given by Lemma \ref{lem-E-action-bd-v}.\par
Now $\mbf{p}^*_1\overline{\ca{V}}_{\Phi_0^\alpha}=\ca{O}_{\ca{S}_{K_{\Phi_0^\alpha}}}\otimes_{\ca{O}_{\overline{\ca{S}}_{K_{\Phi_0^\alpha}}}}\overline{\ca{V}}_{\Phi_0^\alpha}=\ca{V}_{\Phi_0^\alpha}$. 
By projection formula, one has $(\mbf{p}_{1,*}\ca{V}_{\Phi_0^\alpha})^{\mbf{E}_{K_{\Phi_0^\alpha}}}=(\mbf{p}_*\ca{O}_{\ca{S}_{K_{\Phi_0^\alpha}}})^{\mbf{E}_{K_{\Phi_0^\alpha}}}\otimes_{\ca{O}_{\overline{\ca{S}}_{K_{\Phi_0^\alpha}}}}\overline{\ca{V}}_{\Phi_0^\alpha}=\overline{\ca{V}}_{\Phi_0^\alpha}$. We then deduce (\ref{eq-hz-int-v}).
If we know this, the second assertion is true because the two sides of (\ref{eq-hz-int-ext-v}) are isomorphic over a locus of codimension $\geq 2$.\par

\end{proof}
\begin{lem}\label{lem-tensor-mixsh-gen}
\begin{enumerate}
\item The tensors $s_{\lambda,\Phi_0^\alpha,\dr}$ extend uniquely to tensors $\wdtd{s}_{\lambda,\Phi_0^\alpha,\dr}$ in $\ca{V}_{\Phi_0^\alpha}^\otimes$ and tensors $\wdtd{s}_{\lambda,\Phi_0^\alpha,\dr}(\sigma_0^\alpha)$ in $\ca{V}_{\Phi_0^\alpha}(\sigma_0^\alpha)^\otimes$.
\item The tensors $\wdtd{s}_{\lambda,\Phi_0^\alpha,\dr}$ and $\wdtd{s}_{\lambda,\Phi_0^\alpha,\dr}(\sigma_0^\alpha)$ are $\mbf{E}_{K_{\Phi_0^\alpha}}$-invariant. 
\end{enumerate}
\end{lem}
\begin{proof}
The plan of the proof is the following: We assume that $\sigma_0^\alpha$ is a smooth cone first. We then show both of the assertions under this assumption. Finally, we drop this assumption.\par 
The uniqueness is clear, and it is enough to construct tensors $\wdtd{s}_{\lambda,\Phi_0^\alpha,\dr}(\sigma_0^\alpha)$.\par
By Lemma \ref{lem-hz-int-v}, $\ca{V}_{\Phi_0^\alpha}(\sigma_0^\alpha)$ is the pullback to $\ca{S}_{K_{\Phi_0^\alpha}}(\sigma_0^\alpha)$ of a vector bundle $\overline{\ca{V}}_{\Phi_0^\alpha}$ on $\overline{\ca{S}}_{K_{\Phi_0^\alpha}}$. Let $\{U_\omega\}_\omega$ be a Zariski cover of $\overline{\ca{S}}_{K_{\Phi_0^\alpha}}$ such that $\overline{\ca{V}}_{\Phi_0^\alpha}$ is trivial over $U_\omega$. Denote by $\{V_\omega:=U_\omega\times_{\overline{\ca{S}}_{K_{\Phi_0^\alpha}}}\ca{S}_{K_{\Phi_0^\alpha}}(\sigma_0^\alpha)\}$ a Zariski cover of $\ca{S}_{K_{\Phi_0^\alpha}}(\sigma_0^\alpha)$.
Let $\{D_i\}_{i=1}^r$ be the irreducible components of the boundary of $\ca{S}_{K_{\Phi_0^\alpha}}(\sigma_0^\alpha)$. Write $V_\omega:=\spec R_\omega$. Write the ideal sheaf of the pullback of $D_i$ to $V_\omega$ by $I_i=(s_i)$. Let $R_i:=\cpl{R_\omega[s_j^{-1}]_{j\neq i}}{I_i}$. 
The following diagram commutes:
\begin{equation*}
    \begin{tikzcd}
    \sh_{K_0^\alpha}\arrow[d]&{W_{i,\omega}^\circ:=\spec\cpl{R_\omega[p^{-1},s_j^{-1}]_{j\neq i}}{I_i}[s_i^{-1}]}\arrow[l]\arrow[dd]\arrow[r]&W^\circ_\omega:=\spec R_\omega[p^{-1},s_j^{-1}]_{j=1}^r\arrow[r]&\sh_{K_{\Phi_0^\alpha}}\arrow[d]\\
    \sh_{K_0^\alpha}^{\Sigma_0^\alpha}\arrow[d]&&&\sh_{K_{\Phi_0^\alpha}}(\sigma_0^\alpha)\arrow[d]\\
    \ca{S}_{K_0^\alpha}^{\Sigma_0^\alpha}&{W_{i,\omega}:=\spec\cpl{R_\omega[s_j^{-1}]_{j\neq i}}{I_i}}\arrow[l]\arrow[rr]&&\ca{S}_{K_{\Phi_0^\alpha}}(\sigma_0^\alpha).
    \end{tikzcd}
\end{equation*}
By \cite[Prop. 3.1.6]{Mad19}, the pullback of the tensors $s_{\lambda,\Phi_0^\alpha,\dr}$ to $W^\circ_\omega$ and to $W_{i,\omega}^\circ$ is identical to the pullback of $s^\alpha_{\lambda,\dr}$ to $W_{i,\omega}^\circ$. By the left-half of the diagram above and by Theorem \ref{thm-lov-mad}(3), the pullback of $s^\alpha_{\lambda,\dr}$ to $W^\circ_{i,\omega}$ is identical to the pullback of $\wdtd{s}_{\lambda,\dr}^{\alpha,\can}$ to $W^\circ_{i,\omega}$.\par
We now have a tensor $\wdtd{s}_{\lambda,\dr}^{\alpha,\can}$ in $V^{\otimes}|_{W_{i,\omega}}$ and a tensor $s_{\lambda,\Phi_0^\alpha,\dr}$ in $V^\otimes|_{W_\omega^\circ}$ for each $\lambda$ such that they coincide on $W^\circ_{i,\omega}$. Since $\cpl{R_\omega[s_j^{-1}]}{I_i}\cap R_\omega[p^{-1},s^{-1}_j]^r_{j=1}=R_\omega[s_j^{-1}]_{j\neq i}$, we have a tensor $\wdtd{s}_{\lambda,\omega,i}$ in $V^\otimes(R[s_j^{-1}]_{j\neq i})$ for each $i$. Moreover, since the tensor $\wdtd{s}_{\lambda,\omega,i}$ extends $s_{\lambda,\Phi_0^\alpha,\dr}$ on $W_\omega^\circ$, there is a tensor $\wdtd{s}_{\lambda,\omega}$ in $V^\otimes(\cap_{i=1}^r R_i)=V^\otimes(R_\omega)$ extending $\wdtd{s}_{\lambda,\omega,i}$. Finally, the tensors $\wdtd{s}_{\lambda,\omega}$ on varying $V_\omega$ glue to a tensor $\wdtd{s}_{\lambda,\Phi_0^\alpha,\dr}(\sigma_0^\alpha)$ in $\ca{V}_{\Phi_0^\alpha}(\sigma_0^\alpha)^\otimes$ since they all extend the tensor $s_{\lambda,\Phi_0^\alpha,\dr}$. Now we have (1) under the assumption.\par
To check (2), it suffices to check the generic fiber $\sh_{K_{\Phi_0^\alpha}}$, and pass to $\bb{C}$-points. Under the Betti-de Rham comparison, it suffices to check that $s_{\lambda,\Phi_0^\alpha,B}$ is $\mbf{E}_{K_{\Phi_0^\alpha}}(\bb{C})$-invariant, and this is the consequence of the construction in \S\ref{subsubsec-E-action-mix-sh-complex}, the fact that $s_\lambda$ is $U_\Phi(\bb{C})$-invariant and the action of $u\in U_\Phi(\bb{C})$ explained in (\ref{eq-betti-action-u}). The part (2) follows.\par
We now consider the case in general. We choose a smooth projective refinement $\Sigma_0^{\alpha,\prime}$ of $\Sigma_0^\alpha$ and choose any $\tau_0^\alpha$ in the refined cone decomposition that is contained in $|\sigma_0^\alpha|$. The argument above gives us $\mbf{E}_{K_{\Phi_0^\alpha}}$-invariant tensors $\wdtd{s}_{\lambda,\Phi_0^\alpha,\dr}$ in $\ca{V}_{\Phi_0^\alpha}^\otimes$. Then these tensors descend to $\overline{\ca{V}}_{\Phi_0^\alpha}^{\otimes}$. We then pull them back to $\ca{S}_{K_{\Phi_0^\alpha}}(\sigma_0^\alpha)$ to get desired tensors with desired properties.
\end{proof}
\begin{definition}\label{def-torsors-bd}
Define fppf sheaves 
\begin{equation}
    \mathscr{E}_{G_{0,\zbkp},K_{\Phi_0^\alpha}}:=\ul{\mrm{Isom}}_{\{s_\lambda\otimes 1\mapsto \wdtd{s}_{\lambda,\Phi_0^\alpha,\dr}\}}(V_\zbkp\otimes_\zbkp \ca{O}_{\ca{S}_{K_{\Phi_0^\alpha}}},\ca{V}_{\Phi_0^\alpha});
\end{equation}
and
\begin{equation}
 \mathscr{E}_{G_{0,\zbkp},K_{\Phi_0^\alpha}}(\sigma_0^\alpha):=\ul{\mrm{Isom}}_{\{s_\lambda\otimes 1\mapsto \wdtd{s}_{\lambda,\Phi_0^\alpha,\dr}(\sigma_0^\alpha)\}}(V_\zbkp\otimes_\zbkp \ca{O}_{\ca{S}_{K_{\Phi_0^\alpha}}(\sigma_0^\alpha)},\ca{V}_{\Phi_0^\alpha}(\sigma_0^\alpha)).
\end{equation}
\end{definition}
Let us summarize our main results in the Hodge-type case.
\begin{prop}\label{prop-torsor-hz}
The following statements are true: 
\begin{enumerate}
    \item The sheaves $\mathscr{E}_{G_{0,\zbkp},K_{\Phi_0^\alpha}}$ and $\mathscr{E}_{G_{0,\zbkp},K_{\Phi_0^\alpha}}(\sigma_0^\alpha)$ are ({\'e}tale) $G_{0,\zbkp}$-torsors.
    \item There are natural $\mbf{E}_{K_{\Phi_0^\alpha}}$-actions on both $\mathscr{E}_{G_{0,\zbkp},K_{\Phi_0^\alpha}}$ and $\mathscr{E}_{G_{0,\zbkp},K_{\Phi_0^\alpha}}(\sigma_0^\alpha)$ extending the one in Corollary \ref{cor-e-action}.
    \item We have isomorphisms
    \begin{equation}\label{eq-hz-int}
    \mathscr{E}_{G_{0,\zbkp},K_{\Phi_0^\alpha}}\iso \mbf{p}_1^*(\mbf{p}_{1,*}\mathscr{E}_{G_{0,\zbkp},K_{\Phi_0^\alpha}})^{\mbf{E}_{K_{\Phi_0^\alpha}}}
    \end{equation}
    and 
    \begin{equation}\label{eq-hz-int-ext}
    \mathscr{E}_{G_{0,\zbkp},K_{\Phi_0^\alpha}}(\sigma_0^\alpha)\iso \mbf{p}_1(\sigma_0^\alpha)^*(\mbf{p}_{1,*}\mathscr{E}_{G_{0,\zbkp},K_{\Phi_0^\alpha}})^{\mbf{E}_{K_{\Phi_0^\alpha}}}
    \end{equation}
    \item Over $$\cpl{\ca{S}_{K_0^\alpha}^{\Sigma_0^\alpha}}{\ca{Z}_{[(\Phi_0^\alpha,\sigma_0^\alpha)],K_0^\alpha}}\iso \cpl{\ca{S}_{K_{\Phi_0^\alpha}}(\sigma_0^\alpha)}{\ca{S}_{K_{\Phi_0^\alpha},\sigma_0^\alpha}},$$
    the torsors $\mathscr{E}_{G_{0,Z_{(p)}},K_0^\alpha}^\can$ and $\mathscr{E}_{G_{0,\zbkp},K_{\Phi_0^\alpha}}(\sigma_0^\alpha)$ coincide.
\end{enumerate}
\end{prop}
\begin{proof}By Theorem \ref{thm-lov-mad} (2) and (3), and by the proof of Lemma \ref{lem-tensor-mixsh-gen}, we obtain Part (4). Since $G_{0,\zbkp}$ is smooth and $\mathscr{E}_{G_{0,\zbkp},K_{\Phi_0^\alpha}}(\sigma_0^\alpha)$ fpqc locally admits a section and is a torsor by Part (4), we know that $\mathscr{E}_{G_{0,\zbkp},K_{\Phi_0^\alpha}}(\sigma_0^\alpha)$ and its restriction $\mathscr{E}_{G_{0,\zbkp},K_{\Phi_0^\alpha}}$ are representable and are {\'e}tale torsors.
%(cf. the proof of \cite[Prop. 4.5.6]{Lov17}). 
Parts (2) and (3) are derived directly from Lemma \ref{lem-hz-int-v} and Lemma \ref{lem-tensor-mixsh-gen}.
\end{proof}
\begin{rk}
In \cite[Prop. 5.6]{Lan16c}, Lan showed that the principal bundles under $M_\mu$ (whose definition will be given in \S\ref{subsec-auto-vb}) satisfy a property similar to Proposition \ref{prop-torsor-hz} in the PEL-type case.
\end{rk}
\subsubsection{Integral canonical models of principal bundles in the sense of Lovering}\label{subsubsec-lov-can-model}
This appendix supplements \S\ref{subsec-can-ext-hodge}. In what follows, we recall Lovering’s main theorem of integral canonical models of principal bundles. We will only briefly recall the definition since it is not heavily used here.\par
\begin{definition}[{Lovering; see \cite[\S 4.4]{Lov17}}]\label{def-lov-can-model}
Let $(G,X)$ be a Shimura datum of abelian type. Let $R:=\ca{O}_{E,(v)}$ be the localization at a place $v|p$ of $E:=E(G,X)$.
Let $G_{\zbkp}$ be a connected reductive model of $G$ such that $G_{\zbkp}(\bb{Z}_p)=K_p$.
Let $\{\ca{S}_{K}\}_{K^p}$ with $K=K_pK^p$ be the inverse system of smooth integral models over $R$ defined in \cite{Kis10} and \cite{KM15} when $p=2$.\par 
Let $G^c_\zbkp$ be the cuspidal quotient of $G_\zbkp$. Let $\{\mathscr{E}_{K}\}_{K^p}$ be an inverse limit of $G^c_\zbkp$-bundles on $\{\ca{S}_{K}\}_{K^p}$ compatible with $G(\Ap)$-action. We say that $\{\mathscr{E}_K\}_{K^p}$ is an integral canonical model of $\ca{E}_{K_p}=\{\ca{E}_{G^c,K}(G,x)\}_{K^p}$ if the generic fiber is $\ca{E}_{K_p}$ and if, for any closed point $x\in\ca{S}_{K}(\kappa)$ in the special fiber and a lifting $\wdtd{x}\in \sh_K=\ca{S}_K[1/p](L)$ of $x$ for some $L$ finite extension of $E_v$ and $\mrm{Frac}W(\kappa)$, the following two tensor functors coincide:
\begin{itemize}
    \item For any $G^c_{\zbkp}$-representation $V$, one assigns an $\ca{O}_L$-lattice by specializing $\mathscr{E}_{K}\times^{G^c_\zbkp}V$ to the $\ca{O}_L$-point of $\ca{S}_K$ determined by $(\wdtd{x},x)$.
    \item For any $G^c_\zbkp$-representation $V$, one assigns an $\ca{O}_L$-lattice by specializing to $\ca{O}_L$ the $\mathfrak{S}$-module associated with the (crystalline) Galois representation at $\wdtd{x}$ obtained from the $\mathfrak{M}$-functor in \cite[\S 1.2]{Kis10}.
\end{itemize}
\end{definition}
Note that the two lattices above are compared under a canonical de Rham structure of $\ca{E}_K$ (cf. \cite[4.4.4]{Lov17}). One can use \cite[Thm. 5.3.1]{DLLZ} to get such canonical de Rham structures for general Shimura varieties.
\begin{thm}[{\cite{Lov17}}]\label{thm-lovering-ab}
Let $(G_2,X_2)$ be an abelian-type Shimura datum. Assume that $G_2$ is quasi-split and unramified at $p$. With the conventions above, set $K_{2,p}=G_{2,\zbkp}(\bb{Z}_p)$ and $K_2=K_{2,p}K^p_2$. Then: 
\begin{enumerate}
\item The principal bundle $\ca{E}_{G_2^c,K_2}$ over $\sh_{K_2}$ admits an integral canonical model $\mathscr{E}_{G_{2,\zbkp}^c,K_2}$ over the integral canonical model $\ca{S}_{K_2}$ of $\sh_{K_2}$. Moreover, the integrable connection on $\ca{E}_{G^c_2,K_2}$ extends to $\mathscr{E}_{G^c_{2,\zbkp},K_2}$. The construction is compatible with prime-to-$p$ Hecke action of $G_2(\Ap)$.
\item Let $f:(G_1,X_1)\to (G_2,X_2)$ be a morphism from another abelian-type Shimura datum $(G_1,X_1)$ to $(G_2,X_2)$, such that $G_1$ is also quasi-split and unramified at $p$. Choose a smooth reductive model $G_{1,\zbkp}$ of $G_1$ such that the image of $K_{1,p}:=G_{1,\zbkp}(\bb{Z}_p)$ under $f$ is contained in $K_{2,p}$ and $K_1^p\sbst K_2^p$. Then $f^*\mathscr{E}_{G^c_{2,\zbkp},K_2}\iso \mathscr{E}_{G^c_{1,\zbkp},K_1}\times^{G_{1,\zbkp}^c}G_{2,\zbkp}^c.$ The pullback of the connection for $\mathscr{E}_{G_{2,\zbkp}^c,K_2}$ under $f$ also coincides with the connection for $\mathscr{E}_{G^c_{1,\zbkp},K_1}\times^{G_{1,\zbkp}^c}G_{2,\zbkp}^c.$ 
\end{enumerate}
\end{thm}
\begin{proof}
    The first statement is the main theorem \cite[Thm. 4.7.1]{Lov17}, and the extension of integrable connections is checked in Section 4.7 of \emph{loc. cit.} Note that the assumption in \cite{Lov17} that $Z_G^\circ$ splits over a CM field extension is not necessary, because one can define the integral canonical model of the principal bundle on an integral model $\ca{S}_K(G,X)$ as the pullback of that on $\ca{S}_{K^c}(G^c,X^c)$ and $G^c$ satisfies the assumption by Lemma \ref{lem-cusp-cm}.\par 
    The isomorphism in the second statement follows from \cite[Prop. 4.4.8]{Lov17}, and it is checked in \cite[4.4.4 and Prop. 4.7.4]{Lov17} that the source and the target of $f:(G_1,X_1)\to (G_2,X_2)$ both have all the crystalline points and compatible de Rham structures. The compatibility of connections can be checked over generic fiber over complex points.
\end{proof}
\subsection{Canonical extensions on abelian-type compactifications}\label{subsec-can-ext-ab}
\subsubsection{Statement of main theorem} 
Our aim is to show:
\begin{thm}\label{thm-extend-main-theorem}
There is a principal $G_{2,\zbkp}^c$-bundle $\mathscr{E}_{G_{2,\zbkp}^c,K_2}^\can$ on $\ca{S}^{\Sigma_2}_{K_2}$ which extends the canonical extension $\ca{E}_{G_{2,\zbkp}^c,K_2}^\can$ on $\sh_{K_2}^{\Sigma_2}$ (cf. Theorem \ref{thm-harris-can-ex-can-model}), and the integral canonical model of standard principal bundle $\mathscr{E}_{G_{2,\zbkp}^c,K_2}$ on $\ca{S}_{K_2}$.\par
The exact tensor functor corresponding to $\mathscr{E}_{G_{2,\zbkp}^c,K_2}^\can$,
$$\omega_{G_{2,\zbkp}^c,K_2}^\can: \rep(G_{2,\zbkp}^c)\lra \mrm{Vec}_{\ca{S}_{K_2}^{\Sigma_2}},$$
is from the category of $G_{2,\zbkp}^c$-representations over $\zbkp$ to the category of vector bundles on $\ca{S}_{K_2}^{\Sigma_2}$ with integrable log connections.
\end{thm}
The theorem will be proved in this subsection.

\subsubsection{}
Recall that we have $\ca{S}_{K}(G,X_a)_{\ca{O}^{ur}}\iso \ca{S}_K(G,X_b)_{\ca{O}^{ur}}$ and $\ca{S}^{\Sigma}_{K}(G,X_a)_{\ca{O}^{ur}}\iso \ca{S}^{\Sigma}_K(G,X_b)_{\ca{O}^{ur}}$. Set $\ca{S}_{K,\ca{O}^{ur}}:=\ca{S}_K(G,X_b)_{\ca{O}^{ur}}$ and $\ca{S}_{K,\ca{O}^{ur}}^\Sigma:=\ca{S}^{\Sigma}_K(G,X_b)_{\ca{O}^{ur}}.$\par
Recall that $\ca{S}_K^\Sigma=\disju_{\alpha\in I_{G/G_0}} \ca{S}_{K_0^\alpha}^{\Sigma_0^\alpha}/\Delta^\alpha$ is a scheme over $\ca{O}$. We denote the map from $\ca{S}^{\Sigma_0^\alpha}_{K_0^\alpha}$ to $\ca{S}_K^\Sigma$ by $\pi_\alpha: \ca{S}_{K_0^\alpha}^{\Sigma_0^\alpha}\to \ca{S}_K^\Sigma$.\par
Define $\Delta^\alpha_{G^c}:=\ker ({}^{G_0}\ag(G_0)\times^{G_0}G^c\to {}^{G^c}\ag(G)/\lcj{K}{g_\alpha}).$
\begin{lem}\label{lem-delta-g}
The group $\Delta^\alpha_{G^c}$ is profinite, and its projection to $\Delta^\alpha$ is an isomorphism.
\end{lem}
\begin{proof}
There is a commutative diagram
\begin{equation*}
    \begin{tikzcd}
    1\arrow[r]& G^c\arrow[r]\arrow[d,equal]&{{}^{G_0}\ag(G_0)\times^{G_0}G^c}\arrow[r]\arrow[d]&\ag(G_0)\arrow[r]\arrow[d]&1\\
    1\arrow[r]& G^c\arrow[r]&{}^{G^c}\ag(G)/\lcj{K}{g_\alpha}\arrow[r]&\ag(G)/\lcj{K}{g_\alpha}\arrow[r]&1
    \end{tikzcd}
\end{equation*}
where the horizontal arrows are exact. Then this follows from diagram chasing and \cite[Lem. 4.15]{Wu25}.
\end{proof}
\begin{construction}\label{const-delta-G-act-generic}\upshape
We construct a natural action of $\Delta^\alpha_{G^c}$ on $\ca{E}_{G^c,K_0^\alpha}^\can:=\ca{E}^\can_{G_0,K_0^\alpha}\times^{G_0}G^c$. 
Let $(\mbf{g},g,\gamma^{-1})\in \Delta^\alpha_{G^c}$. Then, by construction, there is an element $t\in G(\bb{Q})_+$ such that $(t,t,1)\sim(1,1,\gamma^{-1})$ in ${}^{G^c}\ag(G)$ and such that $(\mbf{g},g,\gamma^{-1})\sim (\mbf{g}t,gt,1)\in(1,\lcj{K}{g_\alpha}Z(\bb{Q})^{\overline{\ }}/Z(\bb{Q})^{\overline{\ }},1)$ in ${}^{G^c}\ag(G)$. This implies that $\mbf{g}t=1$ and $gt\in \lcj{K}{g_\alpha}Z(\bb{Q})^{\overline{\ }}/Z(\bb{Q})^{\overline{\ }}$.\par
By Proposition \ref{prop-functoriality-generic} and Proposition \ref{prop-functoriality-generic-conti}, we have an action
\begin{equation*}
\begin{split}
&\ca{E}^\can_{G^c,K_0^\alpha}\xrightarrow{\cdot \mbf{g}}\ca{E}^\can_{G^c,K_0^\alpha}\times^{G^c}\mbf{g}^{-1}G^c\mbf{g}\xrightarrow{\cdot g} \ca{E}^\can_{G^c,g^{-1}K_0^\alpha g}\times^{G^c}(\mbf{g}^{-1}G^c\mbf{g})\xrightarrow{\cdot \gamma^{-1}}\\ 
&\ca{E}^\can_{G^c,\gamma g^{-1}K_0^\alpha g \gamma^{-1}}\times^{G^c}\gamma G^c \gamma^{-1}\times^{\gamma G^c\gamma^{-1}} (\gamma \mbf{g}^{-1} G^c \mbf{g}\gamma^{-1})=\ca{E}^\can_{G^c,K_0^\alpha}\times^{G^c}G^c=\ca{E}^\can_{G^c,K_0^\alpha}.
\end{split}
\end{equation*}
The last line is computed using the first paragraph.\par
For an element $(\mbf{g},g,\gamma^{-1})\sim (1,1,1)$ in ${}^{G^c}\ag(G_0)$, we check over elements $(x,\mbf{g}_0,g_0)\in G_0(\bb{Q})\bss X\times G^c(\bb{C})\times G_0(\A)/K_0^\alpha= \ca{E}_{G^c,K_0^\alpha}(\bb{C})$: The action
\begin{equation*}
\begin{split}
&\ca{E}_{G^c,K_0^\alpha}(\bb{C})\xrightarrow{\cdot \mbf{g}}\ca{E}_{G^c,K_0^\alpha}(\bb{C})\times^{G^c}\mbf{g}^{-1}G^c\mbf{g}\xrightarrow{\cdot g} \ca{E}_{G^c,g^{-1}K_0^\alpha g}\times^{G^c}(\mbf{g}^{-1}G^c\mbf{g})(\bb{C})\xrightarrow{\cdot \gamma^{-1}}\\ 
&\ca{E}_{G^c,\gamma g^{-1}K_0^\alpha g \gamma^{-1}}\times^{G^c}\gamma G^c \gamma^{-1}\times^{\gamma G^c\gamma^{-1}} (\gamma \mbf{g}^{-1} G^c \mbf{g}\gamma^{-1})(\bb{C})=\ca{E}_{G^c,K_0^\alpha}\times^{G^c}G^c(\bb{C})=\ca{E}_{G^c,K_0^\alpha}(\bb{C})
\end{split}
\end{equation*}
 is computed as
\begin{equation*}
\begin{split}
& (x,\mbf{g}_0,g_0) \xrightarrow{\cdot \mbf{g}} (x,\mbf{g}_0\mbf{g},g_0)\xrightarrow{\cdot g} (x,\mbf{g}_0\mbf{g},g_0 g)\xrightarrow{\cdot \gamma^{-1}}\\ 
& (\gamma x\gamma^{-1},\gamma \mbf{g}_0\mbf{g}\gamma^{-1},\gamma g_0g\gamma^{-1})= (\gamma x,\gamma\mbf{g}_0,\gamma g_0)=(x,\mbf{g}_0,g_0).
\end{split}
\end{equation*}
The last line is due to the facts that $\gamma$ is an element in $G_0(\bb{Q})_+$ and is equal to $g$, and that the image of $\gamma$ in $G^c(\bb{Q})$ is equal to $\mbf{g}$. These facts are obtained from the equivalence $(\mbf{g},g,\gamma^{-1})\sim (1,1,1)$.
So the map is an identity map. It is also an identity map on $\ca{E}^\can_{G^c,K_0^\alpha}$ by density. We then have obtained a well-defined action of $\Delta^\alpha_{G^c}$ on $\ca{E}^\can_{G^c,K_0^\alpha}$.
\hfill$\square$
\end{construction}
\begin{prop}\label{prop-action-delta}
There is an action of $\Delta^\alpha_{G^c}$ on $\mathscr{E}^\can_{G_\zbkp^c,K_0^\alpha}:=\mathscr{E}_{G_{0,\zbkp},K_0^\alpha}^\can\times^{G_{0,\zbkp}}G_\zbkp^c$ extending the action on $\ca{E}_{G^c,K_0^\alpha}:=\ca{E}_{G_0,K_0^\alpha}\times^{G_0}G^c$ described in Construction \ref{const-delta-G-act-generic}. In particular, the action of $\Delta^\alpha_{G^c}$ factors through a finite group $\Delta^\alpha_{G^c}/K_0^\alpha$ since it is so on $\ca{E}_{G^c_\zbkp,K_0^\alpha}$.
\end{prop}
\begin{proof}
It suffices to construct the actions on $\mathscr{E}_{G^c,K_0^\alpha}$ over $\ca{S}_{K_0^\alpha}$, and on $\ca{E}^\can_{G^c,K_0^\alpha}$ on $\sh_{K_0^\alpha}^{\Sigma_0^\alpha}$ that coincide on $\ca{E}_{G^c,K_0^\alpha}$. 
By \cite[Lem. 2.2.6]{Kis10}, we can adjust an element in $G_0(\bb{Q})_+$ to $(\mbf{g},g,\gamma^{-1})$ so that the $p$-component $g_p$ of $g$ is in $K_{0,p}^\alpha=G_{0,\zbkp}(\bb{Z}_p)$. Write $g=g_pg^p$, where $g^p\in G_0(\Ap)$. Then there is an action of $g^p$ on $\varprojlim_{K_0^{\alpha,p}}\mathscr{E}_{G^c,K_0^\alpha}$ by Theorem \ref{thm-lovering-ab} (1), and the right action of $g_p\in K_{0,p}^\alpha$ is trivial. Then the action of $(\mbf{g},g,\gamma^{-1})$ on $\mathscr{E}_{G^c,K_0^\alpha}$ follows from Theorem \ref{thm-lovering-ab} (1) and (2). Now we can conclude by combining this with Construction \ref{const-delta-G-act-generic}. By Lemma \ref{lem-delta-g} and Construction \ref{const-delta-G-act-generic}, the action of $\Delta^\alpha_{G^c}$ factors through $\Delta^\alpha_{G^c}/K_0^\alpha$ and the latter group is finite.
\end{proof}
\begin{rk}
We must consider the contracted product 
$$\mathscr{E}_{G_{0,\zbkp}^c,K_0^\alpha}\times^{G_{0,\zbkp}^c} G^c_\zbkp$$ first to let the action of $\Delta^\alpha_G$ extend on the level of torsors. In other words, there is no $\Delta^\alpha$ or $\Delta^\alpha_{G^c}$-action on $\mathscr{E}_{G_{0,\zbkp}^c,K_0^\alpha}$. One way to see it is that one cannot make sense of ``$\mbf{g}t$'' in $G_0(\A)$ for $\mbf{g}\in G_0(\A)$ but $t\in G(\bb{Q})_+$. Another way to see this is that if there were such an action, there would be a principal $G_{0,\zbkp}$-bundle on $\ca{S}_{K}(G,X_b)$, and we have a contradiction with the theory in characteristic zero. 
\end{rk}
\begin{construction}\upshape
By Proposition \ref{prop-action-delta}, the action of $\Delta^\alpha$ extends to a $\Delta^\alpha_{G^c}$-action on $\mathscr{E}^\can_{G_{\zbkp}^c,K_0^\alpha}$ covering the $\Delta^\alpha$-action on $\ca{S}_{K_0^\alpha}^{\Sigma_0^\alpha}$. Since $\pi_\alpha$ is finite and $\Delta^\alpha_{G^c}$ acts on $\mathscr{E}^\can_{G_{\zbkp}^c,K_0^\alpha}$ through a finite group $\Delta^\alpha_{G^c}/K_0^\alpha$ (the finiteness follows from Lemma \ref{lem-delta-g}),
let $\mathscr{E}^\can_{G_\zbkp^c,g_\alpha,K}:=\mathscr{E}^\can_{G^c_\zbkp,K_0^\alpha}/\Delta^\alpha_{G^c}$. This is a scheme over $\ca{S}_{g_\alpha}:=\ca{S}_{K_0^\alpha}^{\Sigma_0^\alpha}/\Delta^\alpha$. Taking the disjoint union over $I_{G/G_0}$, we define $\mathscr{E}^\can_{G^c_\zbkp,K}:=\disju_{\alpha\in I_{G/G_0}} \mathscr{E}^\can_{G^c_\zbkp,g_\alpha,K}$.
\hfill$\square$\par
\begin{lem}\label{lem-generic-compatible}
Over $\sh_K^\Sigma(G,X_b)_{E'}$, the restriction $\mathscr{E}_{G^c_\zbkp,K}^\can|_{\sh^{\Sigma}_{K,E'}}$ is the canonical extension $\ca{E}^\can_{G^c,K}$ defined in Theorem \ref{thm-harris-can-ex-can-model}.
\end{lem}
\begin{proof}
By Proposition \ref{prop-functoriality-generic}, we have $\ca{E}^\can_{G_0,K_0^\alpha}\times^{G_0}G^c\iso \pi^{b,*}\ca{E}_{G^c,K}^\can$. Then, since $\Delta^\alpha_{G^c}$ acts on $\ca{E}^\can_{G^c,K}$ trivially by Lemma \ref{lem-check-group-quotient}, $\ca{E}^\can_{G_0,K_0^\alpha}\times^{G_0}G^c/\Delta^\alpha_{G^c}\iso \pi^{b,*}\ca{E}^\can_{G^c,K}/\Delta^\alpha_{G^c}\iso \ca{E}^\can_{G^c,K}|_{\sh_{g_\alpha}}$, where $\sh_{g_\alpha}:=\sh^{\Sigma_0^\alpha}_{K_0^\alpha,E'}/\Delta^\alpha$.
\end{proof}
Now the key point is to see that
\end{construction}
\begin{thm}\label{thm-loc-free-quo}
The scheme $\mathscr{E}^\can_{G_\zbkp^c,g_\alpha,K}$ (resp. $\mathscr{E}^\can_{G^c_\zbkp,K}$) is a $G_\zbkp^c$-torsor over $\ca{S}_{g_\alpha}$ (resp. $\ca{S}^\Sigma_K$). 
\end{thm}
\subsubsection{Proof of Theorem \ref{thm-loc-free-quo}}\label{subsubsec-proof-key-thm}
We show Theorem \ref{thm-loc-free-quo} in the whole \S\ref{subsubsec-proof-key-thm}. \par
Choose any $\Phi\in \ca{CLR}(G,X_b)$. Let $[ZP^b(\Phi,\sigma)]$ be a $ZP$-cusp with a cone $\sigma$ defined in \cite[Prop. 1.33]{Wu25}. In \cite[Const. 4.20]{Wu25}, we constructed a normal flat scheme $\ca{Z}_{[ZP^b(\Phi,\sigma)],K}$ whose generic fiber is $\disju_{[(\Phi',\sigma')]\sim_{ZP}[(\Phi,\sigma)]} \mrm{Z}_{[(\Phi',\sigma')],K}$, a disjoint union of strata defined by cusp labels with cones that are equivalent to $[(\Phi,\sigma)]$ under a weaker equivalence relation (see \cite[Def. 1.28]{Wu25}). 
Let $Q_0:=\pi^{b,-1}(Q_\Phi)$ be an admissible $\bb{Q}$-parabolic subgroup. Let $P_0$ be the canonical subgroup associated with $Q_0$ (see \S\ref{subsec-f-u-f} (8)).
\begin{prop}[{\cite{Wu25}}]\label{prop-recall-constr-wu25}
Let $ZP_\Phi$ be the identity component of the group generated by $Z_G$ and $P_\Phi$. Then
\begin{enumerate}
\item There is an integral model $\ca{S}_{\wdtd{K}_\Phi}$ of the mixed Shimura variety $\sh_{\wdtd{K}_\Phi}(ZP_\Phi,ZP_\Phi(\bb{Q})D_\Phi)$ associated with the mixed Shimura datum $(ZP_\Phi,ZP_\Phi(\bb{Q})D_\Phi)$ and the level $\wdtd{K}_\Phi:=ZP_\Phi(\A)\cap g_\Phi K g_\Phi^{-1}$.
\item The tower $$\sh_{\wdtd{K}_\Phi}(ZP_\Phi,ZP_\Phi(\bb{Q})D_\Phi)\to \sh_{\overline{\wdtd{K}}_\Phi}(\overline{ZP}_\Phi,ZP_\Phi(\bb{Q})\overline{D}_\Phi)\to \sh_{\wdtd{K}_{\Phi,h}}(ZP_{\Phi,h},ZP_\Phi(\bb{Q})D_{\Phi,h})$$ as in \S\ref{subsec-f-u-f} (7) has an integral model
$$\ca{S}_{\wdtd{K}_\Phi}\to \overline{\ca{S}}_{\wdtd{K}_\Phi}\to \ca{S}_{\wdtd{K}_{\Phi,h}}, $$
where the first map is a torsor under a split torus $\mbf{E}_{\wdtd{K}_\Phi}$, and the second map is a torsor under an abelian scheme over $\ca{S}_{\wdtd{K}_{\Phi,h}}$. It makes sense to define the affine torus embedding $\ca{S}_{\wdtd{K}_\Phi}(\sigma)$ and its $\sigma$-stratum $\ca{S}_{\wdtd{K}_\Phi,\sigma}$ associated with $\sigma\in \Sigma^+(\Phi)$. 
\item There is an isomorphism
    \begin{equation}\label{eq-zp-strata-iso-at-completion}
   \cpl{\ca{S}^\Sigma_K}{ \ca{Z}_{[ZP^b({\Phi},\sigma)],K}} \iso \cpl{\ca{S}_{\wdtd{K}_\Phi}(\sigma)}{\ca{S}_{\wdtd{K}_\Phi,\sigma}}.
\end{equation}
It is computed as follows:\footnote{We apologize to the readers that there is an error in the current version of \cite[Prop. 4.32]{Wu25} on arXiv. A corrected version has been updated at \href{https://peihang-wu.github.io/}{https://peihang-wu.github.io/}.}
\begin{equation}\label{eq-exp-computation}
\begin{split}
& \cpl{\ca{S}^\Sigma_K}{ \ca{Z}_{[ZP^b({\Phi},\sigma)],K}}\\
&\iso \disju_{\alpha\in I_{G/G_0}}\disju_{\pi^b(g_0^\alpha)\alpha\sim g^b}\disju_{[\sigma_0^\alpha]\in [\sigma]_{ZP}}(\Delta_{\lcj{K}{\alpha}}(G_0,G) \cpl{\ca{S}_{K_0^\alpha}^{\Sigma_0^\alpha}}{\ca{Z}_{[(\Phi^\alpha_0,\sigma^\alpha_0)],K_0^\alpha}})/\Delta_{\lcj{K}{\alpha}}(G_0,G)\\
&\iso\disju_{\alpha\in I_{G/G_0}}\disju_{\pi^b(g^\alpha_0)\alpha\sim g^b}\disju_{[\sigma_0^\alpha]\in [\sigma]_{ZP}} (\Delta_{\lcj{K}{\alpha}}(G_0,G)\cpl{\ca{S}_{K_{\Phi_0^\alpha}}(\sigma_0^\alpha)}{\ca{S}_{K_{\Phi_0^\alpha},\sigma_0^\alpha}})/\Delta_{\lcj{K}{\alpha}}(G_0,G)\\
&\iso\disju_{\alpha\in I_{G/G_0}}\disju_{\pi^b(g^\alpha_0)\alpha\sim g^b}\disju_{[\sigma_0^\alpha]\in [\sigma]_{ZP}}\cpl{\ca{S}_{K_0^\alpha}(\sigma_0^\alpha)}{\ca{S}_{K_{\Phi_0^\alpha},\sigma_0^\alpha}}/\Delta_{\lcj{K}{g_0^\alpha g_\alpha}}(P_{0},ZP_\Phi)\\
&\iso \cpl{\ca{S}_{\wdtd{K}_\Phi}(\sigma)}{\ca{S}_{\wdtd{K}_\Phi,\sigma}}.
\end{split}
\end{equation}
The cusp label representatives $\Phi_0^\alpha$ are the tuples $(Q_{0},X^+_0,g_0^\alpha)$ for fixed $Q_0$ and $X^+_0$. The images of $[(\Phi_0^\alpha,\sigma_0^\alpha)]$ in $\cusp_K(G,X_b,\Sigma)$ are $\sim_{ZP}$ to $[(\Phi,\sigma)]$. The above disjoint union runs over (a subset of all) such $[(\Phi_0^\alpha,\sigma_0^\alpha)]$.
\end{enumerate}
\end{prop}
\begin{proof}
Part 1 follows from \cite[Const. 4.23]{Wu25}. Part 2 is \cite[Prop. 4.30 and Const. 4.31]{Wu25}. Part 3 follows from \cite[Prop. 4.32]{Wu25}. Note that the fourth line in (\ref{eq-exp-computation}) follows from the fact that the stabilizer of both ${\ca{S}_{K_0^\alpha}(\sigma_0^\alpha)}$ and ${\ca{S}_{K_{\Phi_0^\alpha},\sigma_0^\alpha}}$ is $\Delta^\circ_{\lcj{K}{g_0^\alpha g_\alpha}}(P_0,G)=\Delta_{\lcj{K}{g_0^\alpha g_\alpha}}(P_0,ZP_\Phi)$ by \cite[Lem. 4.25 (1)]{Wu25}.
\end{proof}
From the proposition above and Proposition \ref{prop-torsor-hz} (4), we see that it suffices to consider the quotient by $\Delta^{g_0^\alpha}:=\Delta_{\lcj{K}{g_0^\alpha g_\alpha}}(P_0,ZP_\Phi)$ of the torsor $\mathscr{E}_{G^c_\zbkp}(\sigma_0^\alpha):=\mathscr{E}_{G_{0,\zbkp},K_{\Phi_0^\alpha}}(\sigma_0^\alpha)\times^{G_{0,\zbkp}}G^c_\zbkp$ over $\cpl{\ca{S}_{K_0^\alpha}(\sigma_0^\alpha)}{\ca{S}_{K_{\Phi_0^\alpha},\sigma_0^\alpha}}$.\par
More precisely,
define 
$$\Delta^{g_0^\alpha}_{G^c}:=\ker({}^{G_0}\ag(P_0)\times^{G_0}G^c\to {}^{G^c}\ag(ZP_\Phi)/(ZP_\Phi(\A)\cap\lcj{K}{g_0^\alpha g_\alpha})).$$
\begin{lem}[cf. Lemma \ref{lem-delta-g}]\label{lem-delta-g0a}
The group $\Delta^{g_0^\alpha}_{G^c}$ is profinite, and its projection to $\Delta^{g_0^\alpha}$ is an isomorphism.
\end{lem}
\begin{proof}In fact, as Lemma \ref{lem-delta-g}, 
there is a commutative diagram
\begin{equation*}
    \begin{tikzcd}
    1\arrow[r]& G^c\arrow[r]\arrow[d,equal]&{{}^{G_0}\ag(P_0)\times^{G_0}G^c}\arrow[r]\arrow[d]&\ag(P_0)\arrow[r]\arrow[d]&1\\
    1\arrow[r]& G^c\arrow[r]&{}^{G^c}\ag(ZP_\Phi)/ZP_\Phi(\A)\cap\lcj{K}{g_0^\alpha g_\alpha}\arrow[r]&\ag(ZP_\Phi)/ZP_\Phi(\A)\cap\lcj{K}{g_0^\alpha g_\alpha}\arrow[r]&1
    \end{tikzcd}
\end{equation*}
where the horizontal arrows are exact.
We conclude by replacing \cite[Lem. 4.15]{Wu25} with \cite[Lem. 4.22]{Wu25}.
\end{proof}
\begin{construction}\label{const-delta-g0a-generic}\upshape
There is a natural action of $\Delta^{g_0^\alpha}_{G^c}$ on $\ca{E}_{G^c,K_{\Phi_0^\alpha}}:=\ca{E}_{G_0,K_{\Phi_0^\alpha}}\times^{G_0}G^c\to\sh_{K_{\Phi_0^\alpha}}$ and $\ca{E}_{G^c}(\sigma_0^\alpha):=\ca{E}_{G_0,K_{\Phi_0^\alpha}}(\sigma_0^\alpha)\times^{G_0}G^c\to\sh_{K_{\Phi_0^\alpha}}(\sigma_0^\alpha)$ that factors through $\Delta^{g_0^\alpha}_{G^c}/K_{\Phi_0^\alpha}$, which is finite by Lemma \ref{lem-delta-g0a}. We repeat the arguments in Construction \ref{const-delta-G-act-generic} with some changes.\par
Let $(\mbf{g},p,\gamma^{-1})\in \Delta^{g_0^\alpha}_{G^c}$. Then, from the definition of $\Delta_{G^c}^{g_0^\alpha}$, there is an element $t\in ZP_\Phi(\bb{Q})_+:=\stb_{ZP_\Phi(\bb{Q})}(D_\Phi^+)$ such that $(t,t,1)\sim(1,1,\gamma^{-1})$ in ${}^{G^c}\ag(ZP_\Phi)$ (see \cite[\S 4.1.4]{Wu25}) and such that $(\mbf{g},p,\gamma^{-1})\sim (\mbf{g}t,pt,1)\in(1,\lcj{K}{g_0^\alpha g_\alpha}Z(\bb{Q})^{\overline{\ }}/Z(\bb{Q})^{\overline{\ }},1)$ in ${}^{G^c}\ag(ZP_\Phi)$. This implies that $\mbf{g}t=1$ and $gt\in ZP_\Phi(\A)\cap\lcj{K}{g_0^\alpha g_\alpha}Z(\bb{Q})^{\overline{\ }}/Z(\bb{Q})^{\overline{\ }}$.\par
By Proposition \ref{prop-can-mod-HZM}, we have an action
\begin{equation*}
\begin{split}
&\ca{E}_{G^c,K_{\Phi_0^\alpha}}\xrightarrow{\cdot \mbf{g}}\ca{E}_{G^c,K_{\Phi_0^\alpha}}\times^{G^c}\mbf{g}^{-1}G^c\mbf{g}\xrightarrow{\cdot g} \ca{E}_{G^c,p^{-1}K_{\Phi_0^\alpha} p}\times^{G^c}(\mbf{g}^{-1}G^c\mbf{g})\xrightarrow{\cdot \gamma^{-1}}\\ 
&\ca{E}_{G^c,\gamma p^{-1}K_{\Phi_0^\alpha} p \gamma^{-1}}\times^{G^c}\gamma G^c \gamma^{-1}\times^{\gamma G^c\gamma^{-1}} (\gamma \mbf{g}^{-1} G^c \mbf{g}\gamma^{-1})=\ca{E}_{G^c,K_{\Phi_0^\alpha}}\times^{G^c}G^c=\ca{E}_{G^c,K_{\Phi_0^\alpha}}.
\end{split}
\end{equation*}

For an element $(\mbf{g},p,\gamma^{-1})\sim (1,1,1)$ in ${}^{G^c}\ag(P_0)$, we check over elements $(x,\mbf{g}_0,p_0)\in P_0(\bb{Q})\bss D_{Q_0}\times G^c(\bb{C})\times P_0(\A)/K_{\Phi_0^\alpha}= \ca{E}_{G^c,K_{\Phi_0^\alpha}}(\bb{C})$: The action
\begin{equation*}
\begin{split}
&\ca{E}_{G^c,K_{\Phi_0^\alpha}}(\bb{C})\xrightarrow{\cdot \mbf{g}}\ca{E}_{G^c,K_{\Phi_0^\alpha}}(\bb{C})\times^{G^c}\mbf{g}^{-1}G^c\mbf{g}\xrightarrow{\cdot g} \ca{E}_{G^c,p^{-1}K_{\Phi_0^\alpha} p}\times^{G^c}(\mbf{g}^{-1}G^c\mbf{g})(\bb{C})\xrightarrow{\cdot \gamma^{-1}}\\ 
&\ca{E}_{G^c,\gamma p^{-1}K_{\Phi_0^\alpha}p \gamma^{-1}}\times^{G^c}\gamma G^c \gamma^{-1}\times^{\gamma G^c\gamma^{-1}} (\gamma \mbf{g}^{-1} G^c \mbf{g}\gamma^{-1})(\bb{C})=\ca{E}_{G^c,K_{\Phi_0^\alpha}}\times^{G^c}G^c(\bb{C})=\ca{E}_{G^c,K_{\Phi_0^\alpha}}(\bb{C})
\end{split}
\end{equation*}
is computed as
\begin{equation*}
\begin{split}
& (x,\mbf{g}_0,p_0) \xrightarrow{\cdot \mbf{g}} (x,\mbf{g}_0\mbf{g},p_0)\xrightarrow{\cdot p} (x,\mbf{g}_0\mbf{g},p_0p)\xrightarrow{\cdot \gamma^{-1}}\\ 
& (\gamma x\gamma^{-1},\gamma \mbf{g}_0\mbf{g}\gamma^{-1},\gamma p_0p\gamma^{-1})= (\gamma x,\gamma\mbf{g}_0,\gamma p_0)=(x,\mbf{g}_0,p_0).
\end{split}
\end{equation*}
The last line is due to the facts that $\gamma$ is an element in $P_0(\bb{Q})_+$ and is equal to $p$, and that the image of $\gamma$ in $G^c(\bb{Q})$ is equal to $\mbf{g}$.
This map is an identity map, so we have obtained a well-defined action of $\Delta^{g_0^\alpha}_{G^c}$ on $\ca{E}_{G^c,K_{\Phi_0^\alpha}}.$
The functoriality with respect to the affine torus embeddings and (\ref{eq-ext-W-boundary-e}) then induce an action of $(\mbf{g},p,\gamma^{-1})$ on $\ca{E}_{G^c}(\sigma_0^\alpha)$ that is trivial when $(\mbf{g},p,\gamma^{-1})\sim (1,1,1)$ in ${}^{G^c}\ag(P_0)$.\par
By Lemma \ref{lem-check-group-quotient}, the quotient $\ca{E}_{G^c,K_{\Phi_0^\alpha}}/\Delta^{g_0^\alpha}_{G^c}$ is an open and closed subscheme of the $G^c$-torsor $\ca{E}_{ZP_\Phi^c,K_\Phi}\times^{ZP_\Phi^c}G^c$.\hfill$\square$
\end{construction}
Clearly, there is a natural morphism from $\Delta^{g_0^\alpha}$ (resp. $\Delta^{g_0^\alpha}_{G^c}$) to $\Delta^\alpha$ (resp. $\Delta^\alpha_{G^c}$), which defines an action of $\Delta_{G^c}^{g_0^\alpha}$ on $\mathscr{E}^\can_{G^c_\zbkp,K_0^\alpha}$ by pulling back the action defined in Proposition \ref{prop-action-delta}.
\begin{prop}\label{prop-delta-g0a-action}
There is an action of $\Delta^{g_0^\alpha}_{G^c}$ on $\mathscr{E}_{G_\zbkp^c}(\sigma_0^\alpha)$ covering the action of $\Delta^{g_0^\alpha}$ on $\ca{S}_{K_{\Phi_0^\alpha}}(\sigma_0^\alpha)$. The pullback of $\mathscr{E}_{G_\zbkp^c}(\sigma_0^\alpha)$ and this $\Delta^{g_0^\alpha}_{G^c}$-action to $\cpl{\ca{S}_{K_0^\alpha}(\sigma_0^\alpha)}{\ca{S}_{K_{\Phi_0^\alpha},\sigma_0^\alpha}}$ coincide with the pullback of $\mathscr{E}_{G^c_\zbkp,K_0^\alpha}^\can$ and the $\Delta_{G^c}^{g_0^\alpha}$-action on it defined in Proposition \ref{prop-action-delta} to $\cpl{\ca{S}_{K_0^\alpha}(\sigma_0^\alpha)}{\ca{S}_{K_{\Phi_0^\alpha},\sigma_0^\alpha}}$.
\end{prop}
\begin{proof}
By Construction \ref{const-delta-g0a-generic}, there is an action of $\Delta^{g_0^\alpha}_{G^c}$ on $\ca{E}_{G^c}(\sigma_0^\alpha)$ over $\sh_{K_{\Phi_0^\alpha}}(\sigma_0^\alpha)$. We shall check that this action is compatible with the action of $\Delta^{g_0^\alpha}_{G^c}$ on $\mathscr{E}_{G^c_\zbkp}(\sigma_0^\alpha)|_{\cpl{\ca{S}_{K_0^\alpha}(\sigma_0^\alpha)}{\ca{S}_{K_{\Phi_0^\alpha},\sigma_0^\alpha}}}$ defined by pulling back the action on $\mathscr{E}^\can_{G^c_\zbkp,K_0^\alpha}$ via Proposition \ref{prop-torsor-hz} (4). If we have this compatibility, there is a desired action on the entire $\mathscr{E}_{G^c_\zbkp}(\sigma_0^\alpha)$ by gluing morphisms.\par
It is enough to check the compatibility over the open analytic subspace $V(\Phi_0^\alpha)$ as in \S\ref{subsec-f-u-f} (13). Let us now recall some details about this $V(\Phi_0^\alpha)$ (see \cite[6.10]{Pin89} and \cite[2.1.13]{Mad19}). In fact, there is an open subspace 
$U(\Phi_0^\alpha):=P_{\Phi_0^\alpha}(\bb{Q})_+\bss X_0^+\times P_{\Phi_0^\alpha}(\A)/K_{\Phi_0^\alpha}\to P_{\Phi_0^\alpha}(\bb{Q})_+\bss U_{\Phi_0^\alpha}(\bb{C})\tau(X_0^+)\times P_{\Phi_0^\alpha}(\A)/K_{\Phi_0^\alpha}$.
There is also a morphism, 
$$U(\Phi_0^\alpha)\to G_0(\bb{Q})_+\bss X^+_0\times G_0(\A)/g_0^\alpha K_0^\alpha (g_0^\alpha)^{-1}.$$
which is locally an isomorphism. The open subspace $V(\Phi_0^\alpha)\sbst U(\Phi_0^\alpha)$ is taken so that the displayed morphism is an isomorphism on it. Hence, it suffices to check explicitly on $U(\Phi_0^\alpha)$.\par
We compute the action of $(\mbf{g},p,\gamma^{-1})\in\Delta_{G^c}^{g_0^\alpha}$ on $\ca{E}_{G^c,K_{\Phi_0^\alpha}}|_{U(\Phi_0^\alpha)}$ according to Construction \ref{const-delta-G-act-generic} and Construction \ref{const-delta-g0a-generic}, and they are both given by 
$$[(x,\mbf{g}_0,p_0)]\mapsto (x,\mbf{g}_0\mbf{g},p_0)\mapsto (x,\mbf{g}_0\mbf{g},p_0p)\mapsto (\gamma x\gamma^{-1},\gamma \mbf{g}_0\mbf{g}\gamma^{-1},\gamma p_0p\gamma^{-1}).$$
This finishes the verification of the claim on the compatibility, so the proposition follows.
\end{proof}
Now we need to see that 
\begin{prop}\label{prop-boundary-torsor-quo}
The quotient $\mathscr{E}_{G^c_\zbkp}(\sigma_0^\alpha)/\Delta^{g_0^\alpha}_{G^c}$ is a $G_\zbkp^c$-torsor.
\end{prop}
\begin{proof}
This is a special case of Lemma \ref{lem-quotient-tors-general}. We formulate a general lemma below to lighten the heavy notation.\par 
Let us explain why Lemma \ref{lem-quotient-tors-general} is applicable. 
By (\ref{eq-hz-int}), we let $\mathscr{E}_{G_{\zbkp}^c,K_{\Phi_0^\alpha}}(\sigma_0^\alpha)$ be the $\mathscr{E}_0(\sigma_0)$ below and $(\mbf{p}_{1,*}\mathscr{E}_{G_{0,\zbkp},K_{\Phi_0^\alpha}})^{\mbf{E}_{K_{\Phi_0^\alpha}}}$ be the $\overline{\mathscr{E}}_0$ below. 
The group $\Delta^{g_0^\alpha}_{G^c}$ acts through the quotient $\Delta^{g_0^\alpha}_{G^c}/K_{\Phi_0^\alpha}$, which is finite by Lemma \ref{lem-delta-g0a}.
The other symbols are clear in our context. Condition (1) is verified in \cite[Lem. 4.25]{Wu25}.\par
Let us now verify the last sentence in Condition (2). Recall that  $\Delta^{g_0^\alpha}_{G^c}$ is isomorphic to $\Delta^{g_0^\alpha}_{G^c}\sbst \ag(P_0)$ by Lemma \ref{lem-delta-g0a}. A subgroup of $\ag(P_0)$ acts on $\overline{\sh}_{K_{\Phi_0^\alpha}}$ trivially if and only if it is contained in the kernel of $\ag(P_0)\to \ag(\overline{P}_0)/\overline{K}_{\Phi_0^\alpha}$ if and only if it is contained in 
$$\ca{K}:=(Z_{0,ac}(\bb{Q})^{\overline{\ }}U_0(\A)K_{\Phi_0^\alpha}/Z_{0,ac}(\bb{Q})^{\overline{\ }})\times\{1\}$$ 
(cf. \cite[Lem. 4.2]{Wu25}). Here, we denote by $U_0$ the center of unipotent radical of $P_0$, by $\overline{P}_0$ the quotient $P_0/U_0$ and by $Z_0$ the center of $G_0$.
If $(p,1)\in \Delta^{g_0^\alpha}\cap \ca{K}$, $p$ will lie in $(\lcj{K}{g_0^\alpha g_\alpha}\cap ZP_\Phi(\A)/W_\Phi(\A))$ upon adjusting $Z_G(\bb{Q})$, so we have that $p\in Z_G(\bb{Q})(\lcj{K}{g_0^\alpha g_\alpha}\cap ZP_\Phi(\A))$. 
Write 
$$p=z\cdot k=z_0\cdot u\cdot k_0$$ 
for $z\in Z_G(\bb{Q})$, $z_0\in Z_{0,ac}(\bb{Q})$, $u\in U_0(\A)$, $k_0\in K_{\Phi_0^\alpha}$ and $k\in \lcj{K}{g_0^\alpha g_\alpha}\cap ZP_\Phi(\A)$. Then $z\cdot z_0^{-1}\in Z_G(\bb{Q})\cap (\lcj{K}{g_0^\alpha g_\alpha}\cap ZP_\Phi(\A)/W_\Phi(\A))$. Since $\lcj{K}{g_0^\alpha g_\alpha}\cap ZP_\Phi(\A)/W_\Phi(\A)$ is neat, $z\cdot z_0^{-1}\in Z_{G,ac}(\bb{Q})$.
Then this implies that $z\in Z_{G,ac}(\bb{Q})$.\par
Now if $(\mbf{p},p,1)\in \Delta^{g_0^\alpha}_{G^c}$ projects to $(p,1)\in \Delta^{g_0^\alpha}\cap \ca{K}$, then adjusting $z$ trivializes $\mbf{p}$, but $z$ is trivial in $G^c$. So $\mbf{p}$ is trivial. This verifies (2) in Lemma \ref{lem-quotient-tors-general}.
\end{proof}
\begin{lem}\label{lem-quotient-tors-general}
Let $X_0$, $Y_0$, $X$ and $Y$ be quasi-projective schemes over $\zbkp$. Let $\mbf{E}_0$ and $\mbf{E}$ be split tori with an isogeny $\mbf{E}_0\to \mbf{E}$. Suppose that $\mbf{p}_0:X_0\to Y_0$ and $\mbf{p}:X\to Y$ are torsors under $\mbf{E}_0$ and $\mbf{E}$, respectively. 
The setup is:
\begin{enumerate}
    \item Suppose that there is an $(\mbf{E}_0\to \mbf{E})$-equivariant morphism between torsors
\begin{equation}\label{eq-mor-torsors}
\begin{tikzcd}
    X_0\arrow[r,"\mbf{p}_0"]\arrow[d,"/\Delta"]&Y_0\arrow[d,"/\Delta"]\\
    X\arrow[r,"\mbf{p}"]& Y.
    \end{tikzcd}
\end{equation}
We assume that:
\begin{enumerate}
    \item There is an equivariant action of a finite group $\Delta$ on $X_0\to Y_0$ such that the action of $\Delta$ factors through finite groups $\Delta_1$ and $\Delta_2$ on $X_0$ and $Y_0$, respectively. This action commutes with $\mbf{E}_0$-action on $X_0$ as a torsor. 
    \item The action of $\Delta_1$ and $\Delta_2$ on $X_0$ and $Y_0$ are free.
    \item The second torsor $X\to Y$ in the diagram (\ref{eq-mor-torsors}) is obtained by the quotient of $X_0\to Y_0$ by $\Delta$.
\end{enumerate}
Let $\sigma$ be a rational polyhedral cone of the cocharacter group $\mbf{X}_*\mbf{E}_\bb{R}\iso (\mbf{X}_*\mbf{E}_0)_\bb{R}$, and let $\sigma_0$ be its pullback to $(\mbf{X}_*\mbf{E}_{0})_\bb{R}$. Hence, there is an $\mbf{E}(\sigma)$-torsor $\mbf{p}(\sigma):X(\sigma)\to Y$ given by the affine torus embedding with respect to $\sigma$. Similarly, there is a $\mbf{E}_0(\sigma_0)$-torsor $\mbf{p}_0(\sigma_0):X_0(\sigma_0)\to Y_0$.\par
\item Let $H_\zbkp$ be an affine smooth reductive group over $\zbkp$, and let $\overline{\mathscr{E}}_0$ be an $H_\zbkp$-torsor over $Y_0$. Denote by $\mathscr{E}_0(\sigma_0)$ the pullback of $\overline{\mathscr{E}}_0$ to $X_0(\sigma_0)$ under $\mbf{p}_0(\sigma_0)$. Suppose that $\overline{\mathscr{E}}_0\to Y_0$ has an equivariant $\Delta_2$-action and $H_\zbkp$-torsor structure commutes with this action. 
\end{enumerate}
Under all the assumptions above, we have an isomorphism $$\mathscr{E}_0(\sigma_0)/\Delta\iso \mbf{p}(\sigma)^*(\overline{\mathscr{E}}_0/\Delta). $$
Moreover, $\mathscr{E}_0(\sigma_0)/\Delta$ is an $H_\zbkp$-torsor.
\end{lem}
\begin{proof}
We have 
\begin{equation*}
    \begin{split}
    &\ca{O}_{\mathscr{E}_0(\sigma_0)}^{\Delta}\iso (\ca{O}_{X_0}\otimes_{\ca{O}_{Y_0}}\ca{O}_{\overline{\mathscr{E}}_0})^\Delta \\
    &\iso(\ca{O}_{X_0}\otimes_{\ca{O}_{Y_0}}(\ca{O}_{Y_0}\otimes_{\ca{O}_Y}\ca{O}_{\overline{\mathscr{E}}_0/\Delta}))^\Delta\\
    &\iso (\ca{O}_{X_0}\otimes_{\ca{O}_Y}\ca{O}_{\overline{\mathscr{E}}_0/\Delta})^\Delta\\
    &\iso \ca{O}_X\otimes_{\ca{O}_{Y}}\ca{O}_{\overline{\mathscr{E}}_0/\Delta}.
    \end{split}
\end{equation*}
The first-to-second-line isomorphism holds and $\overline{\mathscr{E}}_0/\Delta$ is again an $H_\zbkp$-torsor because the action of the finite group $\Delta_2$ on $Y_0$ is free.
\end{proof}
\begin{proofof}[Theorem \ref{thm-loc-free-quo}]
Now, one can combine Proposition \ref{prop-boundary-torsor-quo} and Proposition \ref{prop-delta-g0a-action} to conclude that the quotient of $\mathscr{E}^\can_{G_\zbkp^c,K_0^\alpha}$ by $\Delta_{G^c}^{g_0^\alpha}$ on each $\cpl{\ca{S}_{K_0^\alpha}(\sigma_0^\alpha)}{\ca{S}_{K_{\Phi_0^\alpha},\sigma_0^\alpha}}$ is a $G_\zbkp^c$-torsor. By the fourth line to the first line of (\ref{eq-exp-computation}), the quotient by $\Delta_{G^c}^\alpha$ of $\mathscr{E}^\can_{G_\zbkp^c,K_0^\alpha}$ is a $G^c_{\zbkp}$-torsor at each $\cpl{\ca{S}^\Sigma_K}{ \ca{Z}_{[ZP^b({\Phi},\sigma)],K}}$.\par
Choose a cover of $\ca{S}^\Sigma_K$ by finitely many affine open subschemes $U_i$.
We then form an fpqc covering of $\ca{S}_K^{\Sigma}$ by taking the spectrum of the completion of $U_i$ at every $\ca{Z}_{[ZP^b(\Phi,\sigma)],K}$. Since $G^c_\zbkp$ is smooth, it suffices to check $\mathscr{E}^\can_{G^c_\zbkp,K_0^\alpha}$ being a torsor over this fpqc cover. So we are done by the last paragraph.
\end{proofof}
\subsubsection{Proof of main theorem}
\begin{construction}\upshape
Let us now construct a principal $G^c_{\zbkp}$-bundle on $\ca{S}^{\Sigma_2}_{K_2}$. (Recall that $G_{2,\zbkp}\hookrightarrow G_{\zbkp}$ with derived groups an isomorphism. We have $G^c_{2,\zbkp}\hookrightarrow G^c_{\zbkp}$.)\par
Recall that there is an isomorphism $s:\ca{S}_K^\Sigma(G,X_a)_{\ca{O}^{ur}}\iso \ca{S}^\Sigma_K(G,X_b)_{\ca{O}^{ur}}$ and there is an open and closed embedding $\pi^a:\ca{S}_{K_2}^{\Sigma_2}\hookrightarrow \ca{S}^\Sigma_K(G,X_a)$. We pullback the $G^c_\zbkp$-bundle $\mathscr{E}^\can_{G^c_\zbkp,K,\ca{O}^{ur}}$ over $\ca{S}^\Sigma_K(G,X_b)_{\ca{O}^{ur}}$ to $\ca{S}^{\Sigma_2}_{K_2}$ via $s\circ\pi^a_{\ca{O}^{ur}}$. We denote this pullback by $\mathscr{E}^\can_{G^c_\zbkp,K_2,\ca{O}^{ur}}$. This is a $G^c_\zbkp$-bundle, and we claim the bundle obtained by reduction of structure group and Galois descent is what we need.\hfill$\square$
\end{construction}

\begin{proofof}[Theorem \ref{thm-extend-main-theorem}]
By Lemma \ref{lem-generic-compatible}, the generic fiber of $\mathscr{E}^\can_{G^c_\zbkp}$ is $\ca{E}^\can_{G^c,K}(G,X_b)$. Over $E^{\prime,p}$, the pullback of $\ca{E}^\can_{G^c,K}(G,X_b)$ via $s$ is $\ca{E}^\can_{G^c,K}(G,X_a)$ by (\ref{eq-deligne-ind-torsor}) (from which we see that $\ca{E}_{G^c,K}(G,X_a)_{E^{\prime,p}}\iso \ca{E}_{G^c,K}(G,X_b)_{E^{\prime,p}}$) and by Deligne's existence theorem (which uniquely characterizes the canonical extensions).\par
By Theorem \ref{thm-lovering-ab} (2) and by the proof of Lemma \ref{lem-generic-compatible}, the bundle $\mathscr{E}^\can_{G^c_\zbkp,K}$ extends $\mathscr{E}_{G^c,K}(G,X_b)$ over $\ca{S}_{K}(G,X_b)$. After pulling back along $s$, $s^*\mathscr{E}_{G^c,K}(G,X_b)_{\ca{O}^{ur}}$ descends to $\mathscr{E}_{G^c,K}(G,X_a)$ by \cite[Lem. 4.4.9]{Lov17}. Now we have shown the claims when $(G_2,X_2)=(G,X_a)$.\par
Then the descent data for $\ca{E}_{G^c,K,E^{\prime,p}}^\can$ from $E^{\prime,p}$ to $E'$, and for $\mathscr{E}_{G^c_{\zbkp},K}(G,X_a)_{\ca{O}^{ur}}$ from $\ca{O}^{ur}$ to $\ca{O}$ determines a descent datum for $\mathscr{E}_{G^c_{\zbkp},K,\ca{O}^{ur}}^\can$. 
Consequently, we obtain a $G^c_\zbkp$-bundle $\mathscr{E}^\can_{G^c_\zbkp,K}(G,X_a)$ over $\ca{S}^\Sigma_K(G,X_a)_{\ca{O}}$ extending $\mathscr{E}_{G^c_{\zbkp},K}(G,X_a)$ and $\ca{E}_{G^c,K,E'}^\can(G,X_a)$.\par
We pull $\mathscr{E}^\can_{G^c_\zbkp,K}(G,X_a)$ back to $\ca{S}^{\Sigma_2}_{K_2}$ along $\pi^a$. We denote this pullback by $\mathscr{E}^\can_{G_{\zbkp}^c,K_2}$. 
We check that: 
\begin{enumerate}
    \item  $\mathscr{E}^\can_{G^c_\zbkp,K_2}$ has a reduction of structure group to a $G_{2,\zbkp}^c$-torsor $\mathscr{E}_{G^c_{2,\zbkp},K_2}^{\can}$ over $\ca{S}_{K_2}^{\Sigma_2}$,
    \item $\mathscr{E}_{G^c_{\zbkp},K_2}^\can$ extends $\ca{E}_{G^c_2,K_2,E'}^{\can}\times^{G^c_{2,\zbkp}}G^c_\zbkp$, and
    \item $\ca{E}_{G^c_{\zbkp},K_2}^{\can}$ extends $\mathscr{E}_{G^c_{2,\zbkp},K_2,\ca{O}}\times^{G^c_{2,\zbkp}}G^c_\zbkp$.
\end{enumerate} 
The claim (2) follows from Proposition \ref{prop-functoriality-generic}, and the claim (3) follows from Theorem \ref{thm-lovering-ab}. For (1), it suffices to see it has a reduction of structure group over a locus of codimension at least $2$. Then (1) follows from (2) and (3). From this, it follows that $\mathscr{E}_{G^c_{2,\zbkp},K_2}^\can$ extends both $\ca{E}_{G^c_2,K_2,E'}^{\can}$ and $\mathscr{E}_{G^c_{2,\zbkp},K_2,\ca{O}}$.\par
Now we consider the descent datum from $\ca{O}_{E',(v')}$ to $\ca{O}_2$. Again, the $G^c_{2,\zbkp}$-bundle $\mathscr{E}^\can_{G^c_{2,\zbkp},K_2}$ descends to $\mathscr{E}^\can_{G^c_{2,\zbkp},K_2,\ca{O}_2}$ over $\ca{S}^{\Sigma_2}_{K_2,\ca{O}_2}$ since one can check this on the union of $\ca{S}_{K_2}$ and $\sh^{\Sigma_2}_{K_2,E'}$. Again, all properties in the second paragraph of Theorem \ref{thm-extend-main-theorem} also hold over $\ca{O}_2$.
\end{proofof}
The following corollary is a consequence of Theorem \ref{thm-lovering-ab} (2) and Proposition \ref{prop-functoriality-generic}.
\begin{cor}
Let $f:(G_1,X_1,G_{1,\zbkp})\to (G_2,X_2,G_{2,\zbkp})$ be a morphism between abelian-type Shimura data with a map between smooth reductive models. Let $K_{1,p}=G_{1,\zbkp}(\bb{Z}_p)$ and $K_{2,p}:=G_{2,\zbkp}(\bb{Z}_p)$. 
Suppose that there is an inclusion between neat prime-to-$p$ levels $K_1^p\sbst K_2^p$. For a fixed pair of compatible admissible projective cone decompositions $(\Sigma_1,\Sigma_2)$ for $(G_1,X_1,K_1)$ and $(G_2,X_2,K_2)$, there is a pair of refinements $(\Sigma_1',\Sigma_2')$ such that $\mathscr{E}^\can_{G_{1,\zbkp}^c,K_1}\to \ca{S}^{\Sigma_1'}_{K_1}$ and $\mathscr{E}^\can_{G_{2,\zbkp}^c,K_2}\to \ca{S}^{\Sigma_2'}_{K_2}$ are defined.\par 
Then the pullback of $\mathscr{E}^\can_{G_{2,\zbkp}^c,K_2}$ under the map $f:\ca{S}_{K_1}^{\Sigma_1'}\to \ca{S}_{K_2}^{\Sigma_2'}$ is canonically isomorphic to $\mathscr{E}^\can_{G^c_{1,\zbkp},K_1}\times^{G^c_{1,\zbkp}}G^c_{2,\zbkp}$.
\end{cor}
\subsection{Other principal bundles and properties}\label{subsec-auto-vb}
\subsubsection{}
Fix a Shimura datum $(G,X)$. Let $\{\mu\}$ be the associated conjugacy class of Hodge cocharacters; this conjugacy class is defined over $E:=E(G,X)$. Let $H$ be a finite Galois extension of $E$ such that, for one place $\wdtd{v}$ of $H$ unramified over a place $v|p$ of $E$, there is a cocharacter $\mu': \bb{G}_{m,\ca{O}_{H,{(\wdtd{v})}}}\to G_{\ca{O}_{H,(\wdtd{v})}}$ whose base change to $\bb{C}$ is in $\{\mu\}$. There is such a cocharacter by \cite[Prop. 3.2.3]{Lov17}.\par
Let $P_{\mu'}$ be the parabolic subgroup of $G_{\ca{O}_{H,(\wdtd{v})}}$ whose Lie algebra is spanned by all non-negative weight spaces in $\lie G_{\ca{O}_{H,(\wdtd{v})}}$ under the adjoint action of $\mu^{\prime,-1}$; the group of $\bb{C}$-points of $P_{\mu'}$ is defined by $$P_{\mu'}(\bb{C}):=\{g\in G(\bb{C})|\lim_{t\to 0}\mu'(t)^{-1}g\mu'(t)\text{ exists}\}.$$ 
Let $P^c_{\mu'}$ be the quotient $P_{\mu_x}/Z_{ac}(G_{\zbkp})_{\ca{O}_{H,(\wdtd{v})}}$, and let $M_{\mu'}$ (resp. $M^c_{\mu'}$) be the Levi quotient of $P_{\mu'}$ (resp. $P^c_{\mu'}$). Then $M^c_{\mu'}\iso M_{\mu'}/Z_{ac}(G_{\zbkp})_{\ca{O}_{H,(\wdtd{v})}}$. 
Denote by $\mu^{\prime,c}$ the composition of $\mu'$ with $G_{\ca{O}_{H,(\wdtd{v})}}\to G^c_{\ca{O}_{H,(\wdtd{v})}}$.
Then $P^c_{\mu'}=P_{\mu^{\prime,c}}$.\par
Denote $Gr_{\mu^{\prime,c}}:=G_{\ca{O}_{H,(\wdtd{v})}}/P_{\mu'}\iso G^c_{\ca{O}_{H,(\wdtd{v})}}/P_{\mu^{\prime,c}}$. Then $Gr_{\mu'}$ canonically descends to a proper smooth scheme $\ca{GR}_{\mu^c}$ over $\ca{O}_{E,(v)}$ by \cite[Prop. 3.2.6]{Lov17}.\par 
For any {\'e}tale algebra $R$ over $\ca{O}_{E,(v)}$, $\ca{GR}_{\mu^c}(R)=$ 
$$\{P/\spec R\text{\ parabolic\ subgroup\ of\ }G^c_{R}|\text{\ }P\ \text{{\'e}tale locally conjugates to }P^c_{\mu'}\}.$$
It is independent of the choice of $H$, $\mu'$ and $\wdtd{v}$ above. 
\subsubsection{Reduction of structure group}
We come back to the setup in \S\ref{subsec-tor-ab-review}.
Let $L$ be a finite extension of $E_2:=E(G_2,X_2)$. Let $w|v_2$ be a place of $L$ unramified over $v_2$. Suppose that there is a cocharacter over $\ca{O}_{L,(w)}$
$$\mu'_2: \bb{G}_{m,\ca{O}_{L,{({w})}}}\to G_{\ca{O}_{L,({w})}}$$
whose base change to $\bb{C}$ is in the conjugacy class of Hodge cocharacters $\{\mu_2\}$ associated with $X_2$. Let $\mu_2^{\prime,c}$ denote the composition of $\mu'_2$ with $G_{2,\ca{O}_{L,(w)}}\to G^c_{2,\ca{O}_{L,(w)}}$.\par
\begin{prop}\label{prop-reduction-para-mu}
With the conventions above, the canonical extension $\mathscr{E}^\can_{G^c_{2,\zbkp},K_2}$ admits a canonical reduction of structure group to a $P_{\mu'_2}^c$-torsor.  
\end{prop}
\begin{proof}
First, we show this over $\ca{S}_{K_2}$. By \cite[Thm. 4.8.1]{Lov17}, there is a diagram
$$\ca{S}_{K_2}\xleftarrow{\ \ \pi_1\ \ } \mathscr{E}_{G^c_{2,\zbkp},K_2}\xrightarrow{\ \ \pi_2\ \ } \ca{GR}_{\mu^{c}_2}$$
such that the right arrow is $G^c_{2,\zbkp}$-equivariant. We take the fiber product
\begin{equation}\label{eq-def-red-group}
    \begin{tikzcd}
        \mathscr{E}_{P^c_{\mu'_2},K_2}\arrow[rr,dashed]\arrow[d,dashed]&& \mathscr{E}_{G^c_{2,\zbkp},K_2}\arrow[d,"{(\pi_1,\pi_2)}"]\\
        \ca{S}_{K_2}\otimes_{\ca{O}_{(v_2)}}\ca{O}_{L,(w)}\arrow[rr,"\mrm{id}\times x"]&&\ca{S}_{K_2}\times_{\ca{O}_{(v_2)}} \ca{GR}_{\mu_2^c}.
    \end{tikzcd}
\end{equation}
The map $x:\spec \ca{O}_{L,(w)}\to \ca{GR}_{\mu_2^c}$ is the point determined by $P_{\mu^{\prime,c}_2}$. This is indeed a $P^c_{\mu'_2}$-torsor: For any {\'e}tale cover $U$ of $\ca{S}_{K_2}$ that trivializes $\mathscr{E}_{G_{2,\zbkp}^c,K_2}$, the diagram (\ref{eq-def-red-group}) is isomorphic to
\begin{equation*}
    \begin{tikzcd}
     U\times_{\ca{O}_{(v_2)}} P^c_{\mu^{\prime,c}_2}\arrow[rr]\arrow[d]&& U\times_{\ca{O}_{(v_2)}}G^c_{2,\ca{O}_{L,(w)}} \arrow[d,"{(\pi_1,\pi_2)}"]\\
    U\otimes_{\ca{O}_{(v_2)}}\ca{O}_{L,(w)}\arrow[rr,"\mrm{id}\times x"]&&U\times_{\ca{O}_{(v_2)}} \ca{GR}_{\mu_2^c}.
    \end{tikzcd}
\end{equation*}
Next, we show this over $\sh^{\Sigma_2}_{K_2}$. By \cite[Lem. 4.4.2]{HZ94}, there is a diagram
$$\sh^{\Sigma_2}_{K_2}\xleftarrow{\ \ \xi_1\ \ }\ca{E}^\can_{G^c_2,K_2}\xrightarrow{\ \ \xi_2\ \ }\ca{GR}_{\mu_2^c,E_2}.$$
Again, there is a reduction of $\ca{E}^\can_{G_2^c,K_2}$ by pulling back through the point $\wdtd{x}:\spec L\to \ca{GR}_{\mu^c_2,E_2}$ determined by $P_{\mu^{\prime,c}_2, E_2}$, i.e., the fiber product
\begin{equation}\label{eq-def-red-group-gen}
    \begin{tikzcd}
        \ca{E}^\can_{P^c_{\mu'_2},K_2}\arrow[rr]\arrow[d]&& \ca{E}_{G^c_{2},K_2}\arrow[d,"{(\xi_1,\xi_2)}"]\\
        \sh_{K_2,L}^{\Sigma_2}\arrow[rr,"\mrm{id}\times \wdtd{x}"]&&\sh_{K_2}\times_{E_2} \ca{GR}_{\mu_2^c,E_2}.
    \end{tikzcd}
\end{equation}
Since (\ref{eq-def-red-group}) and (\ref{eq-def-red-group-gen}) are compatible on $\sh_{K_2,L}$, we see that $\mathscr{E}^\can_{G^c_{2,\zbkp},K_2,\ca{O}_{L,(w)}}$ admits a reduction of structure group on $\sh^{\Sigma_2}_{K_2,L}\cup \ca{S}_{K_2,\ca{O}_{L,(w)}}$. Then there is a reduction of structure group to a $P^c_{\mu_2'}$-torsor on $\ca{S}^{\Sigma_2}_{K_2,\ca{O}_{L,(w)}}$ since $\sh^{\Sigma_2}_{K_2,L}\cup \ca{S}_{K_2,\ca{O}_{L,(w)}}$ has codimension at least $2$.
\end{proof}
\begin{definition}
    We denote the canonical reduction of structure group by $\mathscr{E}_{P^c_{\mu_2'},K_2}^\can$. We call it a principal $P^c_{\mu_2'}$-bundle. Denote $\mathscr{E}_{M^c_{\mu_2'},K_2}^\can:=\mathscr{E}_{P^c_{\mu_2'},K_2}^\can\times^{P^c_{\mu_2'}}M^c_{\mu_2'}$. We call it a principal $M^c_{\mu_2'}$-bundle.
\end{definition}
\begin{cor}\label{cor-gr-extend}
The maps $\pi_2$ and $\xi_2$ above uniquely extend to a map $\pi_2^\can: \mathscr{E}^\can_{G_{2,\zbkp}^c,K_2}\to \ca{GR}_{\mu_2^c}$.
\end{cor}
\begin{proof}
We work over an {\'e}tale base change to $\ca{O}_{L,(w)}$. Let $R$ be an {\'e}tale $\ca{O}_{L,(w)}$-algebra. For any $W\in \rep_{R}(G^c_{2,\zbkp})$, there is a filtration defined by $\bb{G}_{m,R}\xrightarrow{\mu_{2,R}^{\prime,c}} G^c_{2,R}\to \mrm{GL}(W)$ on $\ul{W}:=\mathscr{E}^{\can}_{G^c_{2,\zbkp},K_2,\ca{O}_{L,(w)}}\times^{G^c_{2,\ca{O}_{L,(w)}}}W=\mathscr{E}_{P^c_{\mu_2'},K_2}\times^{P^c_{\mu_2'}} W$. The last equality follows from Proposition \ref{prop-reduction-para-mu}, and the filtration on $W$ induces a filtration on $\ul{W}$ because the structure group $P^c_{\mu'_2}$ stabilizes the filtration on $W$.\par 
It is clear that its restrictions to $\sh^{\Sigma_2}_{K_2,L}$ and $\ca{S}_{K_2,\ca{O}_{L,(w)}}$ are compatible with the filtrations defined by $\xi_2$ and $\pi_2$ above. This is equivalent to having a map $\pi_2^\can: \mathscr{E}^\can_{G_{2,\zbkp}^c,K_2}\to \ca{GR}_{\mu_2^c}$ by \cite[Lem. 3.3.1]{Lov17}.
\end{proof}

\subsubsection{Subcanonical extensions}
\begin{definition}\label{def-L-bundle}
Let $R$ be a $\zbkp$-algebra. Let $W\in \rep_R(G^c_{2,\zbkp})$. We call $\autoshdr{W}_{K_2}^\can:=\mathscr{E}_{G_{2,\zbkp}^c,K_2}^\can\times^{G^c_{2,\zbkp}} W$ the \textbf{canonical extension} of $\autoshdr{W}_{K_2}:=\mathscr{E}_{G_{2,\zbkp}^c,K_2}\times^{G_{2,\zbkp}^c}W$. We call $\autoshdr{W}_{K_2}^{\mrm{sub}}:=\autoshdr{W}_{K_2}^\can(-D)=\autoshdr{W}_{K_2}\otimes \ca{O}_{\ca{S}^{\Sigma_2}_{K_2}}(-D)$ the \textbf{subcanonical extension} of $\autoshdr{W}_{K_2}$, where $D=(\ca{S}^{\Sigma_2}_{K_2}\bss \ca{S}_{K_2})_{\red}$.\par
Fix $\ca{O}_{L,(w)}$ and $\mu'_2$ as above. Let $R$ be an $\ca{O}_{L,(w)}$-algebra. Let $W\in \rep_R(P^c_{\mu_2'})$. We can also define $\ul{W}_{K_2}^\can:= \mathscr{E}_{P_{\mu_2'}^c,K_2}^\can\times^{P_{\mu_2^c}^c}W$ and $\ul{W}_{K_2}^{\mrm{sub}}:=\ul{W}_{K_2}^\can(-D)$. When $W\in \rep(M^c_{\mu_2'})$, we lift $W$ to a representation of $P^c_{\mu_2'}$. So the corresponding $\ul{W}_{K_2}^\can$ and $\ul{W}^{\mrm{sub}}_{K_2}$ are also defined. 
\end{definition}
%\newpage
\bibliography{references}
\end{document}